\documentclass{ip-journal}

\usepackage{mathrsfs}

\usepackage{bookmark}

\usepackage{stmaryrd} 

\usepackage{amsfonts}
\usepackage{amsthm}
\usepackage{amssymb}

\usepackage{graphicx}
\usepackage{amsmath}
\usepackage{txfonts}

\usepackage{tabu}

\usepackage[outline]{contour}

\newcommand{\dbv}[1]{\left\| #1 \right\|}

\usepackage{fontenc}

\usepackage{tikz-cd}

\usepackage{tikz}

\usepackage{soul}

\usepackage{hyperref}

\usepackage{verbatim}
\usepackage{empheq}
\usepackage[most]{tcolorbox}

\usepackage{bm}

\newcommand{\ben}{\begin{equation*}}
\newcommand{\een}{\end{equation*}}

\newcommand{\be}{\begin{equation}}
\newcommand{\ee}{\end{equation}}

\newcommand{\rd}[1]{\left(#1\right)}
\newcommand{\sq}[1]{\left[#1\right]}
\newcommand{\cl}[1]{\left\{#1\right\}}
\newcommand{\fl}[1]{\lfloor #1 \rfloor}
\newcommand{\ag}[1]{\left\langle#1\right\rangle}
\newcommand{\vb}[1]{\left| #1 \right|}

\newcommand{\ov}[1]{\frac{1}{#1}}

\newcommand{\pd}{\partial}

\newcommand{\Endo}{\text{End}}

\newcommand{\atp}[2]{\left. #1 \right|_{#2}}

\newcommand{\mbc}{\mathbb{C}}
\newcommand{\mbd}{\mathbb{D}}

\newcommand{\mbf}{\mathbb{F}}

\newcommand{\mbh}{\mathbb{H}}

\newcommand{\mbn}{\mathbb{N}}

\newcommand{\mbp}{\mathbb{P}}

\newcommand{\mbr}{\mathbb{R}}

\newcommand{\mbz}{\mathbb{Z}}

\newcommand{\mca}{\mathcal{A}}
\newcommand{\mcb}{\mathcal{B}}

\newcommand{\mce}{\mathcal{E}}
\newcommand{\mcf}{\mathcal{F}}
\newcommand{\mcg}{\mathcal{G}}
\newcommand{\mch}{\mathcal{H}}
\newcommand{\mci}{\mathcal{I}}

\newcommand{\mcl}{\mathcal{L}}
\newcommand{\mcm}{\mathcal{M}}

\newcommand{\mco}{\mathcal{O}}
\newcommand{\mcp}{\mathcal{P}}

\newcommand{\mcs}{\mathcal{S}}
\newcommand{\mct}{\mathcal{T}}

\newcommand{\msa}{\mathscr{A}}

\newcommand{\msm}{\mathscr{M}}

\newcommand{\msp}{\mathscr{P}}

\newcommand{\mfg}{\mathfrak{g}}
\newcommand{\mfh}{\mathfrak{h}}

\newcommand{\mfm}{\mathfrak{m}}

\newcommand{\mfs}{\mathfrak{s}}

\newcommand{\mfu}{\mathfrak{u}}

\newcommand{\alp}{\alpha}
\newcommand{\sig}{\sigma}
\newcommand{\eps}{\epsilon}

\newcommand{\usig}{\bm{\sigma}}
\newcommand{\ulam}{\bm{\lambda}}
\newcommand{\ulamz}{\bm{\lambda}_0}
\newcommand{\ulami}{\bm{\lambda_\infty}}
\newcommand{\umu}{\bm{\mu}}

\newcommand{\uls}{\bm{s}}

\newcommand{\ulu}{\bm{u}}

\newcommand{\ula}{\bm{a}}
\newcommand{\ulb}{\bm{b}}
\newcommand{\ulc}{\bm{c}}
\newcommand{\ule}{\bm{e}}

\newcommand{\ulm}{\bm{m}}
\newcommand{\ulD}{\underline{D}}

\newcommand{\ulzero}{\bm{0}}

\newcommand{\tr}{\text{tr}}

\newcommand{\Hom}{\text{Hom}}

\newcommand{\re}{\text{Re}}
\newcommand{\ord}{\text{ord}}

\newcommand{\rk}{\text{rank}}

\newcommand{\tx}[1]{\text{#1}}

\newcommand{\pmt}[1]{\begin{pmatrix} #1 \end{pmatrix}}

\usepackage{scalerel,stackengine}
\stackMath
\newcommand\reallywidehat[1]{%
\savestack{\tmpbox}{\stretchto{%
  \scaleto{%
    \scalerel*[\widthof{\ensuremath{#1}}]{\kern-.6pt\bigwedge\kern-.6pt}%
    {\rule[-\textheight/2]{1ex}{\textheight}}
  }{\textheight}%
}{0.5ex}}%
\stackon[1pt]{#1}{\tmpbox}%
}
\newtheorem{theorem}{Theorem}[section]
\newtheorem{thm}{Theorem}[section]
\newtheorem{lem}[theorem]{Lemma}
\newtheorem{prop}[theorem]{Proposition}
\newtheorem{cor}[theorem]{Corollary}

\newtheorem{defn}[theorem]{Definition}

\newcommand{\lam}{\lambda}

\newcommand{\bet}{\beta}
\newcommand{\gam}{\gamma}

\newcommand{\lto}{\longrightarrow}
\newcommand{\lmapsto}{\longmapsto}
\newcommand{\lxrightarrow}[1]{\xrightarrow[\hspace*{.75cm}]{#1}}

\begin{document}

\title{Limiting configurations for the SU(1,2) Hitchin equation}
\author{Xuesen Na}
\date{\today}

\maketitle

\begin{abstract}
We study the limiting behavior of the solutions $h_t$ of the Hitchin's equation associated with a family of stable SU(1,2) Higgs bundles $(L,F,t\beta,t\gamma)$ on a compact connected Riemann surface $X$ as $t\to\infty$ under the assumption that the quadratic differential $q=\beta\cdot\gamma$ have simple zeros at $D$. The spectral data of the SU(1,2) Higgs bundle $(L,F,\beta,\gamma)$ can be represented by a Hecke modification of $V=L^{-2}K_X\oplus LK_X$. We show by a gluing construction that after appropriate rescaling, the limit is given by a metric on $V$ singular at $D$, induced by harmonic metrics adapted to parabolic structures on $L$ and on $K_X$ at $D$. We give rules to determine the parabolic weights of the limit.
\end{abstract}

\section{Introduction}
\label{sec:intro}

Let $X$ be a compact connected Riemann surface of genus $g \ge 2$ with a K\"ahler form $\omega$ with $\int_X \omega=2\pi$. Let $(E,h_0)$ be a rank $r$ holomorphic Hermitian vector bundle and $\Phi\in \Omega^{1,0}(\tx{End}(E))$. The Hitchin's self-duality equation is an equation of the pair $(A,\Phi)$:
\ben
i\Lambda_\omega(F_A+[\Phi\wedge \Phi^\ast])=\mu,\,\, \bar\pd_A\Phi=0
\een
with $\mu=\deg(E)/r$, $A$ an $h_0$-unitary connection, $\Phi^\ast$ the $h_0$-adjoint of $\Phi$. This is equivalent to the Hermitian-Yang-Mills equation
\ben
i\Lambda_\omega(F_{\nabla_h}+[\Phi\wedge\Phi^{\ast_h}])=\mu
\een
with $\nabla_h$ Chern connection of $h$. It was studied by Hitchin \cite{Hit87} as dimension reduction of the self-dual Yang-Mills equation. Together with the work Simpson \cite{Sim88}, the gauge-theoretic moduli space of its irreducible solutions $\mcm$, called Hitchin moduli space, is identified with the moduli space of stable Higgs bundles $(E,\Phi)$. The rich structures on $\mcm$ build bridges between the fields of algebraic geometry, symplectic geometry and topology.

An important feature of $\mcm$ is the Hitchin map $\tx{Hit}: \mcm\to\mcb$ given by taking the characteristic polynomial of $\Phi$. The space of coefficients $\mcb=\bigoplus_{j}H^0(X,K_X^j)$ is a finite dimensional vector space. Equivalently $\mcb$ parametrizes the space of spectral curves: for $b\in \mcb$, the curve $\Sigma_b\subset |K_X|$ may be viewed as marking eigenvalues of $\Phi$. It is given by divisor $\det(\lambda-p^\ast \Phi)$ of $p^\ast K_X^r$ where $p:|K_X|\to X$ is the projection and $\lambda\in p^\ast K_X$ is the tautological section. For generic choice of $b\in\mcb$, $\Sigma_b$ is smooth.

The Hitchin moduli space $\mcm$ is not compact. As a consequence of Uhlenbeck's weak compactness, both $\|\Phi\|_{L^2}$ and the Hitchin map are proper (see proofs in \cite{Wen16}). A sequence $(E_j,\Phi_j)$ on $\mcm$ escapes to infinity only if $\|\Phi_j\|_{L^2}\to\infty$ and $\tx{Hit}(\Phi_j)\to\infty$ on $\mcb$. It is an interesting problem to study the limiting behavior of the family of stable Higgs bundles $(E,t\Phi)$ as $t\to\infty$. Gaiotto-Moore-Netizkhe \cite{GMN10} conjectured that as radial variable $t\to\infty$ on the Hitchin base, the natural hyperk\"ahler metric -- the Hitchin metric on $\mcm$ is asymptotic to a semi-flat metric, constructed from the special K\"ahler structure on $\mcb$. Inspired by this work, Mazzeo-Swoboda-Wei\ss-Witt \cite{MSWW16} studied the above family for rank-two Higgs bundles where the quadratic differential $\det \Phi$ has simple zeros. This condition is equivalent to the spectral curve being smooth. In particular they constructed solutions $h_\infty$ called limiting configurations which solve the decoupled version of the Hitchin's equation
\ben
F_{\nabla_{h_\infty}}^\perp=0,\,\, [\Phi\wedge \Phi^{\ast_{h_\infty}}]=0
\een
where $\perp$ it the trace-free part. They then showed that for each $h_\infty$ there is a family $(E,t\Phi)$ whose Hermitian-Yang-Mills metric $h_t$ converge in $C^\infty$ sense to it, away from the zeros of $\det \Phi$. Fredrickson \cite{ Fre18} extended the convergence result to arbitrary rank with smooth spectral curve. Mochizuki \cite{Moc16} obtained similar convergence (after appropriate rescaling) more generally for any rank-two Higgs bundle with generically semisimple Higgs field. 

The situation of limiting behaviors in higher rank with a singular spectral curve is more complicated and largely unexplored. The choice of the real form SU(1,2) of the complex Lie group SL$(3,\mathbb{C})$ thus provides a natural next step in this direction. In generic case, its spectral curve is reduced with two irreducible components meeting at nodal singularities and the form of the Higgs field is more restrained than the general SL$(3,\mathbb{C})$-Higgs bundle with same spectral curve. 

An SU(1,2) Higgs bundle is given by a triple $(F,\beta,\gamma)$, or equivalently an SL$(3,\mathbb{C})$-Higgs bundle of the form
\begin{equation}
\left(E=L\oplus F,\,\, \Phi=\begin{pmatrix}
0 & \beta\\
\gamma & 0
\end{pmatrix}\right)
\end{equation}
where $L=\det F^\ast$, $q=\gamma\circ \beta\in H^0(X,K_X^2)$ is a quadratic differential and the spectral curve $\Sigma$ is given by the divisor $Z(\lambda(\lambda^2-q))$ on $|K_X|$. We assume all zeros of $q$ are simple and let $D=p_1+\ldots+p_{4g-4}$ be the zero divisor. Therefore $\Sigma$ has $4g-4$ simple nodes. The goal of this article is to study the limiting behavior of the Hermitian-Yang-Mills metric $h_t$ for such a family $(F,t\beta,t\gamma)$ of stable SU(1,2) Higgs bundle as $t\to\infty$.

In order to describe the SU(1,2) limiting configuration, we need a good description of the Hitchin fiber of the Hitchin map $\mcm_{\tx{SU(1,2)}}\to \mcb_{\tx{SU(1,2)}}=H^0(X,K_X^2)$. This is provided by the author in \cite{Na21}: with $L$ and $q$ fixed, the data in $(F,\beta,\gamma)$ is equivalent to a holomorphic Hecke modification of $L^{-2}K_X\oplus LK_X$ at $D$: an injective map of $\mathcal{O}_X$-modules $\iota: F\to L^{-2}K_X\oplus K_X$ isomorphic over $X-D$. Solutions $h_\infty$ of the decoupled equation has the form
\ben
h_\infty=\iota^\ast\rd{h_L^{-2}h_K\oplus h_Lh_K}
\een
with $h_K$ resp. $h_L$ metrics on $K_X$ resp. $L$ such that induced metric on $K_X^2$ satsifies $\vb{q}\equiv 1$ and $h_L$ is a harmonic metric adapted to filtered line bundle $(L,\ulam)$ for some weights with $\sum_j \lambda_j=-\deg L$. The difference between two Hermitian metrics is an automorphism. We write $h'=h\cdot g$ if $h'(\sigma,\mu)=h(g\sigma, \mu)$ and we say that $h_t\to h_\infty$ in $C^\infty$ if $h_t=h_\infty\cdot g_t$ and $g_t\to 1$ in $C^\infty$ under some fixed background metric. The main result of this article is the following

\begin{thm} \label{thm:main}
Fix $x_0\in X-D$, $v_0\in \atp{L}{x_0}$ and $X_0\subset X-D$ a compact set. Let $h_t$ be Hermitian-Yang-Mills metric of $(F,t\beta,t\gamma)$ and $\iota: F\to L^{-2}K_X\oplus LK_X$ the Hecke modification corresponding to $(F,\beta,\gamma)$. Identify the two bundles by $\iota$ over $X-D$ and let $\widetilde h_t$ the normalization of $h_t$ (by a diagonal automorphism) such that
\ben
\vb{(v_0^{-2}\sqrt{q(x_0)},0)}=\vb{(0,v_0\sqrt{q(x_0)})}=1\,.
\een
Let $h_\infty=h_{L,\infty}^{-2}h_K\oplus h_{L,\infty}h_K$ where $h_{L,\infty}$ is adapted to the filtered bundle $\rd{L,\ulami}$ over $(X,D)$ such that 
\begin{itemize}
\item $\lambda_{\infty,j}=1/4$ (resp. $-1/4$) for $\beta(p_j)=0$ (resp. $\gamma(p_j)=0$), and
\item $\lambda_{\infty,j}$ is a constant independent of $j$ for $\beta(p_j)$, $\gamma(p_j)\neq 0$
\item $\sum_j \lam_{\infty,j}=-\deg L$
\end{itemize}
Then we have
\ben
\widetilde h_t\xrightarrow{\tx{in }C^\infty} h_\infty\,\,\tx{ as }\,\,t\to\infty
\een
\end{thm}

The method of proof is a combination of that in \cite{MSWW16} and \cite{Moc16}. In \cite{MSWW16}, the authors constructed an approximate solution $h_t^{\tx{app}}$ by gluing local model solutions on small disks around points in $D$ to a limiting configuration. The local model solution has an explicit form in terms of solutions of an ODE of Painlev\'e III type, asymptotic to the limiting configuration outside the disk when $t\gg 1$. They then showed by contraction mapping principle that for $t\gg 1$ there is a Hermitian-Yang-Mills metric $h_t$ close to $h_t^{\tx{app}}$. Let $\mch_t(u)=F_{\nabla_{h'}}+[\Phi\wedge \Phi^{\ast_{h'}}]$ where $h'=h_t^{\tx{app}}\cdot e^u$. The key estimates are certain lower bounds of the operator $L_t$ given by linearizing $\mch_t$ at $u=0$. In \cite{Moc16} an explicit local model is not readily available, the author instead relied on existence theorem of Hermitian-Yang-Mills metric for wild harmonic bundle by Biquard-Boalch \cite{BB04} and convergence is not proven through contraction mapping principle. As a consequence, the rate of convergence to limiting configuration is not given in \cite{Moc16}.

The local model solution needed to prove Theorem \ref{thm:main} is constructed as follows. Note first there are three types of simple zeros of $q=\gam\circ\bet$. It can either be a zero of $\bet$, a zero of $\gam$ or neither. The Hecke modification map $\iota: F\to L^{-2}K_X\oplus LK_X$ provides a choice of local frames in which the Higgs field has a certain standard local form near each type of zero. For the first two types, we use an explicit local model solution in terms of Painlev\'e transcendentals. Near the boundary of the disk for $t \gg 1$, these local models are asymptotic to the decoupled solution with parabolic weight on $L$ given by $\lambda=\pm 1/4$. For the last type, we apply Biquard-Boalch theorem and construct it from a wild harmonic bundle over $\mbp^1$ with SU(1,2) symmetry. These local models are asymptotic to decoupled solution with parabolic weight $-1/4<\lambda<1/4$. The stability condition restricts possible numbers of zeros of each type. As a result, the space of admissible tuple of weights $\ulam$ is contained in a convex polytope inside a hyperplane in $\mbr^{4g-4}$. The set of partitions of zeros in $D$ satisfying stability condition is in one-to-one correspondence with the faces of the polytope.

For a fixed $t\ge 1$, the parabolic weight $\lambda$ determines leading order terms $\sim\lambda\log|z|$ for the harmonic metrics on $L$ as a part of local model solution. However by the uniqueness in Biquard-Boalch theorem, the next-to-leading order is also determined by the weight. In order to construct a good approximate solution $h_t^{\tx{app}}$ by gluing the local model to decoupled solution given by a tuple of weights $\ulam=(\lambda_1,\ldots,\lambda_{4g-4})$, it will be necessary to match the next-to-leading order as well. In fact, this issue is also discussed in \cite{Moc16} and we adapt the proof to show that the next-to-leading order constant $c_{\lam}$ depends continuously on $\lambda$ and use it to show that for a partition of zeros corresponding to a certain face of the polytope of admissible weights, for $t\gg 1$ there exists a family $\ulam(t)$ matching the next-to-leading order. Furthermore we show that as $t\to\infty$, the matching weight $\ulam(t)\to \ulami$ given in Theorem \ref{thm:main}, which is also characterized by being the centers of the corresponding face in the polyhedron of admissible weights.

The approximate solution $h_t^{\tx{app}}$ is then constructed by gluing local model solutions to decoupled solution outside of the disks, for the carefuly chosen tuple of weights $\ulam(t)$. Due to this $t$-dependency of weight, the ensuing analysis including the key estimates of $L_t$ are significantly more complicated than that of \cite{MSWW16}. This and other complications (e.g. a $t$-independent $L^2$-lower bound of operator $L_t$ does not exist, see \S \ref{sec:L2lowerbound}) partily explains the length of the article.

There are a multitude of potential avenues to continue and extend the present work. Limiting configurations serves as natural class of objects to compactify the Hitchin moduli space. A clear picture of the interplay between spectral data for more general singular spectral curves and the limiting behavior of solutions is an important topic for this program. Another potential direction is the asymptotic geometry on the moduli space of $G$-Higgs bundles. Furthermore, there has recently been some important progress (see \cite{OSWW20}) connecting the SL$(2,\mathbb{C})$-limiting configurations via non-abelian Hodge correspondence to equivariant pleated surfaces in $\mathbb{H}^3$. It will be interesting to apply and extend the present work to explore limiting objects on the other side of non-abelian Hodge correspondence for both SU(1,2) and perhaps other rank-one Lie groups such as $\tx{SU}(1,n)$ for $n\ge 3$ and $\tx{SO}_0(1,n)$.

Lastly we comment on a recent work of Mochizuki \cite{MS23}. It contains two results related to the limiting behavior of $h_t$ for families of SL$(r,\mathbb{C})$-Higgs bundles. The first result is the independence of limit $h_\infty$ on choice of subsequences when the spectral curve is irreducible. The second is the exponential convergence of $h_t$ under the condition that the corresponding spectral data is given by a line bundle $L$ on the normalized spectral curve. Let $n:\widetilde\Sigma\to \Sigma$ be the normalization $\lambda'$ the pullback of $\lambda\in p^\ast K_X$ on $\widetilde \Sigma$. This condition applies when the Higgs bundle is given by $(E,\Phi)=((p\circ n)_\ast L,(p\circ n)_\ast \lambda')$ for a line bundle $L$ on $\widetilde\Sigma$. Both results go beyond the restriction of smoothness of the spectral curve. However, neither conditions apply to the present work. For the first condition, the spectral curve $\Sigma$ of an SU(1,2)-Higgs bundles discussed in this work is always reducible; for the second result with spectral curve $\Sigma$, the only SU(1,2)-Higgs bundles induced by a line bundle $L$ on the normalization $\widetilde \Sigma$ are the strictly polystable ones. In that case, the question of limiting behavior is reduced to that of SU(1,1) $\cong$ SL$(2,\mathbb{R})$.

The rest of the article is organized as follows. In \S \ref{sec:setup} we introduce some notations and conventions, and review the notion of filtered bundle as well as the description of spectral data of SU(1,2) Higgs bundles using Hecke modification. In \S \ref{sec:constapproxsoln}, we construct local model and approximate solution. In \S \ref{sec:proof}, we study linearization of the equation and prove the main theorem.

\section*{Acknowledgements}
This article is a part of the Ph.D. thesis of the author. The author wishes to thank his advisor Prof. Richard Wentworth, for suggesting this problem and for his patient guidance and encouragement during my doctoral studies. The author is also thankful to Brian Collier, Johannes Horn, Qiongling Li and Andrew Neitzke for helpful discussions and comments. The author is supported by 2020-21 Patrick and Marguerite Sung Fellowship in Mathematics at the University of Maryland.

\section{Preliminaries}
\label{sec:setup}

\subsection{SU(1,2) Higgs bundle and stability conditions}
\label{sec:su12stab}

We review the definitions of a ($G$-)Higgs bundle, in particular an SU(1,2) Higgs bundle, the stability conditions, as well as the notion of a filtered bundle following notations in \S 3 of \cite{Moc16}. We will also review the description of spectral data for SU(1,2) Higgs bundles from \cite{Na21}. Throughout this work $X$ is a fixed closed Riemann surface of genus $g(X)\ge 2$. All vector bundles are holomorphic.

Let $G^c$ be a complex reductive Lie group. A $G^c$-Higgs bundle is a pair $(P,\Phi)$ with $P$ a holomorphic principal $G^c$-bundle over $X$ and $\Phi\in H^0(X,P\times_{\tx{Ad}}\mfg^c\otimes K_X)$. For $G^c\subset \tx{GL}(r,\mbc)$, a $G^c$-Higgs bundle can also be viewed as a rank-$r$ Higgs bundle. The celebrated non-abelian Hodge correspondence \cite{Don87, Cor88, Hit87, Sim88} established a homeomorphism between the moduli of polystable $G^c$-Higgs bundles and the moduli space of reductive representations $\tx{Hom}^+(\pi_1(X),G^c)\sslash G^c$. Hitchin \cite{Hit92} exploited this correspondence to study the topology of character variety of real representations $\tx{Hom}^+(\pi_1(X),\tx{PSL}(n,\mbr))\sslash \tx{PSL}(n,\mbr)$ and constructed the Hitchin component. Along this line, the notion of $G$-Higgs bundle is developed in \cite{BGG06}. Let $G$ be a connected reductive real Lie group, $H\subset G$ a maximal compact subgroup, and $\mfg=\mfh\oplus \mfm$ a Cartan decomposition. Let $H^c$ and $\mfm^c$ be the respective complexifications and $\iota: H^c\to GL(\mfm^c)$ the isotropy representation.

\begin{defn}
A $G$-Higgs bundle over $X$ is a pair $(P,\varphi)$ where $P$ is a holomorphic principal $H^c$-bundle over $X$, $\varphi\in H^0(X,P\times_\iota \mfm^c\otimes K_X)$.
\end{defn}

The Hitchin map for $G$-Higgs bundle is given by (see, e.g. \cite{GPR18})
\ben
\tx{Hit}: \mcm_G\to \mcb_G=\bigoplus_{i=1}^a H^0(X,K_X^{m_i})
\een
where $m_i$'s are the exponents of $G$ and $a$ is the real rank of $G$.

For $G^c=\tx{SL}(3,\mbc)$, there are up to conjugation three real forms: SU(3), SU(1,2) and SL$(3,\mbr)$. Higgs bundle corresponding to compact real form SU(3) has vanishing Higgs field. On the other hand, the real form SL$(3,\mbr)$ gives no restriction to the spectral curve. We will therefore focus on the real form $G=$SU(1,2). The Lie algebra
\ben
\mfg=\mfs\mfu(1,2)=\cl{X\in\tx{Mat}_3(\mbc)| X^\ast J+JX=0,\,\, \tr X=0}
\een
where $J=\tx{diag}(-1,1,1)$ and $X^\ast$ is the conjugate transpose of $X$. $\theta: X\mapsto -X^\ast=JXJ$ is the Cartan involution on $\mfg$. The Cartan decomposition $\mfg=\mfh\oplus \mfm$ is the eigendecomposition for $\theta$ with eigenvalues 1 and -1. We have
\ben
\mfh=\left\{\pmt{-\tr Y & 0 \\ 0 & Y}\,\middle|\, Y\in \mfu(2)\right\}, \,\, \mfm=\cl{\pmt{0 & Z \\ Z^\ast & 0}}
\een

\begin{defn} \label{def:su12}
An SU(1,2) Higgs bundle is a triple $(F,\bet,\gam)$: $F$ is a rank-two bundle, $\gam\in H^0(X,\Hom(L,F)\otimes K_X)$, $\bet\in H^0(X,\Hom(F,L)\otimes K_X)$ where $L=\det F^\ast$.

Equivalently $\gam: LK_X^{-1}\to F$, $\bet: F\to LK_X$ are holomorphic bundle maps. The composition $q=\bet\circ\gam: LK_X^{-1}\to LK_X$ is a holomorphic quadratic differential. We say that $(F,\bet,\gam)$ is an SU(1,2) Higgs bundle associated to $q$ and $L$.
\end{defn}

Via $G\to G^c=\tx{SL}(3,\mbc)$, an SU(1,2) Higgs bundle is also an SL$(3,\mbc)$ Higgs bundle
\ben
E=L\oplus F, \,\, \Phi=\pmt{0 & \bet \\ \gam & 0}
\een
The corresponding notions of (poly)stability are equivalent (see \cite{BGG03} for a more general statement).

\begin{prop}[Lemma 2.2, \cite{Got01}]
Let ($E$, $\Phi$) as above, $E'\subset E$ a $\Phi$-invariant subbundle. There are subbundles $L'\subset L$ and $F'\subset F$ such that $\mu(E')\le \mu(L'\oplus F')$.
\end{prop}

As a result, we need only test slope stability on nonzero proper subbundles of the form $L'\oplus F'$. Since $\deg E=0$, and the only subbundles of a line bundle $L$ are $0$ and $L$ itself, we have

\begin{prop} \label{prop:su12stab1}
$(F,\bet,\gam)$ is stable iff
\begin{itemize}
\item for a subbundle $0\neq F'\subsetneq F$, $\bet|_{F'}=0$ we have deg$(F')<0$, and
\item for a subbundle $F''\subsetneq F$ with $\gam(L)\subset F''\otimes K_X$ has deg$(L\oplus F'')<0$
\end{itemize}
\end{prop}

We make the following assumption throughout this work: $q=\bet\circ\gam$ has simple zeros. Let $D=p_1+\ldots+p_{4g(X)-4}$ be the zero divisor.

\begin{defn} \label{def:qD}
A zero of $q$ has three types: $\beta(p)=0$, $\gam(p)=0$, or neither vanishes. This gives a partition $D=D_\bet+D_\gam+D_r$. Let $d_\bet=\deg D_\bet$, $d_\gam=\deg D_\gam$, $d_r=\deg D_r=4g(X)-4-d_\bet-d_\gam$. Denote $\ulD=(D_\bet,D_\gam,D_r)$ to denote a partition of $D$ into three effective divisors.
\end{defn}

The Toledo invariant $\tau=2\deg L$ labels the connected components of the moduli space. We have the Milnor-Wood type inequality $|\tau|<2(g-1)$. With $d_\bet$, $d_\gam$ we have the following refinement.

\begin{prop} \label{prop:su12stab2}
$(F,\bet,\gam)$ is stable iff 
\be \label{eq:stabdef1}
-(g-1)+\ov{2}d_\bet<d<(g-1)-\ov{2}d_\gam
\ee
equivalently
\begin{align}
& d_\bet<2(g-1+d)\,\tx{ and,} \label{eq:stabdef3} \\
& d_\gam<2(g-1-d) \label{eq:stabdef2}\,.
\end{align}
In particular if $(F,\bet,\gam)$ stable then $d_r>0$. Furthermore, $(F,\bet,\gam)$ is strictly polystable iff $<$ in above inequalities are replaced with equality, in which case $F\cong L^{-2}K_X^{-1}(D_\bet)\oplus LK_X^{-1}(D_\gam)$.
\end{prop}

In the following we will say $\ulD$ is stable (resp. strictly polystable) if (\ref{eq:stabdef3}), (\ref{eq:stabdef2}) (resp. their equality versions) hold.

\begin{proof}
The saturation of img$(\gam)$ gives a subbundle $F''\cong LK_X^{-1}(D_\gam)$. This is the unique proper subbundle of $F$ such that $\gam(L K_X^{-1})\subset F''$. We have $\deg F''=-d-2(g-1)+d_\gam$.

On the other hand, $\bet: F\to LK_X$ factors through the surjective sheaf map $\widetilde \bet: F\to LK_X(-D_\bet)$. We have that $\tx{ker}(\bet)=\tx{ker}(\widetilde \bet)$ giving a line subbundle $F'\subset F$. This is the unique nonzero proper subbundle of $F$ such that $\left.\bet\right|_{F'}=0$. We have $\deg F'=\deg F-\deg LK_X(-D_\bet)=-2d-2(g-1)+d_\bet$. The statements about stability now follow from Prop \ref{prop:su12stab1}.

The spectral curve of $(E,\Phi)$ (corresponding to $(F,\bet,\gam)$ as above) is given by zero divisor of the section $\lam(\lam^2-q)\in \pi^\ast K_X^3$ where $\pi: |K_X|\to X$ is the projection. This has two irreducible components (one of which is the zero section in $|K_X$) corresponding to the two $\Phi$-invariant subbundles (one of which is ker$(\bet)\cong L^{-2}K_X^{-1}(D_\bet)$). $(F,\bet,\gam)$ is polystable iff $E$ is a direct sum of these two subbundles, i.e. $F\cong L'\oplus L^{-2}K_X^{-1}(D_\bet)$, each of which a stable Higgs bundle of slope 0. Under this identification $E\cong \rd{L\oplus L'}\oplus L^{-2}K_X^{-1}(D_\bet)$, $\Phi=\Phi'\oplus 0$ where
\ben
\Phi=\pmt{0 & \bet' \\ \gam' & 0}
\een
We have $D_\bet=Z(\bet')$, $D_\gam=Z(\gam')$, $D=D_\bet+D_\gam$ and $\deg L'=d-2(g-1)+d_\gam=-d$ and $\deg L^{-2}K_X^{-1}(D_\bet)=-2d-2(g-1)+d_\bet=0$. These implies the equalities instead of $<$ in (\ref{eq:stabdef3}), (\ref{eq:stabdef2}). Conversely if equalities instead of $<$ holds in (\ref{eq:stabdef3}), (\ref{eq:stabdef2}), $D=D_\bet+D_\gam$. We have that $\gam$ factors through $s_\gam: LK_X^{-1}\to F''$ with simple zeros at $D_\gam$, $\bet=s_\bet\circ\widetilde\bet$ where $s_\bet: LK_X(-D_\bet)\to LK_X$ with simple zeros at $D_\bet$. Since $\bet\circ\gam=q: LK_X^{-1}\to LK_X$ with simple zeros at $D=D_\bet+D_\gam$, the composition $F''\to F\xrightarrow{\widetilde\bet}LK_X(-D_\bet)$ is an isomorphism. Since $\mco_X(D)\cong K_X^2$ and $D=D_\bet+D_\gam$, we have $F''\cong LK_X^{-1}(D_\gam)\cong LK_X(-D_\bet)$. Therefore $F\cong F''\oplus L^{-1}K_X(D_\bet)$, where the latter summand is ker$(\bet)$. It follows that $(E,\Phi)$ is a direct sum of $\Phi$-invariant stable Higgs bundles of slope 0.
\end{proof}

\subsection{SU(1,2) Hitchin equation}
In this part we write down the Hitchin's equation for an SU(1,2) Higgs bundle in global and local form. We also fix some convention and notations about local form of a Hermitian metric and some other related objects.

\begin{thm}
For $(F,\bet,\gam)$ stable, there is a unique metric $h$ on $F$ such that
\be \label{eq:hitchin}
F_{\nabla_h}+\gam^{\ast_h}\wedge\gam+\bet\wedge\bet^{\ast_h}=0
\ee
\end{thm}

In more generality the above is a consequence of Theorem 3.21 of \cite{GGM12}. By viewing SU(1,2) Higgs bundle as SL$(3,\mbc)$-Higgs bundle, there is also a more straightforward proof:

\begin{proof}
By the existence and uniqueness theorem of Hitchin and Simpson \cite{Hit87, Sim88}, there is a unique Hermitian metric $H$ on $E=L\oplus F$ solving Hitchin equation $F_{\nabla_H}+[\Phi\wedge \Phi^{\ast_H}]=0$ inducing the standard metric on $\det E=\mco_X$. Let $\alpha=1_L\oplus(-1_F)$, we have $\alpha^{-1}\Phi\alpha=-\Phi$. By uniqueness, $\alpha^\ast H=H$. Thus $H=\tx{diag}(\det h^{-1},h)$ with $h$ a Hermitian metric on $F$. The form (\ref{eq:hitchin}) follows by a direct calculation.
\end{proof}

For $\uls=(s_1,\ldots,s_r)$ be a local frame, the matrix $h_{\uls}$ is defined by $\rd{h_{\uls}}_{ij}=h(s_j,s_i)$. For $\psi\in \Hom(E,E')$ with local frames $\uls$, $\ulu$, we define the matrix $\psi_{\uls,\ulu}$ by $\psi s_j=\sum_j\rd{\psi_{\uls,\ulu}}_{ij}u_j$. We write $\psi_{\uls}$ if $E'=E$, $\ulu=\uls$. Let $(E_i,h_i)$, $i=1,2$ be Hermitian vector bundles. Let $\bm{s_i}$ be local frames of $E_i$ and $H_i=(h_i)_{\bm{s_i}}$. Let $\phi\in\Hom(E_1,E_2)$ with $M=\phi_{\bm{s_1},\bm{s_2}}$. Then under induced metric we have $|\phi|_{h_1,h_2}^2=\tr(\phi\phi^\ast)=\tr(MH_1^{-1}M^\ast H_2)$. Let $A, H\in \tx{Mat}_n(\mbc)$ where $H\ge 0$ is Hermitian, we denote
\be \label{eq:Hnorm}
\vb{A}_H^2=\tx{tr}\rd{AH^{-1}A^\ast H}
\ee
Let $(F,\bet,\gam)$ be an SU(1,2) Higgs bundle and $\cl{s_1,s_2}$ local frame of $F$ over $(U;z)$, $\cl{(s_1\wedge s_2)^{-1}}$ the induced frame on $L=\det F^\ast$. Let $\bet_{\uls}=\bet_0 dz$, $\gam_{\uls}=\gam_0 dz$, $H=h_{\uls}$, then (\ref{eq:hitchin}) is given by the following equation
\be \label{eq:localform}
\pd_{\bar z}\rd{H^{-1}\pd_z H}=\gam_0\gam_0^\ast H \rd{\det H}-\rd{\det H}^{-1} H^{-1}\bet_0^\ast\bet_0
\ee

\subsection{Filtered bundle}
\label{sec:filt}

In the following, we review the notion of filtered bundle following \S 3 of \cite{Moc16}. Let $D\subset X$ be a finite set of $N$ points and $\mco=\mco_X$, $\mco'=\mco_X(\ast D)$ sheaf of meromorphic functions with poles in $D$. Let $\mce$ be a locally free $\mco'$-module. A metric on $\mce$ will mean a metric on $\atp{\mce}{X-D}$.

\begin{defn} \label{def:fil}
A filtered bundle structure on $\mce$ over $(X,D)$ is a family of coherent $\mco$-submodules $\mcp_\ast \mce$ labeled by a tuple $\ula=\rd{a_P}_{P\in D}$, $a_P\in \mbr$ such that
\begin{itemize}
\item $\mcp_{\ula}\mce\otimes_\mco \mco'=\mce$,
\item the stalk of $\mcp_{\ula}\mce$ at $p\in D$ depends only on $a_P\in \mbr$, denoted by $\mcp_{a_P}\mce_P$,
\item $\mcp_{a}\mce_P\subset \mcp_{b}\mce_P$ for $a\le b$ and there is $\epsilon>0$ such that $\mcp_{a}\mce_P=\mcp_{a+\epsilon}\mce_P$,
\item $\mcp_{a}\mce_P\otimes_{\mco_{X,P}}\mco(nP)_P=\mcp_{a+n}\mce_P$ for $n\in \mbz$.
\end{itemize}
A filtered Higgs bundle is a pair $(\mcp_\ast \mce, \theta)$ where $\theta\in H^0(X,\Endo(\mce)\otimes \Omega_X^1)$. $(\mcp_\ast \mce, \theta)$ is an unramifiedly good filtered Higgs bundle if at each $P$ there is a finite subset of germs of meromorphic functions, $\mci(P)\subset \mco'_P$ with $\mci(P)\to \mco_X(\ast D)_P/\mco_{X,P}$ injective, and a decomposition $\mcp_a\mce_P = \bigoplus_{f\in \mci(P)} \mcp_a\mce_{P,f}$ such that
\ben
\rd{\theta-\rd{df}\tx{Id}} \mcp_a\mce_{P,f}\subset \mcp_a\mce_{P,f}\otimes \Omega_X^1(\log D)_P
\een
For each $a\in \mbr$, let $\delta\mcp_a\mce_P = \mcp_a\mce_P/\mcp_{<a}\mce_P$.
\end{defn}

Note that by the last property in the above definition, $\mcp_{a_P}^P \mce_P/\mcp_{a_P-1}^P\mce_P$ is the fiber of the vector bundle $\mcp_{\ula}\mce$ at $P$. For $-1<\alpha\le 0$, we have filtration $F_\alpha^P=\mcp_{a_P-1+\alpha}^P \mce_P/\mcp_{a_P-1}^P\mce_P$ of the fiber. In this way for any $\ula$, a filtered bundle structure $\mcp_\ast \mce$ is equivalent to a parabolic structure on holomorphic vector bundle $\mcp_{\ula}\mce$.

\begin{defn} \label{def:adapted}
A metric $h$ on $E$ is called adapted to $\mcp_\ast\mce$ on $(X,D)$ if
\begin{itemize}
\item $\mcp_{\ula}\mce \cong \mce$ as $\mco$-module on $X-D$ for all $\ula$, and
\item $s\in \mcp_{\ula}\mce$ iff
\be \label{eq:defadapted}
|s|_h=\mco\rd{|z_P|^{-a_P-\epsilon}},\,\,\, \forall \epsilon>0
\ee
where $z_P$ is a holomorphic coordinate centered at $P$. Given a metric $h$ on $\mce$, the second condition defines a filtered bundle structure $\mcp_\ast^h \mce$.
\end{itemize}
\end{defn}

For a tuple $\ulb$, define
\ben
\deg \mcp_\ast \mce=\deg \mcp_{\ulb} \mce-\sum_{P\in D}\sum_{b_P-1<a\le b_P} a\cdot \dim_\mbc \rd{\mcp_a\mce_P/\mcp_{<a}\mce_P}
\een
where $\mcp_{<a}\mce_P=\bigcup_{c<a}\mcp_c\mce_P$. This is independent of the choice of $\ulb$. For $\mcg\subset \mce$ locally free $\mco'$-submodules, let $\mcp_{\ula}\mcg=\mcp_{\ula}\mce\cap \mcg$. This gives a filtered bundle structure on $\mcg$. Let $\mu(\mcp_\ast\mce):=\deg\mcp_\ast\mce/\rk\, \mce$. A filtered Higgs bundle $(\mcp_\ast\mce,\theta)$ is stable if any nonzero proper $\theta$-invariant locally free $\mco'$-submodule $\mcg$ satisfies $\mu(\mcp_\ast\mcg)<\mu(\mcp_\ast\mce)$. The following theorem from \cite{BB04} is a generalization of \cite{Hit87, Sim88} on closed Riemann surfaces and of \cite{Sim90} on punctured Riemann surface for `tame' case and of \cite{Biq91} for parabolic vector bundles.

\begin{thm} \label{thm:biquard}
Let $(\mcp_\ast\mce,\theta)$ be a stable unramifiedly good filtered Higgs bundle with $\deg \mcp_\ast \mce=0$. Then there exists $h$ (unique up to $\mbr_+$) adapted to $(\mcp_\ast\mce,\theta)$ with $F_{\nabla_h}+[\theta\wedge\theta^{\ast h}]=0$ on $X-D$.
\end{thm}

\begin{thm} \label{thm:paralb}
Let $\mcl$ be a line bundle and $\ulam=(\lam_1,\ldots,\lam_N)$ a tuple of real numbers satisfying
\be \label{eq:pardegzero}
\deg\mcl +\sum_{j=1}^N \lam_j=0
\ee
Let $(U_j;z_j)$ be disjoint local charts centered at $p_j\in D$ and $s_j$ a nowhere vanishing section over $U_j$. Then there is a metric $h$ on $\mcl$, unique up to positive constant, such that $F_{\nabla_h}=0$ on $X-D$ and 
\be \label{eq:adapted}
\log |s_j|_h=\lam_j \log |z_j|+O(1) \,\, \tx{ as }\,\, \vb{z_j}\to 0
\ee
We call $h$ a harmonic metric adapted to filtered line bundle $(\mcl,\ulam)$.
\end{thm}

\begin{proof}
Let $\widetilde\mcl=\mcl\otimes_{\mco}\mco'$ and $\mcl_{\ulm}$ the $\mco$-submodule generated by $z_j^{-m_j}s_j$ over $U_j$. We have $\mcl_{\ulm}=\mcl\rd{\sum_j m_j[p_j]}$. Let $\mcp_{\ulb}\widetilde \mcl=\mcl_{\fl{b_1+\lam_1},\ldots,\fl{b_N+\lam_N}}$. We have $\deg \mcp_\ast \widetilde\mcl=\deg \mcl+\sum_j \lam_j=0$. By Theorem \ref{thm:biquard}, there is a harmonic metric $h$ on $\widetilde\mcl$ adapted to $\mcp_\ast \widetilde\mcl$ unique up to positive constant.

Note that $u_j:=\log |s_j|_h$ is a harmonic function on $U_j-\{p_j\}$. Take $\xi < \lam_j$. Since $s_j\in \mcp_{-\ulam}$ we have by (\ref{eq:defadapted}) that there is some constant $C$ such that $-u_j+\xi\log|z_j|+C\ge 0$ for $|z_j|$ small enough. By B\^ocher's Theorem (see, e.g. \cite{Axl01} Theorem 3.9) characterizing positive harmonic functions on a punctured disk, we have $u_j=\mu\log|z_j|+\tx{Re}(f)$ with some $\mu\in\mbr$ and $f\in \mco(U_j)$. By (\ref{eq:defadapted}) $\mu\ge \lam_j$. If $\mu>\lam_j$ we have some $\ulb$ with $b_j=-\mu<-\lam_j$ such that $s_j\in \mcp_{\ulb}$, a contradiction. Therefore $\mu=\lam_j$.
\end{proof}

In the following we give a second proof in terms of solution to Poisson equations with prescribed logarithmic poles.

\begin{lem} \label{lem:laplace}
Let $\ula=(a_1,\ldots,a_N)$, $n=-\sum_j a_j$ and $D=\{p_1,\ldots,p_N\}$ finite set on $X$. Then there is $\varphi: X-D\to \mbr$, unique up to a constant on $X$, satisfying $\Delta_{\pd}\varphi=n$ where $\Delta_\pd=\pd^\ast\pd+\pd\pd^\ast=i\Lambda\bar\pd \pd$ on functions and $\varphi+2 a_j \log|z_j|$ are bounded near $p_j$.
\end{lem}

\begin{proof}
Suppose $n=0$, $N=2$ and $\ula=(1/2,-1/2)$. Let $D=\{p,q\}$ with $p\neq q$. Consider the divisor $p-q$ and denote by $\mcl$ the associated line bundle. Let $s\in \mcl\otimes_{\mco}\msm^\ast$ be the defining meromorphic section with a zero at $p$, a pole at $q$ and holomorphic on $X-D$. Let $h$ be a Hermitian-Einstein metric on $\mcl$, i.e. $i\Lambda F_{\nabla_h}=0$ on $X-D$. Set $\varphi=\log |s|_h$. For a chart $(U;z)$ (resp. $(V;w)$) centered at $p$ (resp. $q$), $\varphi-\log|z|$ (resp. $\varphi+\log|w|$) is bounded and we have $\Delta_\pd\varphi=0$ on $X-D$.

For $n=1$, $N=1$ and $\ula=(1)$. Let $D=\cl{p}$, $\mcl$ a line bundle of degree 1, $h$ (resp. $h_0$) a Hermitian-Einstein metric on $\mcl$ (resp. $\mcl(-p)$) and $s: \mcl(-p)\to \mcl$ the inclusion induced by $\mco(-p)\subset \mco$. As bundle map $s$ has a simple zero at $p$ and is isomorphic away from $p$. By this inclusion we view $h$ as a metric on $\left.\mcl(-p)\right|_{X-D}$ and write $h=h_0e^{\varphi}$. We have $F_{\nabla_{h_0}}=0$ away from $p$ and $\Delta_\pd\varphi=i\Lambda\bar\pd \pd \varphi=i\Lambda F_{\nabla_h}=1$ on $X-D$.

Let $\sig_0$ (resp. $\sig$) be a trivializing section for $\mcl(-p)$ (resp. $\mcl$) on a chart $(U;z)$ centered at $p$, we have $
e^{\varphi/2}\vb{\sig_0}_{h_0}=\vb{\sig_0}_h=\vb{f}\,\,\vb{\sig}_h$ with $f$ holomorphic with simple zero at $p$. Therefore $\varphi=2\log\vb{z}+O(1)$ as $z\to 0$

In general the solution is a linear combination of the above cases. Uniqueness follows from the fact that a harmonic funtion on $X$ is a constant.
\end{proof}

We now give a second proof of Theorem \ref{thm:paralb}:

\begin{proof}
Let $h$ be a solution of the Hermitian-Einstein equation for $\mcl$, i.e. $i\Lambda F_{\nabla_h}=\deg \mcl \tx{id}_{\mcl}$. By the above lemma, there is $\varphi$ on $X-D$ such that $\Delta_\pd\varphi=-\deg\mcl$ and $\varphi=2\lam_j\log|z_j|+O(1)$ at $p_j$. It is now easy to verify that $h e^{\varphi}$ is a harmonic metric adapted to filtered line bundle $(\mcl,\ulam)$.
\end{proof}

We first fix some local charts around zeros of $q$.

\begin{defn} \label{def:disc}
Let $(F,\bet,\gam)$ be a stable SU(1,2) Higgs bundle and $q=\bet\circ\gam$, $D=Z(q)=D_\bet+D_\gam+D_r$ as in Def \ref{def:qD}. Let $(\mbd_j;\zeta_j)$ be disjoint charts centered at $p_j\in D$ such that $q=\zeta_j(d\zeta_j)^2$ and $\mbd_j=\left\{\vb{\zeta_j}<R\right\}$ for some $R>0$. Denote by $\mbd_j^\times =\mbd_j-\cl{p_j}$. Let
\be \label{eq:defDjprimes}
\mbd_j'=\cl{|\zeta_j|<2R/3},\,\, \mbd_j''=\cl{|\zeta_j|<R/3}\,.
\ee
\end{defn}

The construction in the above proof allows us to get rid of the ambiguity of positive constant and give a $\ulam$-family of harmonic metric adapted to $(L,\ulam)$. This will be useful in constructing the approximate solution later.

\begin{defn}\label{def:someharmfcns}
Fix $h_{L,\tx{HE}}$ a solution to the Hermitian-Einstein equation for the line bundle $L$, i.e. $i\Lambda F_{\nabla_{h_{L,\tx{HE}}}}=\tx{deg}L\tx{id}_L$. Fix $G_j$ for $1\le j\le N$ (existence follows by Lemma \ref{lem:laplace}), such that
\ben
\Delta_\pd G_j=\begin{cases}
0 & 1\le j\le N-1, \tx{ on }X-\cl{p_j,p_{j+1}}\\
1 & j=N, \tx{ on }X-\cl{p_N}
\end{cases}
\een
and 
\ben
\begin{cases}
G_j-2\log\vb{\zeta_j}\tx{ bounded on }\mbd_j^\times & 1\le j\le N\\
G_j+2\log\vb{\zeta_{j+1}}\tx{ bounded on }\mbd_{j+1}^\times & 1\le j\le N-1
\end{cases}
\een
For $\ulam\in \mbr^N$, let
\be \label{eq:defvarphiulam}
\varphi_{\ulam}=\sum_{j=1}^N \rd{\sum_{\ell=1}^j\lam_\ell}G_j
\ee
Define
\be \label{eq:defh0Lulam}
h_{L,\ulam}^0=h_{L,\tx{HE}}e^{\varphi_{\ulam}}\,.
\ee 
This is a harmonic metric adapted to $(L,\ulam)$ with 
\ben
\begin{cases}
\Delta_\pd \varphi_{\ulam}=-d\\
\varphi_{\ulam}=2\lam_j \log|\zeta_j|+O(1)\tx{ as }\zeta_j\to 0
\end{cases}
\een
Define on $\mbd_\ell$ bounded harmonic functions
\ben
g_{j\ell}=\begin{cases}
G_j-2\log|\zeta_j| & \ell=j, 1\le j\le N\\
G_{j+1}+2\log|\zeta_{j+1}| & \ell=j+1, 1\le j\le N-1\\
G_j & \tx{otherwise}
\end{cases}
\een
\end{defn}

Note that the for filtered line bundles, the estimate in (\ref{eq:defadapted}) sharpens to (\label{eq:adapted}) thanks to B\^ocher's Theorem. We will need a similar sharpening for higher rank case. Let $\mce$ be a locally free sheaf of $\mco'$-modules and $h$ a metric on $\mce$ as a vector bundle on $X-D$. Fix $P\in D$ and $(U;z)$ a chart centered at $P$.
\ben
\mcp_a^{\vee}\mce^\ast_P=\left\{
s\in\mce^\ast_P \middle| 
\begin{array}{l}
\tx{for all }c\in \mbr, \,\, e\in \mcp_c^h\mce,\\
s(e)=O(|z|^{-a-c-\eps})\,\, \forall \,\, \eps>0
\end{array}
\right\}\,.
\een
Let $h^\ast$ be the induced metric on $\mce^\ast$. By Prop 3.1 of \cite{Sim90}, $\mcp_a^\vee\mce^\ast_P=\mcp_a^{h^\ast}\mce^\ast_P$ for all $a\in\mbr$. It is also straightforward to see that for nonzero representative $e\in \delta\mcp_a^h\mce_P$, there is $s\in \mcp_{-a}^\vee \mce^\ast_P$ such taht $s(e)=1$. The following result is a consequence of the proof of Lemma 6.2 in \cite{Sim90}.

\begin{prop} \label{prop:sharpenadaptedness}
Let $\mce$, $h$ and $(U;z)$ as above and suppose that near $P\in D$ we have
\ben
\vb{F_{\nabla_h}}_h\le f\in L^p
\een
Then
\ben
\log\vb{e}_h=-a\log\vb{z}+O(1)\,\,\tx{ as }\vb{z}\to 0
\een
where $e\neq 0 \in\delta\mcp_a^h\mce_P$.
\end{prop}

\begin{proof}
We view $e\in \delta\mcp_a^h\mce_P$ as a nonzero section defined in some open subset of $U$ containing $P$. For all $\eps>0$ there is $C_\eps$ with $\log \rd{|z|^{-a}\vb{e}_h}\le \eps\vb{\log |z|}+C_\eps$. Let $s\in \mcp_{-a}^{\vee}\mce^\ast_P=\mcp_{-a}^{h^\ast}\mce^\ast_P$ be such that $s(e)=1$. By Cauchy-Schwarz we have $1=\vb{s(e)}\le \vb{s}_{h^\ast}\vb{e}_h$. For all $\eps>0$ there is $C_\eps'$ such that
\ben
\log\rd{|z|^{-a}\vb{e}_h} \ge -\log\rd{|z|^a\vb{s}_{h^\ast}}\ge -\eps\vb{\log |z|}-C_\eps'\,.
\een
Let $g=\log\rd{|z|^{-a}\vb{e}_h}$, it follows that
\ben
\vb{\frac{g}{\log |z|}}\lto 0\,\,\tx{ as }\,|z|\lto 0\,.
\een
Let $\mbd$ be a sufficiently small disk centered at $P$. The fact that curvature decreases in holomorphic subbundle (spanned by $e\in \mce$) implies that we have $-\Delta g\le h$ away from $x$ for some $h\in L^p$. Let $u\in L_2^p(\mbd)\subset C^0(\mbd)$ be some bounded function such that $\Delta u=h$ and $g<u$ on $\pd \mbd$. We have $-\Delta(g-u)\le 0$ on $\mbd-\cl{P}$. We have that $\vb{(g-u)/\log |z|}\to 0$ as $|z|\to 0$. By \cite[Lemma 2.2]{Sim90}, $-\Delta(g-u)\le 0$ weakly on all of $\mbd$. By the maximum principle, $g\le u$, i.e. $g$ is bounded from above. Replace $g$ with $-g$, consider holomorphic subbundle spanned by $s\in \mce^\ast$, note that the induced metric has $F_{\nabla_{h^\ast}}=-F_{\nabla_h}^t$ (see e.g. Prop 4.3.7 (iii) \cite{Huy05}). Repeating the above argument shows that $g$ is bounded from below as well. The conclusion then follows.
\end{proof}

\subsection{SU(1,2) Higgs bundle as Hecke modification}
\label{sec:hecke}

We next review from \cite{Na21} the construction an stable SU(1,2) Higgs bundle $(F,\bet,\gam)$ associated to a $q\in H^0(X,K_X^2)$ and a line bundle $L$ via holomorphic Hecke modifications of 
\ben
V=L^{-2}K_X\oplus LK_X\,.
\een

A Hecke modification of a holomorphic vector bundle $E$ at $D=\cl{x_1,\ldots,x_N}$ is a pair $(\hat E,s)$ where $\hat E$ is a holomorphic vector bundle and $s:\atp{\hat E}{X-D}\to \atp{E}{X-D}$ is an isomorphism which induces isomorphism $\hat\mce\xrightarrow{\sim} \mce$ where $\hat\mce=\hat E\otimes_{\mco}\mco(\ast D)$, $\mce=E\otimes_{\mco}\mco(\ast D)$. This is related to the Hecke operator and the Hecke eigensheaves, which play central roles in the geometric Langlands program (see, e.g. \cite{Fre07}). \cite{Hor22} used the Hecke modification to parametrize the non-abelian part of the singular Hitchin fiber for SL$(2,\mbc)$-Higgs bundles. $(\hat E,s)$ is called a holomorphic Hecke modification if $s$ is a holomorphic bundle map $\hat E\to E$. The data in a holomorphic Hecke modification of $E$ at $D$ is equivalent to specifying a locally free subsheaf of $\mco_X$-modules of the same rank that is isomorphic on $X-D$ to $E$. We introduce three sets equipped with some equivalence relations.

\begin{defn} \label{def:threesets}
Let $\ulD=(D_\bet,D_\gam,D_r)$ be a partition of the finite set $D$ and set
\ben
\mcm_{L,q,\ulD}=\left\{
\begin{array}{l}
\rd{F,\bet,\gam}\\
\tx{SU(1,2) Higgs bundle}
\end{array} \,\middle|\, 
\begin{array}{l}
\det F^\ast=L, \,\, \bet\circ\gam=q \\
Z(\bet)=D_\bet,\, Z(\gam)=D_\gam\,.
\end{array}
\right\}\,,
\een
with $(F,\bet_1,\gam_1)\sim (F',\bet_2,\gam_2)$ if there is an isomorphism $\psi: F\to F'$, $\Lambda^2 \psi=\tx{id}_{L^\ast}$ with $\bet_1=\bet_2\circ\psi$, and $\gam_2=\psi\circ\gam_1$. Let
\ben
\mch_{L,q,\ulD}=\left\{
\rd{F,\iota}
\,\middle|\,
\begin{array}{l}
\iota: F\to V\,\,\tx{ injective map of $\mco$-modules, isomorphic over $X-D$} \\
\Lambda^2 \iota=q:\, L^{-1}\lto \det V \\
\iota\rd{\atp{F}{p_j}}\subseteq \atp{L^{-2}K}{p_j}\subset \atp{V}{p_j}\,\tx{ for }p_j\in D_\bet\\
\iota\rd{\atp{F}{p_j}}\subseteq \atp{LK}{p_j}\subset \atp{V}{p_j}\,\tx{ for }p_j\in D_\gam\\
\iota\rd{\atp{F}{p_j}}\nsubseteq \tx{ either summand in }\atp{V}{p_j}\,\tx{ for }p_j\in D_r
\end{array}
\right\}
\een
with $(F,\iota)\sim (F',\iota')$ if there is an isomorphism $\psi: F\xrightarrow{\sim}F'$, $\Lambda^2\psi=\tx{id}_{L^\ast}$ and $\lam\in\mbc^\times$ such that 
\be \label{eq:eqvrel01}
\tx{diag}\rd{\lam^2, \lam^{-1}}\circ \iota=\iota'\circ \psi
\ee
For a finite set $\delta\subset X$ and a line bundle $L$ on $X$, set
\ben
\mcf_{L,\delta}=\prod_{p\in \delta}\rd{\atp{L^3}{p}}^\times =\cl{\ulb=\rd{b_p}_{p\in \delta}}
\een
with $\ulb\sim \bm{b'}$ if there is $\tau\in \mbc^\times$ such that $b_p'=\tau b_p$ for all $p\in \delta$. For $\delta=D_r$, $\ulb\in \mcf_{L,D_r}$ is called an admissible Hecke parameter with respect to $\ulD$.
\end{defn}

Note that the equivalence classes in $\mcm_{L,q,\ulD}$ are the isomorphism classes of stable SU(1,2) Higgs bundles associated to $q$ and $L$. The following gives a correspondence between the three sets, see \cite{Na21} for more details of the corresponding moduli problem.

\begin{thm} \label{thm:hecke}
There are bijections
\begin{center}
\begin{tikzcd}
\mcm_{L,q,\ulD} \arrow[r,"\sim"] & \mch_{L,q,\ulD} \arrow[r,"\sim"] & \mcf_{L,D_r} \\
\rd{F,\bet,\gam} \arrow[r,mapsto] & \rd{F,\iota} \arrow[r,mapsto] & \ulb
\end{tikzcd}
\end{center}
respecting the equivalence relations. For each $\ulb\in \mcf_{L,D_r}$ and $p_j\in D_r$, there are frame $s_{0,j}$ of $L$, $\cl{s_{1,j},s_{2,j}}$ of $F$, and $\sig_{1,j}=\rd{\rd{s_{0,j}}^{-2}d\zeta_j,0}$, $\sig_{2,j}=\rd{0,s_{0,j}d\zeta_j}$ of $V$ over $\mbd_j$ with
\begin{itemize}
\item $s_{0,j}^{\otimes 3}=b_{p_j}$ for each $p_j\in D_r$, 
\item local form of $\iota: F\to V$ is given by $\ov{\sqrt{2}}\pmt{\zeta_j & -1 \\ \zeta_j & 1}$,
\item $s_{1,j}\wedge s_{2,j}=s_{0,j}^{-1}$, and
\item $\bet$, $\gam$ have local forms over $\mbd_j$ with $p_j\in D_r$ given by
\ben
\bet=\rd{1/\sqrt{2}}\pmt{\zeta_j & 1}d\zeta_j,\,\,\gam=\rd{1/\sqrt{2}}\pmt{1 & \zeta_j}^Td\zeta_j \,.
\een
\end{itemize}
We say $\usig_j=\rd{\sig_{1,j},\sig_{2,j}}$, $\uls_j=\rd{s_{1,j},s_{2,j}}$ are induced by $\cl{s_{0,j}}$ of $\atp{L}{\mbd_j}$.
\end{thm}

\begin{proof}
For a rank-two vector bundle $\mce$, denote $\phi_\mce: \mce^\ast\otimes\det \mce\xrightarrow{\sim}\mce$ with $\ell\otimes s_1\wedge s_2\mapsto \ell(s_2)s_1-\ell(s_1)s_2$. This is an isomorphism that induces identity on $\det\mce$.

Given $(F,\bet,\gam)\in \mcm_{L,q,\ulD}$, set $\iota_2=\bet$, and let $\iota_1$ be the composition
\ben
F\lxrightarrow{\phi_F} F^\ast \otimes L^{-1} \lxrightarrow{\gam^T\otimes 1}L^{-1}K\otimes L^{-1}=L^{-2}K\,.
\een
Let $\iota=\iota_1\oplus \iota_2: F\to V=L^{-2}K\oplus LK$. Let $\cl{s_1,s_2}$ be a local holomorphic frame of $F$ over $(U,\zeta)$. The induced local frame of $L$ consists of a nowhere vanishing holomorphic section $(s_1\wedge s_2)^{-1}$. Under these frames there are $\bet_1$, $\bet_2$, $\gam_1$, $\gam_2\in \mco(U)$ such that
\be \label{eq:betagammaform01}
\bet=\pmt{\gam_1 & \gam_2}d\zeta, \,\, \gam=\pmt{\bet_1 & \bet_2}^T d\zeta
\ee
and
\be \label{eq:iotaform01}
\iota=\pmt{\gam_2 & -\gam_1\\ \bet_1 & \bet_2}\,.
\ee
It follows that $\iota\rd{\atp{F}{p_j}}$ lies in fiber of first (resp. second) summand iff $\bet(p_j)=0$ (resp. $\gam(p_j)=0$), therefore we get a map $\mcm_{L,q,\ulD}\to \mch_{L,q,\ulD}$. From the above local form, it is also clear that the map is a bijection preserving respective equivalent relations.

Given $(F,\iota)\in \mch_{L,q,\ulD}$, we have an exact sequence of $\mco$-modules
\begin{center}
\begin{tikzcd}
0\arrow[r] & F\arrow[r,"\iota"] & V=L^{-2}K\oplus LK\arrow[r,"\pi"] & \mco_D \arrow[r] & 0
\end{tikzcd}
\end{center}
where $\mco_D$ is a direct sum of skyscraper sheaves with stalks $\mbc$ at points in $D$, and $\pi$ is unique up to $\tx{Aut}_{\mco}\mco_D\cong (\mbc^\times)^{4g-4}$. For each $p\in D_r$, $\iota\rd{\atp{F}{p}} = \ker\,\pi$ for $\pi\neq 0$ lying in a line $\ell\in \mbp\rd{\rd{\atp{L^{-2}K_X}{p}\oplus \atp{LK_X}{p}}^\ast}$, corresponding to $b_p\in \atp{L^{\otimes 3}}{p}$ which gives the slope. This gives the map (dependent on $\ulD$) from $\mch_{L,q,\ulD}$ to $\mcf_{L,D_r}$. Again it is easy to verify that this is a bijection preserving respective equivalences.

Let $s_1'$, $s_2'$ be a local frame of $F$ over $(U;\zeta)$ around $p\in D_r$. Without loss of generality, suppose $s_1'$ generates the subbundle ker$\,\bet$ over $U$. Under local frames $(s_1'\wedge s_2')^{-1}$ of $L$ and $d\zeta$ of $K_X$, we have
\ben
\bet=\pmt{0 & \bet_2},\,\, \gam=\pmt{\gam_1 & \gam_2}^T
\een
where $\bet_2\gam_2=\zeta$. By definition of $D_r$, $\bet_2(\zeta=0)\neq 0$, therefore $\gam_2(\zeta=0)=0$ and $\gam_1(\zeta=0)\neq 0$. We may assume $\bet_2$ is nowhere vanishing on $U$. Therefore $\bet_2^{-1}\in \mco(U)$ and $\gam_2=\zeta \bet_2^{-1}$. Let $c$ be a cube root of the nonzero complex number $\bet_2^2\gam_1(\zeta=0)$ and set
\begin{align*}
&s_1=\rd{\sqrt{2}c^2\bet_2^{-1}+\ov{\sqrt{2} c\bet_2}(\bet_2^2\gam_1-c^3)}s_1'+\frac{\zeta}{\sqrt{2}c}s_2',\\
&s_2=\ov{\sqrt{2}c\zeta \bet_2\rd{\bet_2^2\gam_1-c^3}}s_1'+\ov{\sqrt{2}c}s_2'
\end{align*}
Under this local frame for $F$, $(s_1\wedge s_2)^{-1}$ for $L$ and $d\zeta$ for $K_X$, we have
\ben
\bet=(1/\sqrt{2})\pmt{\zeta & 1},\,\, \gam=(1/\sqrt{2})\pmt{1 & \zeta}^T\,.
\een
Under $s_1$, $s_2$ for $F$, $s_0=(s_1\wedge s_2)^{-1}$ for $L$ and $\sig_1=(s_0^{-2}d\zeta,0)$, $\sig_2=(0,s_0d\zeta)$ for $V$, by (\ref{eq:iotaform01}) $\iota$ takes the form in the statement. $\iota\rd{\atp{F}{p}}$ is the line spanned by $(-1,1)^T$ in $\atp{V}{p}=\atp{L^{-2}K_X}{p}\oplus \atp{LK_X}{p}$ under $\sig_1$, $\sig_2$. It follows straightforwardly that $b_p=1$ under the frame $s_0^3$, therefore the rest of the statement follows.
\end{proof}

It is worth mentioning that there is an equivalent interpretation of the Hecke parameters which is in a sense dual to the above construction. Given $(F,\bet,\gam)\in \mcm_{L,q,\ulD}$, $F$ may be viewed as extension of line bundles in two different ways
\begin{center}
\begin{tikzcd}
0 \arrow[r] & LK^{-1}\rd{D_\gam}\arrow[r,"(1)"] \arrow[d,"s"] & F \arrow[d,equal] \arrow[r] & L^{-2}K\rd{-D_\gam} \arrow[r] & 0 \\
0 & LK\rd{-D_\bet} \arrow[l] & F \arrow[l,"(2)"] & L^{-2}K^{-1}\rd{D_\bet} \arrow[l] & 0 \arrow[l]
\end{tikzcd}
\end{center}
where (1) (resp. (2)) are induced by $\gam$ (resp. $\bet$) and $s$ have simple zeros at $D_r$. The line bundles $LK^{-1}\rd{D_\gam}$ (resp. ker$\,\bet\cong L^{-2}K^{-1}\rd{D_\bet}$) are realized as subbundles img$\,\gam$ (resp. ker$\,\bet$) inside $F$. Their fibers meet precisely at $D_r$. Up to a global constant, the tuple $\ulb$ provides a recipe to match their fibers at $D_r$ inside of $F$. This picture of extensions is similar to the reconstruction of an SO$(2m+1)$ bundle from an Sp$(2m)$ bundle in \S 4.2 of \cite{Hit07} which inspired our description of SU(1,2) spectral data. 

We consider an example of SU(1,2) Higgs bundle with $q\in H^0(X,K_X^2)$ having simple zeros:
\be \label{eq:specialSU12}
\rd{F=K^{-1}\oplus K,\,\, \bet=\rd{1/\sqrt{2}}\pmt{q & 1},\,\, \gam=\rd{1/\sqrt{2}}\pmt{1 & q}^T}\,.
\ee
It is stable by Prop \ref{prop:su12stab2}. (\ref{eq:iotaform01}) holds globally and the corresponding Hecke modification $\iota:F \to V$ fits into a short exact sequence
\begin{center}
\begin{tikzcd}
0 \arrow[r] & F=K^{-1}\oplus K \arrow[r, "\iota"] & V=K\oplus K \arrow[r,"\pi"] & \mco_D \arrow[r] & 0
\end{tikzcd}
\end{center}
where 
\ben
\iota=\ov{\sqrt{2}}\pmt{q & -1 \\ q & 1}
\een
and $\pi_p(f_1d\zeta,f_2d\zeta)^T=f_1(p)+f_2(p)$ where $(U;\zeta)$ is a chart centered at $p$. The corresponding tuple $\ulb=(b_p)_{p\in D_r}$ is given by $b_p=1$.

For an SU(1,2) Higgs bundle $(F,\bet,\gam)$ and $p\in D_r$, Theorem \ref{thm:hecke} provides local frames under which Higgs field takes a standard form. We give below similar frames at points in $D_\bet$ and $D_\gam$. We view the sections of $F$ as sections of $V$ via $\iota$ from the correspondence in Theorem \ref{thm:hecke}.

\begin{defn} \label{def:heckeinducedframe2}
For $p_j\in D_\bet$, given frame $\cl{s_{0,j}}$ of $\atp{L}{\mbd_j}$ define a local frame of $V$ by $\sig_{1,j}=\rd{s_{0,j}^{-2}d\zeta_j,0}$, $\sig_{2,j}=\rd{0,s_{0,j}d\zeta_j}$ and set $s_{1,j}=\zeta_j\sig_{2,j}$, $s_{2,j}=-\sig_{1,j}$. These form a frame of $F$ over $\mbd_j$. Under these frames,
\ben
\bet=\pmt{\zeta_j & 0}d\zeta_j,\,\, \gam=\pmt{1 & 0}^Td\zeta_j\,.
\een
For $p_j\in D_\gam$, given frame $\cl{s_{0,j}}$ of $\atp{L}{\mbd_j}$ define a local frame $\sig_{1,j}$, $\sig_{2,j}$ of $V$ as above and set $s_{1,j}=\sig_{2,j}$, $s_{2,j}=-\zeta_j\sig_{1,j}$. These form a frame of $F$ over $\mbd_j$. Under these frames,
\ben
\bet=\pmt{1 & 0}d\zeta_j,\,\, \gam=\pmt{\zeta_j & 0}^Td\zeta_j\,.
\een
We will say that frames $\usig_j=\cl{\sig_{1,j},\sig_{2,j}}$ and $\uls_j=\cl{s_{1,j},s_{2,j}}$ are induced by $\cl{s_{0,j}}$.
\end{defn}

It will be useful to note that if $s_{0,j}$ induces frame $s_{1,j}$, $s_{2,j}$, then for $f$ a nowhere vanishing holomorphic function on $\mbd_j$, $fs_{0,j}$ induces frame $fs_{1,j}$, $f^{-2}s_{2,j}$. Furthermore, note that under $\sig_{1,j}$, $\sig_{2,j}$ on $\mbd_j^\times$ we have
\be \label{eq:stdform4}
\bet=\pmt{0 & \zeta_j^{-1}}d\zeta_j,\,\, \gam=\pmt{0 & \zeta_j^2}^Td\zeta_j
\ee

\section{Construction of approximate solution}
\label{sec:constapproxsoln}

Let $(F,\bet,\gam)$ be a stable SU(1,2) Higgs bundle, $t\ge 1$ and recall notations in Def \ref{def:disc}. In this part we construct an approximate solution to the Hitchin equation for $(F,t\bet,t\gam)$ for $t\gg 1$ by gluing an decoupled solution on $X-\coprod_j \mbd_j$ to local model solutions on $\mbd_j''$. In \S \ref{sec:decoupledsolution}, we characterize solutions to decoupled equation. In \S \ref{sec:localmodel}, we describe a family of local model solutions parametrized by $\lam\in (-1/4,1/4)$ as well as two other local model solutions which may be viewed as the cases $\lam=\pm 1/4$. In \S \ref{sec:asymp}, we apply estimates from \cite{Moc16} to study asymptotic properties of these local model solutions. In particular in \S \ref{sec:3.3.4}, we show that the next-to-leading order coefficients of the $\lam$-family of local solutions depends continuously on $\lam\in(-1/4,1/4)$. In \S \ref{sec:approxsoln}, we show that for $t\gg 1$, there is a family $\ulam(t)$ for which we can glue the local model solutions to a decoupled solution. A family of approximate solution $h_t^{\tx{app}}$ is constructed using this $t$-dependent tuple of parabolic weights $\ulam(t)$ and we show that it is very close to solving the Hitchin equation.

\subsection{Decoupled solution}
\label{sec:decoupledsolution}

A metric $h$ on $\atp{F}{X-D}$ solving the decoupled form of (\ref{eq:hitchin})
\be \label{eq:decoupled}
F_{\nabla_h}=0,\,\,\, \gam\wedge\gam^{\dagger_h}+\bet^{\dagger_h}\wedge\bet=0\,\,\tx{ on }X-D\,.
\ee
will be refered to as a decoupled solution. Let $\iota: F\to V=L^{-2}K\oplus LK$ be the Hecke modification associated to $(F,\bet,\gam)$ as in Theorem \ref{thm:hecke}.

\begin{lem} \label{lem:decoupledsoln}
$h_\infty$ on $\atp{F}{X-D}$ is a decoupled solution iff it has the form
\be \label{eq:formofdecoupledsolution}
h_\infty=\iota^\ast \rd{h_L^{-2}h_K\oplus h_L h_K}
\ee
where $h_K$ is the unique metric on $K$ such that $\vb{q}\equiv 1$ under induced metric on $K^2$ and $h_L$ is a harmonic metric adapted to a filtered line bundle $(L,\ulam)$ with $\deg L+\sum_j \lam_j=0$.
\end{lem}

\begin{proof}
Over $X-D$ we use $\iota$ (resp. $\det \iota$) to identify $F$ (resp. $\det F=L^{-1}$) with $V=L^{-2}K\oplus LK$ (resp. $\det V=L^{-1}K$) and write $h_\infty$ as $\iota^\ast h$ for some metric $h$ on $\atp{V}{X-D}$. On $X-D$ we have $\rd{F,\bet,\gam}\cong \rd{V,b=\pmt{0 & q^{-1}}, c=\pmt{0 & q^2}^T}$. Locally let $\sigma$ (resp. $\theta$) be nowhere vanishing holomorphic section of $L$ (resp. $K$) such that $\theta^2=q$. The sections $\sigma_1=(\sigma^{-2}\otimes \theta,0)$, $\sigma_2=(0,\sigma \otimes\theta)$ (resp. $\cl{\sigma\otimes\theta^{-2}}$) form a frame of $V$ (resp. $\det V^\ast$). With respect to these, the Higgs field is given by $b=(0,1)$, $c=(0,1)^T$. Let $H=h_{\usig}$. The second equation in (\ref{eq:decoupled}) implies $c c^\ast  H \rd{\det H}=\rd{\det H}^{-1} H^{-1} b^\ast b$. We have by a direct calculation
\be \label{eq:relationHermitianmetric}
H_{12}=H_{21}=0, \,\,\, H_{11}=H_{22}^{-2}\,.
\ee
Therefore the holomorphic direct sum $V=L^{-2}K\oplus LK$ is also $h$-orthogonal. We have $h=h_1\oplus h_2$ where $h_1$, $h_2$ are metrics on the two summands on $X-D$. Equivalently, these are induced by metrics $h_L$ (resp. $h_K$) on $L$ (resp. $K$). By flatness of $\nabla_h$, $\nabla_{h_L}$, $\nabla_{h_K}$ are both flat on $X-D$. We have $H_{11}=\vb{\sigma}_{h_L}^{-4} \vb{\theta}_{h_K}=\vb{\sigma}_{h_L}^{-4}\vb{q}_{h_K^2}^{1/2}$, $H_{22}=\vb{\sigma}_{h_L}^2 \vb{q}_{h_K}^{1/2}$. Therefore, $|q|_{h_K^2}=\rd{H_{11}H_{22}^2}^{2/3}\equiv 1$.

Let $f=\log\vb{\sig}_{h_L}$. Since $F_{\nabla_{h_L}}=0$ on $\mbd_j^\times$, $f$ is a positive harmonic function on $\mbd_j^\times$. By B\^ocher's theorem (see Theorem 3.9 \cite{Axl01}), it has the form $\lam_j \log\vb{\zeta_j}+f_0$ with $f_0$ harmonic on $\mbd_j$. The condition on the sum of $\lam_j$ now follows from Theorem \ref{thm:paralb}.
\end{proof}

Given $s_0$ a nowhere vanishing section of $L$ over $(\mbd,\zeta)$ centered at $p\in D$ and let $\sig_1$, $\sig_2$ be the induced frame as in the statement of Theorem \ref{thm:hecke}. A decoupled solution $h_\infty=\iota^\ast \rd{h_L^{-2}h_K\oplus h_Lh_K}$ under these frames are given by
\be \label{eq:decoupledsolutionlocalform}
h_{\infty}=\pmt{f^{-2}\rho^{-1} & \\ & f \rho^{-1}}
\ee
where $f=|s_0|_{h_L}^2$ and $\rho=\vb{\zeta}$.

\subsection{Local model}
\label{sec:localmodel}

In this part, we construct a family of local model solutions in $\mbd_j$ depending on a real parameter $\lam$ asymptotic to $h_\infty$ defined in Lemma \ref{lem:decoupledsoln} outside $\mbd_j$ with $\lam=\lam_j$. It will be defined by harmonic metric adapted to an unramifiedly good filtered Higgs bundle over $(\mbp^1,[1:0])$ for $p_j\in D_r$ and by an explicit formula using solution to a Painlev\'e III type equation for $p_j\in D_\bet$, $D_\gam$. It will be shown below that the stability condition of these filtered Higgs bundles are equivalent to
\ben
\begin{cases}
\lam_j=-1/4 & p_j\in D_\bet \\
\lam_j=1/4 & p_j\in D_\gam \\
\vb{\lam_j}<1/4 & p_j\in D_r\,.
\end{cases}
\een
A tuple $\ulam$ satisfying these conditions will be called admissible with respect to a partition $\ulD$ of $D$. We will first describe the set of these admissible weights. Let
\be \label{eq:defulD}
\msp_{\ulD}=\left\{\ulam\middle| \,\,\ulam\tx{ is admissible w.r.t }\ulD\right\}\subset \mbr^{4g-4}
\ee
Set
\ben
\msp_{d,g}=\coprod_{\ulD\tx{ stable}}\msp_{\ulD}\subset \mbr^{4g-4}\,.
\een
We have the following simple fact

\begin{lem}
Let $C$ be the cube in $\mbr^{4n}$ with vertices $V=\cl{(\pm 1/4,\ldots,\pm 1/4)}$ (i.e. $C$ is the convex hull of $V$), $m$ an integer with $|m|<n$. $f(x_1,\ldots,x_{4n})=\sum_j x_j$ and $H=f^{-1}(m)$. For $\bm{x}\in C\cap H$, let $d_1(\bm{x})=\#\cl{j|x_j=1/4}$, $d_2(\bm{x})=\#\cl{j|x_j=-1/4}$. Then we have that either (i) $d_1<2(n+m)$, $d_2<2(n-m)$ or (ii) $d_1=2(n+m)$, $d_2=2(n-m)$. The latter case corresponds to $\bm{x}$ being one of the vertices.
\end{lem}

With $n=g-1$, $m=-\deg L$ and $d_1=d_\gam$, $d_2=d_\bet$, it follows from Prop \ref{prop:su12stab2} and the above lemma that the closure $\overline{\msp_{d,g}}$ is the intersection $C\cap H$. It is easy to show that no edges of the cube $C$ passes through $H$, therefore $V\cap H$ is the extremal set of $C\cap H$. This is a compact convex polytope whose vertices $V\cap H=\overline{\msp_{d,g}}-\msp_{d,g}$ consists of $\msp_{\ulD}$ with $\ulD$ strictly polystable. The interior $\rd{\msp_{d,g}}^{\circ}=\msp_{\ulD}$ with $d_\bet=d_\gam=0$, whereas $\msp_{\ulD}$ for $\ulD$ stable and $d_\bet$ or $d_\gam>0$ are the interiors of the positive-dimensional faces of $\overline{\msp_{d,g}}$. 

Note that the face $\msp_{\ulD}$ has vertices consisting of tuples $\ulam$ such that $\lam_j=1/4$ for $p_j\in D_\bet$, $\lam_j=-1/4$ for $p_j\in D_\gam$ and $\#\cl{j\,|\, \lam_j=1/4,\,\, p_j\in D_r}=2(g-1-d)-d_\bet$, $\#\cl{j\,|\,\lam_j=-1/4,\,\, p_j\in D_r}=2(g-1+d)-d_\gam$. It follows easily that the barycenter is characterized by the conditions listed in Theorem \ref{thm:main}.

We begin by constructing the local model for points in $D_r$. Let $[x_0:x_1]$ be the homogeneous coordinate on $\mbp^1$ and $U_1=\cl{x_1\neq 0}$ (resp. $U_0=\cl{x_0\neq 0}$) the affine chart with coordinate $z=x_0/x_1$ (resp. $w=x_1/x_0$). Set
\be \label{eq:doublebranchedcoverdefn}
p: \mbp^1\to \mbp^1, \,\, [x_0:x_1]\lmapsto [x_0^2:x_1^2]
\ee
and $\zeta=z^2$ coordinate on affine chart $p(U_1)$ and let $\rho=\vb{\zeta}$. Denote $\mco'=\mco(\ast [1:0])$ and $\mco=\mco_{\mbp^1}$. Let $\mce=\rd{\mco'}^{\oplus 3}$ with free generators $\cl{e_0, e_1, e_2}$. Note $z\in \mco'$ whereas $w\notin \mco'$, and $dz\in \mco'\otimes_{\mco} \Omega_{\mbp^1}^1$ as in \S \ref{sec:filt}. Define
\be \label{eq:tilhig}
\theta: \mce\to \mce\otimes \Omega^1_{\mbp^1}, \,\,\theta=\sqrt{2}z\pmt{ & z^2 & 1\\ 1 & & \\ z^2 & & }dz\,.
\ee
Define $v_{j}=T_{k+1,j+1}e_{k}\in\mce$ for $j=0,1,2$ where
\be \label{eq:tilvfrm}
T=\pmt{\sqrt{2}z & -\sqrt{2}z & 0 \\
1 & 1 & 1 \\ z^2 & z^2 & -z^2}\,.
\ee
These generate $\mce$ over $U_0$, and $\theta v_0= 2z^2\,v_0\otimes dz$, $\theta v_1=-2z^2\,v_1\otimes dz$, and $\theta v_2=0$. 

Let $m_0$, $m_1$, $m_2\in\mbz$ and 
\ben
\mcf_{m_0,m_1,m_2}=\left\{ \sum_{j=0}^2 f_j v_j \middle| \tx{ord}_{[1:0]} f_j\ge -m_j \right\}\,.
\een
It is a locally free $\mco$-submodule of $\mce$. Sections $e_0$, $e_1$, $e_2$ (resp. $w^{-m_0} v_0$, $w^{-m_1} v_1$, $w^{-m_2} v_2$) trivialize $\mcf_{m_0,m_1,m_2}$ over $U_1$ (resp. $U_0$). The corresponding transition matrix is $G_{01}=T\cdot \tx{diag}\rd{z^{m_1},z^{m_2},z^{m_3}}$. We have $\det G_{01}=\tx{(const)}z^{m_1+m_2+m_3+3}$ is the transition function of the line bundle $\det \mcf_{m_0,m_1,m_2}$. Therefore, $\deg \mcf_{m_0,m_1,m_2}=3+m_1+m_2+m_3$. For $\ulc=(c_0,c_1,c_2)$, the $b$-family of $\mco$-submodules
\be \label{eq:defmcpbulcmce}
\mcp_{b}^{\ulc}\mce=\mcf_{\fl{b-c_0},\fl{b-c_1},\fl{b-c_2}}
\ee
defines a filtered bundle structure on $\mce$ with $\deg(\mcp_\ast^{\ulc}\mce)=3-c_0-c_1-c_2$. By Def \ref{def:fil}, $\rd{\mcp_b^{\ulc}\mce,\theta}$ is an unramifiedly good filtered Higgs bundle.

For $0\le i<j\le 2$, let $\mcl_j=\mco' v_j$ and
\ben
\mcs_{ij}=\left\{\rd{\frac{a_{-1}}{z}+g_1} v_i+\rd{-\frac{a_{-1}}{z}+g_2} v_j\middle| g_1,g_2\in\mco',\,\, a_{-1}\in\mbc\right\}
\een
with induced filtered bundle structure $\mcp_b\mcl_j=\mcl_j \cap \mcp_b^{\ulc}\mce$, $\mcp_b \mcs_{ij}=\mcs_{ij}\cap\mcp_b^{\ulc}\mce$. We have
\ben
\deg \mcp_\ast \mcl_j=-c_j,\,\, \deg \mcp_\ast \mcs_{ij}=1-c_i-c_j\,.
\een
These are the only nonzero proper $\theta$-invariant $\mco'$-submodules. Let
\be \label{eq:stabc}
S=\left\{ (c_0,c_1,c_2)\in\mbr^3 \middle| 
\begin{array}{l}
c_0+c_1+c_2=3, \,\, c_j>0\,\, \forall\, j,\\
c_i+c_j>1\,\, \forall\, i\neq j
\end{array}
\right\}\,.
\ee
We have that for all $t\in \mbr_+$, $\rd{\mcp_b^{\ulc}\mce,t\theta}$ is a stable with $\deg \mcp_\ast^{\ulc}\mce=0$ iff $\ulc\in S$.

\begin{lem} \label{lem:locmod}
For $t>0$, $-1/4<\lam<1/4$, there is a unique smooth function $K=K_{t,\lam}: \mbc_z\to i\mfu(2)$, positive definite, satisfying
\begin{itemize}
\item 
\be \label{eq:localform3z}
\pd_{\bar z} \rd{K^{-1}\pd_z K}=4|z|^2 t^2\rd{\gam_1\gam_1^\ast K \det K-\rd{\det K}^{-1} K^{-1}\bet_1^\ast\bet_1}
\ee
with 
\be
\bet_1=\rd{1/\sqrt{2}}\pmt{z^2 & 1},\,\,\, \gam_1=\rd{1/\sqrt{2}}\pmt{1 & z^2}^T\,.
\ee
\item Let
\be \label{eq:blockformtilK}
\widetilde K_{t,\lam}=\pmt{\det K_{t,\lam}^{-1} & \\ & K_{t,\lam}}\,.
\ee
The metric defined by $h(e_i,e_j)=\rd{\widetilde K_{t,\lam}}_{j+1,i+1}$ for $i,j=0,1,2$ is adapted to filtered bundle $\mcp_{\ast}^{\ulc}\mce$ with $\ulc=(1+2\lam,1+2\lam,1-4\lam)$.
\end{itemize}
\end{lem}

Let $\zeta=z^2$. We have at $z\neq 0$, 
\ben
\pd_{\bar \zeta}\rd{K^{-1}\pd_\zeta K} = (4|z|^2)^{-1}\pd_{\bar z}\rd{K^{-1}\pd_z K}=t^2\rd{\gam_1\gam_1^\ast K \det K-\rd{\det K}^{-1} K^{-1}\bet_1^\ast\bet_1}\,.
\een

\begin{proof}
We have $\ulc=(1+2\lam,1+2\lam,1-4\lam)\in S$ iff $-1/4<\lam<1/4$, in which case $\rd{\mcp_\ast^{\ulc}\mce,t\theta}$ is a stable unramifiedly good filtered Higgs bundle. By Theorem \ref{thm:biquard}, there is a harmonic metric $h=h_{t,\lam}$ on $\mce$ adapted to it, unique up to $\mbr_+$.

Let $g:\mce\to \mce$ be given by $e_0\mapsto - e_0$, $ e_j\mapsto  e_j$ for $j=1,2$. We have $g(v_0)=v_1$, $g(v_1)=v_0$, and $g(v_2)=v_2$. For any $v\in \mce$, $b\in\mbr$, we have that $v\in \mcp_b^{\ulc}\mce$ iff $g(v)\in \mcp_b^{\ulc}\mce$. Therefore, $g^\ast h$ is also adapted to filtered bundle $\mcp_\ast^{\ulc}\mce$. 

We have $g^{-1}\cdot F_{\nabla_h}\cdot g=F_{\nabla_{g^\ast h}}$, $g^{-1}[\theta\wedge\theta^{\ast_h}]g=[\theta^g\wedge \rd{\theta^g}^{\ast_{\rd{g^\ast h}}}]$ and $\theta^g=g^{-1}\theta g=-\theta$. Therefore
\ben
F_{\nabla_{g^\ast h}}+t^2\sq{\rd{-\theta}\wedge\rd{-\theta}^{\ast_{\rd{g^\ast h}}}}=0\,.
\een
It follows that $g^\ast h$ is a harmonic metric adapted to $\rd{\mcp_\ast^{\ulc}\mce,t \theta}$. Thus there is $c>0$ such that $g^\ast h=ch$. For $j=1,2$, we have $-h(e_0,e_j)=h(ge_0,ge_j)=g^\ast h(e_0,e_j)=ch(e_0,e_j)$. Therefore, $h(e_0,e_j)=0$, and $h=h_{t,\lam}$ has block diagonal form which we may write as
\be \label{eq:tildehine}
\rd{h_{t,\lam}}_{\ule}=\pmt{
\delta_{t,\lam}^{-1} & \\ & K_{t,\lam}
}
\ee
where $K_{t,\lam}$ is $2\times 2$ Hermitian matrix-valued function of $z\in \mbc$ and $\delta_{t,\lam}>0$. Note that $v_0\wedge v_1\wedge v_2=-4\sqrt{2}z^3 e_0\wedge e_1\wedge e_2$ and $c_0+c_1+c_2=3$. We have that $\vb{e_0\wedge e_1\wedge e_2}_{\det h}^2$ is harmonic on $U_0$ and bounded near $[1:0]\in\mbp^1$. Therefore, $\delta_{t,\lam}^{-1}\det K_{t,\lam}$ is a constant. We may normalize this to 1 and the resulting $K_{t,\lam}$, $\delta_{t,\lam}$ is uniquely determined. We have $\delta_{t,\lam}=\det K_{t,\lam}$. (\ref{eq:localform3z}) follows from (\ref{eq:localform}) and (\ref{eq:tilhig}).
\end{proof}

Let $B$, $C$ (resp. $K$) be function of $z\in \mbc$ valued in $3\times 3$ matrices (resp. $3\times 3$ positive-definite Hermitian matrices). Define
\begin{align*}
& R(K)=\pd_{\bar z}\rd{K^{-1}\pd_z K} \nonumber \\
& S(K,B,C)=C\, K \, \rd{\det K}-\rd{\det K}^{-1} K^{-1} B\,,
\end{align*}
and let $E_t(K,B,C)=R(K)-t^2 S(K,B,C)$. (\ref{eq:localform3z}) is equivalent to $E_t(K,\bet_1^\ast\bet_1,\gam_1\gam_1^\ast)=0$. Let $\rho_\phi: z\mapsto e^{i\phi}z$ and 
\be \label{eq:defgphi}
g_\phi=\tx{diag}\rd{e^{i\phi},e^{-i\phi}}\,.
\ee
It is easy to verify that $R(K)=\rho_\phi^\ast R\rd{\rho_{-\phi}^\ast K}$, $\rho_\phi^\ast B=g_\phi^\ast B g_\phi$, $\rho_\phi^\ast C=g_\phi^{-1} C\rd{g_\phi^\ast}^{-1}$ where $B=\bet_1\bet_1^\ast$, $C=\gam_1^\ast\gam_1$ with $\bet_1$, $\gam_1$ as in Lemma \ref{lem:locmod}. It follows that $S(K,B,C)=\rho_\phi^\ast g_\phi \, S\rd{g_\phi^\ast \rd{\rho_{-\phi}^\ast K}g_\phi,B,C}\, g_\phi^{-1}$. Thus we have
\ben
E_t(K,B,C)=\rho_\phi^\ast g_\phi E_t\rd{g_\phi^\ast \rd{\phi_{-\phi}K} g_\phi, B, C}g_\phi^{-1}\,.
\een
As a result, $z\mapsto g_\phi^\ast K_{t,\lam}(ze^{-i\phi})g_\phi$ satisfies (\ref{eq:localform3z}). Let 
\ben
\Gamma_\phi=\pmt{1 & \\ & g_\phi},\,\, T=\pmt{\sqrt{2}z & -\sqrt{2}z & 0 \\ 1 & 1 & 1\\ z^2 & z^2 & -z^2}\,.
\een
We have that $v_{j}=T_{k+1,j+1}e_{k}$ for $j,k=0,1,2$ and $\rho_\phi^\ast T=e^{i\phi}\Gamma_\phi T$. The metric on $\mce$ corresponding to $z\mapsto g_\phi^\ast K_{t,\lam}(ze^{-i\phi})g_\phi$ is given by
\ben
T^\ast \rho_{-\phi}^\ast \rd{\Gamma_\phi^\ast \pmt{\det K_{t,\lam}^{-1} & \\ & K_{t,\lam}} \Gamma_\phi} T=\rho_{-\phi}^\ast \rd{T^\ast \pmt{\det K_{t,\lam}^{-1} & \\ & K_{t,\lam}} T }
\een
under $\cl{v_0,v_1,v_2}$. Note that $\ord_{[1:0]}f=\ord_{[1:0]}\rho_\phi^\ast f$ and $\phi\in\mbr$. It follows that $v\in \mcp_b^{\ulc}\mce$ iff $\rho_\phi^\ast v\in \mcp_b^{\ulc}\mce$. Therefore, $z\mapsto g_\phi^\ast K_{t,\lam}(ze^{-i\phi})g_\phi$ is a smooth function $\mbc\to i\mfu(2)$ satisfying properties in Lemma \ref{lem:locmod}. It follows from the uniqueness that
\be \label{eq:rotsym}
K_{t,\lam}(z e^{i\phi})=g_\phi^\ast K_{t,\lam}(z)g_\phi\,.
\ee

For $t>0$, let $\eta_t: z\mapsto t^{1/3}z$ and
\be \label{eq:defGamt}
\Gamma_t=\tx{diag}\rd{t^{1/3},t^{-1/3}}\,.
\ee
We have that 
\ben
R(K)=t^{2/3}\eta_t^\ast R\rd{\eta_{t^{-1}}^\ast K}=t^{2/3}\eta_t^\ast \Gamma_t R\rd{\Gamma_t^\ast \cdot \eta_{t^{-1}}^\ast K \cdot \Gamma_t}\Gamma_t^{-1}
\een
and that 
\ben
\eta_t^\ast B=t^{4/3}\Gamma_t^\ast B\Gamma_t,\,\, \eta_t^\ast C=t^{4/3}\Gamma_{t^{-1}}C\Gamma_{t^{-1}}^\ast\,.
\een
It follows that $S(K,B,C)=t^{-4/3}\eta_t^\ast S\rd{\eta_{t^{-1}}^\ast H,\Gamma_t B\Gamma_t^\ast,\Gamma_{t^{-1}}^\ast C\Gamma_{t^{-1}}}$. By direct calculation, we have
\ben
E_t(K,B,C)=t^{2/3}\eta_t^\ast \Gamma_t\cdot E_1\rd{\Gamma_t^\ast\rd{\eta_{t^{-1}}^\ast K}\Gamma_t,B,C}\cdot \Gamma_t^{-1}\,.
\een
Thus $z\mapsto \Gamma_t^\ast K_{t,\lam}(t^{-1/3}z)\Gamma_t$ solves (\ref{eq:localform3z}). On the other hand, 
\ben
\eta_t^\ast T=t^{1/3}\pmt{1 & \\ & \Gamma_t^{-1}}T\,.
\een
It follows that the metric given by $z\mapsto \Gamma_t^\ast K_{t,\lam}\rd{t^{-1/3}z}\Gamma_t$ on $\mce$ has a local form under $\cl{v_0,v_1,v_2}$ given by
\ben
T^\ast \pmt{1 & \\ & \Gamma_t^\ast} \cdot \eta_{-t}^\ast\pmt{\det K_{t,\lam}^{-1} & \\ & K_{t,\lam}}\cdot \pmt{1 & \\ & \Gamma_t}T=t^{2/3}\eta_{-t}^\ast\rd{T^\ast\pmt{\det K_{t,\lam}^{-1} & \\ & K_{t,\lam}}T}\,.
\een
For $f_j\in\mco'$, we have that $\sum_j f_jv_j \in \mcp_b^{\ulc}\mce$ iff $\sum_j t^{2/3}\eta_{-t}^\ast f_jv_j \in \mcp_b^{\ulc}\mce$. Therefore, $z\mapsto \Gamma_t^\ast K_{t,\lam}(t^{-1/3}z)\Gamma_t$ gives a harmonic metric adapted to $\mcp_\ast^{\ulc}\mce$. By the uniqueness in Lemma \ref{lem:locmod}, we have
\be \label{eq:tscalinglaw}
K_{1,\lam}(t^{1/3}z)=\Gamma_t^\ast K_{t,\lam}(z)\Gamma_t\,.
\ee

It follows from (\ref{eq:rotsym}) and (\ref{eq:tscalinglaw}) that $K_{t,\lam}$ has the form
\ben
K_{t,\lam}(re^{i\theta})=\pmt{
t^{-2/3}f_1\rd{t^{1/3}r} & f_3\rd{t^{1/3}r}e^{-2i\theta}\\
\overline{f_3\rd{t^{1/3}r}}e^{2i\theta} & t^{-2/3}f_2\rd{t^{1/3}r}
}\,,
\een
where $f_j$ for $j=1,2,3$ are smooth functions on $\mbr_{>0}$, and each entry is a smooth function on $\mbc$. It is clear that there is a function $H_{t,\lam}: \mbc\to i\mfu(2)$ such that $K_{t,\lam}(z)=H_{t,\lam}(z^2)$ for all $z\in\mbc$. By the continuity and positive-definiteness of $K_{t,\lam}$, in a bounded neighborhood of the origin, $\det H_{t,\lam}$ is bounded away from zero. Furthermore, it is not hard to see that we have $f_j\in C^\infty([0,\infty))$ with $f_j'(0)=0$ and $f_3(0)=0$. A direct calculation shows $\pd_z F$, $\pd_{\bar z}F=O(\vb{z})$ as $z\to 0$. As $\zeta=z^2$ we have $\pd_\zeta=(2z)^{-1}\pd_z$ and $\pd_{\bar\zeta}=(2\bar z)^{-1}\pd_{\bar z}$. The entries of $\pd_\zeta H_{t,\lam}$, $\pd_{\bar\zeta} H_{t,\lam}$ are bounded at the origin. Thus the same is true for $\pd_\zeta \rd{\det H_{t,\lam}^{-1}}$ and $\pd_{\bar \zeta} \rd{\det H_{t,\lam}^{-1}}$. Let
\ben
M:=\pmt{
\det H_{t,\lam}^{-1} & \\ & H_{t,\lam}
}\,.
\een
We have by (\ref{eq:localform3z}), $M$ solves
\ben
\pd_{\bar\zeta}\rd{M^{-1}\pd_\zeta M}-[\varphi,M^{-1}\varphi^\ast M]=0\,,
\een
on $\mbc^\times$ with
\ben
\varphi=\ov{\sqrt{2}}\pmt{0 & \zeta & 1 \\ 1 & 0 & 0 \\ \zeta & 0 & 0}\,,
\een
and for any $0\in \Omega$ bounded, $M\in C^\infty(\Omega^\times)\cap L_1^\infty(\Omega)$ where $\Omega^\times=\Omega-\cl{0}$.

Let
\begin{align*}
D: M\lmapsto & M\rd{\pd_{\bar\zeta}\rd{M^{-1}\pd_\zeta M}-[\varphi,M^{-1}\varphi^\ast M]}\\
 & = \ov{4}\Delta M-\mct_1(M)-\mct_2(M)
\end{align*}
where $\Delta=4\pd_{\bar\zeta}\pd_\zeta$ is the Laplacian and
\begin{align*}
& \mct_1(M)=\rd{\pd_{\bar \zeta }M}M^{-1}\rd{\pd_\zeta M} , \\
& \mct_2(M)=M[\varphi,M^{-1}\varphi^\ast M]\,.
\end{align*}
A bootstrap argument can be used to improve our knowledge on the regularity of $M$ at 0. Consider two intermediate open neighborhoods: $0\in \Omega''\Subset \Omega'\Subset \Omega$. By elliptic regularity estimate (see, e.g. \cite{DK90} Appendix III), there is some $C>0$ such that
\ben
\dbv{M}_{L_2^2(\Omega'')}\le C\rd{\dbv{\ov{4}\Delta(M)}_{L^2(\Omega')}+\dbv{M}_{L^2(\Omega'')}}\,.
\een
To show $M\in L_2^2$ near the origin, it suffices to bound $\dbv{\ov{4}\Delta(M)}_{L^2}\le \dbv{\mct_1(M)}_{L^2}+\dbv{\mct_2(M)}_{L^2}$. Using the elliptic regularity of the Cauchy-Riemann operator $\pd_{\bar\zeta}$, we have $C'>0$ such that
\begin{align*}
& \dbv{M^{-1}\pd_\zeta M}_{L_1^2(\Omega')}\le C'\rd{ \dbv{\pd_{\bar\zeta}\rd{M^{-1}\pd_\zeta M}}_{L^2(\Omega)}+\dbv{M^{-1}\pd_\zeta M}_{L^2(\Omega')} } \\
& =C'\rd{ \dbv{\sq{\varphi,M^{-1}\varphi^\ast M}}_{L^2(\Omega)}+\dbv{M^{-1}\pd_\zeta M}_{L^2(\Omega')} }\,.
\end{align*}
From the above $M^{\pm 1}$ is bounded at 0 and $\dbv{\pd_\zeta M}_{L^2(\Omega)}<\infty$. By the Sobolev embedding $L_1^2\subset L^4$, $\dbv{\pd_\zeta M}_{L^4(\Omega')}$ is bounded. This implies the same bound on $\dbv{\pd_{\bar\zeta} M}_{L^4(\Omega')}$ since $\pd_{\bar\zeta}M=\rd{\pd_\zeta M^\ast}^\ast=\rd{\pd_\zeta M}^\ast$. Thus there is $C''>0$ such that
\ben
\dbv{\mct_1(M)}_{L^2(\Omega')}\le C''\dbv{M}_{L_1^4(\Omega')}^2<\infty\,.
\een
By $L^\infty$ bound on $M^{\pm 1}$ and $\varphi$, $\dbv{\mct_2(M)}_{L^2(\Omega')}<\infty$. Therefore $M\in L_2^2(\Omega'')$.

Suppose $M\in L_{k+1}^2(\Omega_k)$ for some $k\ge 1$ and $0\in \Omega_k'\Subset \Omega_k$. By elliptic estimate, there is $C>0$ such that
\ben
\dbv{M}_{L_{k+2}^2(\Omega_k')}\le C\rd{\dbv{\mct_1(M)}_{L_k^2(\Omega_k)}+\dbv{\mct_2(M)}_{L_k^2(\Omega_k)}+\dbv{M}_{L^2(\Omega_k')}}\,.
\een
Use $\nabla$ to denote either $\pd_\zeta$ or $\pd_{\bar\zeta}$, and $\nabla^\ell$ will denote $\pd_\zeta^{\ell_1}\pd_{\bar\zeta}^{\ell_2}$ with $\ell_1+\ell_2=\ell$. Then $\nabla^k \mct_1(M)$ is a sum of terms of the form
\ben
M^{n_1}\rd{\nabla^{\ell_1}M}M^{n_2}\ldots M^{n_m}\rd{\nabla^{\ell_m}M}M^{n_{m+1}} \tag{$\ast$}
\een
where $n_j\in \mbz$ and $\ell_j\ge 0$, $\sum_{j=1}^m \ell_j=k+2$, and for at least two indices $j$, we have $\ell_j\ge 1$. For the terms with $\ell_j\le k$ for all $j$, take $p_1,\ldots,p_m>2$ with $1/2=\sum_j p_j^{-1}$ we have:
\ben
\dbv{(\ast)}_{L^2(\Omega_k)}\le \prod_{j=1}^m \dbv{\nabla^{\ell_j}M}_{L^{p_j}}\le \prod_{j=1}^m \dbv{M}_{L_{\ell_j}^{p_j}}\,.
\een
We have $\ell_j-\frac{2}{p_j}\le \ell_j\le k$. Therefore, by Sobolev embedding theorems (e.g. \cite{DK90} Appendix IV), there is $C'>0$ such that $\dbv{M}_{L_{\ell_j}^{p_j}}\le C'\dbv{M}_{L_{k+1}^2(\Omega_k)}<\infty$. The only terms left are of the form $\rd{\nabla^{k+1}M}M^{-1}\rd{\nabla M}$ or $\rd{\nabla M}M^{-1}\rd{\nabla^{k+1}M}$. These have finite $L^2(\Omega_k)$-norms since both $M^{-1}$ and $\nabla M$ have bounded entries at 0 and $M\in L_{k+1}^2(\Omega_k)$ by assumption. It follows by induction on $k$ and Sobolev embedding theorems that $M\in C^\infty$. 

Summarize the above discussion, we have proven the following.

\begin{prop} \label{prop:locmodsym}
For $-1/4<\lam<1/4$, there is a smooth function $H_{t,\lam}:\mbc\to i\mfu(2)$ scuh that $K_{t,\lam}(z)=H_{t,\lam}(z^2)$ where $K_{t,\lam}$ is the function in Lemma \ref{lem:locmod}. In particular,
\be \label{eq:Htpol}
H_{t,\lam}\rd{\rho e^{i\psi}}=\pmt{\rho f_{1,\lam}\rd{t^{2/3}\rho} & f_{3,\lam}\rd{t^{2/3}\rho}e^{-i\psi} \\
\overline{f_{3,\lam}\rd{t^{2/3}\rho}}e^{i\psi} & \rho^{-1} f_{2,\lam}\rd{t^{2/3}\rho}}\,,
\ee
where $f_{j,\lam}$ is a smooth function on $\mbr_+$ and $H=H_{t,\lam}$ satisfies
\be \label{eq:localform3}
\pd_{\bar\zeta} \rd{H^{-1}\pd_{\zeta}H}=t^2\rd{\gam_0\gam_0^\ast H \det H-\rd{\det H}^{-1} H^{-1}\bet_0^\ast\bet_0}
\ee
with
\ben
\bet_0=\rd{1/\sqrt{2}}\pmt{\zeta & 1},\,\,\gam_0=\rd{1/\sqrt{2}}\pmt{1 & \zeta}^T \,.
\een
\end{prop}

Let $(F,\bet,\gam)$ be a stable SU(1,2) Higgs bundle. Recall notations in Def \ref{def:disc} and $s_0$ be a trivializing section of $L$ over the coordinate disk $\rd{\mbd;\zeta}$ centered at $p\in D_r$ and let $\uls=\cl{s_1,s_2}$ (resp. $\usig=\cl{\sig_1,\sig_2}$) be frames of $F$ (resp. $V$) over $D$ induced by $s_0$ as in the statement of Theorem \ref{thm:hecke}. We fix notation for Hermitian metrics corresponding to the Hermitian-matrix valued functions above.

\begin{defn} \label{def:metrics}
Denote by $h_{t,\lam}$ the metric on $F$ over $D$ with $h_{t,\lam}=\rd{H_{t,\lam}}_{\uls}$ where $H_{t,\lam}$ is as in Prop \ref{prop:locmodsym}. Denote by $\widetilde h_{t,\lam}$ the metric on $\mce$ given by $\widetilde h_{t,\lam}\rd{e_{i+1},e_{k+1}}=\rd{\widetilde K_{t,\lam}}_{ki}$ where $i,k=0,1,2$ and $\widetilde K_{t,\lam}$ is as in Lemma \ref{lem:locmod}.
\end{defn}

We have by a direct calculation
\be \label{eq:htlaminetafrm}
\rd{h_{t,\lam}}_{\usig}=t^{2/3}M_\lam\rd{t^{2/3}\rho}
\ee
where
\be \label{eq:Mlam2H1lam}
M_\lam=\rd{S^\ast}^{-1}H_{1,\lam}S^{-1},\,\, S=\ov{\sqrt{2}}\pmt{\zeta & -1 \\ \zeta & 1}
\ee
and in particular we have the following explicit expression.
\be \label{eq:defMlamfirst}
M_{\lam}\rd{\rho}=\ov{2\rho}\pmt{
f_{1,\lam}+f_{2,\lam}-f_{3,\lam}-\overline{f_{3,\lam}} & f_{1,\lam}-f_{2,\lam}+f_{3,\lam}-\overline{f_{3,\lam}} \\
f_{1,\lam}-f_{2,\lam}-f_{3,\lam}+\overline{f_{3,\lam}} & f_{1,\lam}+f_{2,\lam}+f_{3,\lam}+\overline{f_{3,\lam}}
}\,.
\ee
Thus $H_{t,\lam}$ is determined by a Hermitian-matrix-valued radial function $M_\lam$. Alternatively, $M_\lam$ is a submatrix of $\rd{\widetilde h_{1,\lam}}_{\widetilde{\usig}}$ where the frame on $\mce$ is given by $\widetilde\sig_j$ for $j=0,1,2$ where $\widetilde\sig_{j}=\sum_{\ell=0}^2 \rd{T_{e\sig}}_{\ell+1, j+1}e_{\ell}$ where
\be \label{eq:Tesig}
T_{e\sig}=\ov{\sqrt{2}}\pmt{\sqrt{2} & 0 & 0 \\ 0 & z^{-2} & z^{-2} \\ 0 & -1 & 1}\,.
\ee
We have
\be \label{eq:defMlam}
\rd{\widetilde h_{t,\lam}}_{\widetilde{\usig}}=\pmt{
t^{-4/3}\vb{z}^{-4}m_\lam\rd{t^{2/3}\vb{z}^2}^{-1} & \\
 & t^{2/3}M_\lam\rd{t^{2/3}\vb{z}^2}
}\,,
\ee
where $m_\lam=\det M_\lam$. For later use, we have
\begin{align}
& \widetilde h_{1,\lam}\rd{v_2,v_2}=2|z|^4 \rd{M_\lam}_{11}(|z|^2),\,\, \widetilde h_{1,\lam}\rd{v_0,v_2}=2|z|^4 \rd{M_\lam}_{12}(|z|^2) \nonumber \\
& \widetilde h_{1,\lam}\rd{v_0,v_0}=2|z|^2 \rd{|z|^2\rd{M_\lam}_{22}(|z|^2)+ |z|^{-4}m_\lam(|z|^2)^{-1}}, \nonumber \\
& \widetilde h_{1,\lam}\rd{v_0,v_1}=2|z|^2 \rd{|z|^2\rd{M_\lam}_{22}(|z|^2)- |z|^{-4}m_\lam(|z|^2)^{-1}}\,. \label{eq:tilhv}
\end{align}
Note that $\lam$ in $H_{t,\lam}$ is characterized by $\det H_{t,\lam}(\zeta)=O(\vb{\zeta}^c)$ as $\zeta\to 0$ for $c<-2\lam$ and $\det H_{t,\lam}(\zeta)\neq O(\vb{\zeta}^c)$ as $\zeta\to 0$ for $c>-2 \lam$.

For local model around points in $D_\bet$, $D_\gam$, replace $\theta$ in (\ref{eq:tilhig}) by
\ben
\theta=2z\pmt{0 & 1 & 0 \\ z^2 & 0 & 0 \\ 0 & 0 & 0}dz, \,\,\tx{ resp. }\,\, \theta=2z\pmt{0 & z^2 & 0 \\ 1 & 0 & 0 \\ 0 & 0 & 0}dz\,
\een
and let $v_j\in \mce$ with $j=0,1,2$ be such that $\theta\,v_0=\sqrt{2}v_0\otimes dz$, $\theta\,v_1=-\sqrt{2}zv_1\otimes dz$, and $\theta\, v_2=0$. Same procedures as above defines an unramifiedly good filtered Higgs bundle $\rd{\mcp_\ast^{\ulc}\mce,\theta}=\rd{\mcp_\ast^{(c_0,c_1)}\mce',\theta_1}\oplus \rd{\mco'e_2,0}$ where $\mce'=\mco'e_0\oplus \mco' e_1$. By a similar discussion the rank-two summand is stable if $c_0,c_1>0$ and $c_0+c_1=1$. By arguments similar to the proof of Lemma \ref{lem:locmod}, the corresponding metric has a local form (\ref{eq:blockformtilK}) iff $c_0=c_1=1/2$. In fact in these cases, there is a more explicit formula similar to the local fiducial solution of \cite{MSWW16} the local model, in terms of the a smooth solution of an ODE of Painlev\'e type III:
\be \label{eq:painleveiii}
\rd{x\pd_x}^2\psi=\frac{x^2}{2}\sinh\rd{2\psi}\,.
\ee
Following \cite{MTW77} (with a simplified proof in \cite{Wid00}), there is the unique solution $\psi$ with following properties
\be \label{eq:painleveiiibc}
\begin{cases}
\psi(x)\sim -\log\rd{x^{1/3}\sum_j a_j x^{4j/3}}, & x\to 0\\
\psi'(x)<0, & x>0\\
\psi(x)=O\rd{x^{-1/2}e^{-cx}}, & x\to+\infty
\end{cases}\,.
\ee
There exists a harmonic metric adapted to strictly polystable good filtered Higgs bundle $(\mcp_\ast^{\ulc}\mce,\theta)$ is unique up to positive scalar on each summand. The result below therefore follows by a direct calculation verifying that $H_t$ indeed satisfies (\ref{eq:localform3}) with respective $\bet_0$, $\gam_0$ and that $\psi_P$ defined below is smooth at the origin.

\begin{prop} \label{prop:locmodpainleve}
Let $\cl{s_{1,j},s_{2,j}}$ be a frame of $F$ over $(\mbd_j;\zeta_j)$ centered at $p_j$ as in Def \ref{def:heckeinducedframe2}. Then
\ben
H_t(\zeta_j)=\begin{cases}
\tx{diag}\rd{c^{-1}\rho^{1/2}e^{\psi_P},c^2} & p_j\in D_\bet\\
\tx{diag}\rd{c\rho^{-1/2}e^{-\psi_P},c^2} & p_j\in D_\gam
\end{cases}\,,
\een
where $c\in\mbc^\times$, and
\be \label{eq:defpsiP}
\psi_P=\psi\rd{\frac{8}{3}t\rho^{3/2}}
\ee
(where $\psi$ is the unique solution above and $\rho=\vb{\zeta}$) solves (\ref{eq:localform3}) with 
\ben
\bet_0=\begin{cases}
\pmt{\zeta & 0} & p_j\in D_\bet \\
\pmt{1 & 0} & p_j\in D_\gam
\end{cases},\,\,\, 
\gam_0=\begin{cases}
\pmt{1 & 0}^T & p_j\in D_\bet\\
\pmt{\zeta & 0}^T & p_j\in D_\gam
\end{cases}
\een
and $\log \det H_t(\zeta)=-2\lam_j \log|\zeta|+O(1)$ as $\zeta\to 0$ with $\lam=-1/4$ for $p_j\in D_\bet$ (resp. $\lam=-1/4$ for $p_j\in D_\gam$). Furthermore, the solution satisfying the above asymptotic estimate is unique up to the choice of $c\in\mbc^\times$.
\end{prop}

In the following we will refer to $\psi_P$ as the Painlev\'e function. The estimate below follows easily from (\ref{eq:painleveiiibc}) and will be useful later.

\begin{lem} \label{lem:painleveasymp}
There are $C_1$, $c_1>0$ and $x_1>0$ such that for all $x\ge x_1$, and $k=0,1,2$,
\ben
\vb{\pd_x^k\psi(x)}\le C_1e^{-c_1 x}\,.
\een
\end{lem}

Note that the Painlev\'e function also provide explicit formula for $H_{t,\lam}$ in Prop \ref{prop:locmodsym} with $\lam=0$: 
\be \label{eq:specialHt0}
H_{t,0}\rd{\rho e^{-\psi}}=\tx{diag}\rd{\rho e^{2\psi_P}, \rho^{-1} e^{-2\psi_P}}\,.
\ee

\subsection{Asymptotics of local models}
\label{sec:asymp}
In order for the gluing construction to work well it will be necessary to know the behavior of function $M_\lam$ at large radius with bounds uniform in $\lam$ in some interval. We fix a small enough $\delta>0$ and let $I=(-1/4+\delta,1/4-\delta)$. Constants in an inequality will be said to be independent of $\lam$ if it holds for all $\lam\in I$. (These could still depend on choise of $\delta$.) In particular, we will show:

\begin{prop} \label{prop:asympsum}
There are $\rho_0>1$ and $C_1,C_2>0$ such that for all $\lam\in I$ and $\rho\ge \rho_0$, we have
\begin{align*}
& \vb{M_\lam(\rho)-M_{\infty,\lam}(\rho)},\,\,\, \vb{\pd_\rho \rd{M_\lam(\rho)-M_{\infty,\lam}(\rho)}},\,\,\, \vb{\rho^{-1}\pd_\rho\rd{\rho\pd_\rho\rd{M_\lam(\rho)-M_{\infty,\lam}(\rho)}}},  \\
& \vb{\rho^2m_\lam(\rho)-\mu_\lam(\rho)},\,\,\, \vb{\pd_\rho \rd{\rho^2m_\lam(\rho)-\mu_\lam(\rho)}},\,\,\, \vb{\rho^{-1}\pd_\rho\rd{\rho\pd_\rho\rd{\rho^2 m_\lam(\rho)-\mu_\lam(\rho)}}}  \\
& \le C_1e^{-C_2\rho}
\end{align*}
where $M_{\infty,\lam}=\tx{diag}\rd{\rho^{-1}\mu_\lam^2,\rho^{-1}\mu_\lam^{-1}}$ with
\be \label{eq:defmulam}
\mu_\lam(\rho)=4c_\lam^{-1}\rho^{-2\lam}\,.
\ee
where $\lam\mapsto c_\lam$ is continuous.
\end{prop}

The coefficient $c_\lam$ will play a central role in the gluing construction. Constants in the estimates below which does depend on $\lam\in I$ will often depend on it through a continuous function of $c_\lam$. Furthermore, the above gives asymptotics of the metric $\widetilde h_{1,\lam}$ defined in Def \ref{def:metrics}. Therefore by (\ref{eq:specialHt0}), we have $c_0=4$.

The estimates will build upon relevant results in \cite{Moc16}. We first review these results. Then we use the local form of $\widetilde h_{1,\lam}$ under $\widetilde{\usig}$ as well as Prop \ref{prop:sharpenadaptedness} to get the estimates in value. The gradient estimate relies on results similar to Lemma 3.12 and 3.13 in \cite{Moc16}. The continuity of $c$ and uniformity of constants $C_1$, $C_2$ on $I$ are similar to Prop 3.15 \cite{Moc16} using an identity from \cite{Sim88}.

\subsubsection{Asymptotic estimates of Mochizuki}

$\rd{E,\bar\pd_E,\theta,h}$ is called a harmonic bundle if $\bar\pd_E\theta=0$ and $h$ solves the Hitchin equation $F_{\nabla_h}+\sq{\theta\wedge \theta^{\ast_h}}=0$. The following combines Propositions 2.1, 2.3, 2.10, 2.12 and Cor 2.6 of \cite{Moc16}.

\begin{thm} \label{thm:mochizuki}
Let $\rd{E,\bar\pd_E,\theta,h}$ be a harmonic bundle of rank $N$ with holomorphic decomposition $\rd{E,\bar\pd_E,\theta}=\bigoplus_{j=1}^N \rd{E_j,\bar\pd_{E_j},\theta_j}$ over a chart $(\mbd;z)\subset X$ with $\mbd=\cl{|z|<R}$ where $E_j$ is a line bundle over $\mbd$. Let $\theta_j=f_j dz$ and $M>0$ such that $\vb{f_j}<M\tx{ on }\mbd$. Suppose $d=\min\left\{\vb{f_i-f_j}: \,\, i\neq j\right\}\ge 1$. Then for any $0<r<R$ let $\mbd'=\cl{|z|<r}$, there are $C,c>0$ depending only on $R,r,M/d,N$ such that
\begin{align}
& \vb{F_{\nabla_h}}_{h,g_\mbc}, \vb{F_{\nabla_{h_j}}}_{h_j,g_\mbc}, \vb{\pd_h \pi_j}_{h,g_\mbc}\le  Ce^{-c d} \tx{ on }\mbd' \label{eq:curvbd} \\
& h(s_i,s_j)\le Ce^{-c d}\vb{s_i}_h \vb{s_j}_h\tx{ on }\mbd'\,\,\tx{ for }s_i\in E_i,\,s_j\in E_j,\,i\neq j \label{eq:hsisjest}
\end{align}
where $h_j$ is the restriction of $h$ to $E_j$, $\pi_j$ the holomorphic projection onto $E_j$, and $g_\mbc$ the Euclidean metric on $\mbd$.
\end{thm}

The following bounds the covariant derivative by the gradient of the norm. It is a slight generalization of Lemma 3.12 in \cite{Moc16}.

\begin{lem} \label{lem:pdhvj}
Notation as in Theorem \ref{thm:mochizuki} with $N\ge 2$ and suppose $C e^{-cd}<10^{-2}N^{-3/2}$. Then for any nonzero section $v_j\in E_j$, we have on $\mbd'$
\ben
\vb{\pd_h v_j}_{h,g_\mbc}\le 10 \vb{v_j}_h \rd{\vb{\pd_\zeta \log|v_j|_h^2}+C e^{-c d}}\,.
\een
\end{lem}

\begin{proof}
By assumption $Ce^{-c d}<1/(2N)$. Let $w=\sum_{j=1}^N c_j v_j$ be a section with $v_j\in E_j$, we get by (\ref{eq:hsisjest}) $\ov{2}\sum_{i=1}^N |c_i|^2\,\, |v_i|_h^2 \le |w|_h^2 \le \frac{3}{2}\sum_{i=1}^N |c_i|^2\,\, |v_i|_h^2$.

For a $\sigma$ define operator $\pd$ by $\pd_h \sigma=\rd{\pd \sigma}d\zeta$ where $\pd_h$ is the (1,0)-part of the Chern connection of the induced metric. Let $\pd v=\sum_\ell c_\ell v_\ell$, then
\begin{align}
&\sum_i \vb{v_i}_h^{-1}\vb{h\rd{\pd v_j,v_i}}\ge \sum_i \vb{v_i}_h^{-1}\rd{
\vb{c_i}\vb{v_i}_h^2-\vb{\sum_{\ell\neq i}c_\ell h\rd{v_\ell,v_i}}}\nonumber \\
&\ge \sum_i \vb{c_i}\vb{v_i}_h-\sum_i \sum_{\ell\neq i}Ce^{-c d}\vb{c_\ell}\vb{v_\ell}_h\ge \ov{2}\sum_i \vb{c_i}\vb{v_i}_h\ge \ov{\sqrt{6}}\vb{\pd_h v_j}_h\,. \label{eq:pdhvjeq1}
\end{align}
The $i=j$ term in $\sum_i \vb{v_i}_h^{-1}\vb{h\rd{\pd v_j,v_i}}$ is $\vb{v_j}_h\vb{\pd_\zeta \log\vb{v_j}_h^2}$ whereas $i\neq j$ terms satisfy, by (\ref{eq:curvbd}),
\begin{align}
& \vb{v_i}_h^{-1}\vb{h\rd{\pd v_j,v_i}} \le \vb{v_i}_h^{-1}\rd{\vb{h\rd{\rd{\pd \pi_j}v_j,v_i}}+\vb{h\rd{\pi_j\rd{\pd v_j},v_i}}}\nonumber \\
& \le Ce^{-c d}\rd{\vb{\pd \pi_j}_h\vb{v_j}_h +\vb{\pi_j\rd{\pd v_j}}_h}\le \rd{Ce^{-c d}}^2\vb{v_j}_h + Ce^{-c d}\vb{c_j}\vb{v_j}_h \label{eq:pdhvjeq2}
\end{align}
We have 
\be \label{eq:pdhvjeq3}
\vb{c_j}\vb{v_j}_h\le \sum_i \vb{c_i}\vb{v_i}_h\le \sqrt{N}\vb{\pd v_j}_h\,. 
\ee
We have $N\sqrt{N}C e^{-c d}\le 1/100$. Thus by (\ref{eq:pdhvjeq1}), (\ref{eq:pdhvjeq2})
\ben
\vb{\pd v_j}_h\le \sqrt{6} \rd{\vb{v_j}_h\vb{\pd_\zeta \log\vb{v_j}_h^2} + \ov{2} C e^{-cd} \vb{v_j}_h+\ov{100}\vb{\pd v_j}_h}\,.
\een
\end{proof} 

\subsubsection{Asymptotics of metrics}
\label{sec:asympval}
Recall definition of coordinates $z$ and $w=z^{-1}$ from \S \ref{sec:localmodel}. Consider the harmonic bundle $\rd{\mce, \theta,\widetilde h_{1,\lam}}$ on $\cl{|w|< 1}\subset \mbp^1$ with $\widetilde h_{1,\lam}$ defined in Def \ref{def:metrics}. We have $\rd{\mce, \theta}=\bigoplus_{j=0}^2 \rd{\mco'v_j, f_j dz}$ where $\mco'=\mco_{\mbp^1}\rd{\ast[1:0]}$ and $\rd{f_1,f_2,f_3}=2z \rd{z,-z,0}$. In the following we denote $\widetilde h_{1,\lam}$ by $\widetilde h$. An estimate is said to hold for $\rho\gg 1$ if there is $\rho_0>1$ such that it holds for $\rho\ge \rho_0$. Note that the distance between eigenvalues
\ben
d=\min_{i\neq j}\rd{|f_i-f_j|}=2|z|^2\ge 1
\een
are bounded below, and $L=\max_j |f_j|=2|z|^2$. In particular $L/d=1$ is independent of $|z|$. We have for $h_j=\left.\widetilde h\right|_{\mcl_j}$, $\vb{\rd{F_{\nabla_{h_j}}}}_{h_j,g_\mbc}=\vb{\pd_z\pd_{\bar z}\log \vb{v_j}_{\widetilde h}}$. By Theorem \ref{thm:mochizuki}, there are $C,c>0$ independent of $\lam$ such that $\vb{\pd_\zeta\pd_{\bar \zeta}\log \vb{v_j}_{\widetilde h}}\le C e^{-c|\zeta|}$ and
\be \label{eq:offdiagexpdecay}
\vb{\widetilde h\rd{ v_i, v_j}}\le C \vb{ v_i}_{\widetilde h} \vb{ v_j}_{\widetilde h} e^{-c \rho},\,\, 0\le i\neq j\le 2\,.
\ee
Furthermore, by Theorem \ref{thm:mochizuki} $|F_{\nabla_{\widetilde h}}|_{\widetilde h}^2$ is bounded by an $L^p$ function near $[1:0]$. By Prop \ref{prop:sharpenadaptedness}, there are constants $0<C_1<C_2$, which could depend on $\lam$, such that
\begin{align}
&C_1\rho^{1+2\lam}\le \vb{v_j}_{\widetilde h}^2\le C_2\rho^{1+2\lam},\,\, j=0,1\nonumber \\
&C_1\rho^{1-4\lam}\le \vb{v_2}_{\widetilde h}^2\le C_2\rho^{1-4\lam} \label{eq:parabweightasympform}\,.
\end{align}
We can improve the norm estimate and give a gradient estimate using the following result, which follows easily from Proposition 3.10 in \cite{Moc16}.

\begin{lem} \label{lem:convcurvbd}
Let $f>0$ be a smooth radial function on $\cl{|\zeta|\ge 1}$ such that there are $C_1, c_1>0$ with
\be \label{eq:curvaturecondition01}
\vb{\pd_{\zeta}\pd_{\bar\zeta}\log f}\le C_1 e^{-c_1 |\zeta|^\ell}
\ee
for some $\ell\in\mbn$ and $\alpha\in\mbr$, $C_2,C_2'>0$ such that
\be \label{eq:powlawinlem}
C_2|\zeta|^\alpha\le f\le C_2'|\zeta|^\alpha\,.
\ee
Then there is $b=b(\alp,f)>0$ and $C_3,c_3>0$ independent of $\alpha$ such that
\ben
\vb{\log f-\log b\vb{\zeta}^\alpha}, \vb{\zeta\pd_{\zeta}\log f-\alpha/2}\le C_3 e^{-c_3|\zeta|^\ell}\,.
\een
\end{lem}

Applying the above and (\ref{eq:tilhv}), there are $C_3, c_3>0$ independent of $\lam$ and $c^{(1)}_\lam$, $c^{(2)}_\lam$ such that
\begin{align}
&\vb{
\log \rd{2\rho^2 \rd{M_\lam}_{11}(\rho)}-\log\rd{c^{(2)}_\lam \rho^{1-4\lam}}
},\,\, \vb{
\zeta\pd_\zeta \log \rd{2\rho^2 \rd{M_\lam}_{11}(\rho)} - \rd{1-4\lam}/2}, \nonumber \\
&\vb{
\log\rd{2\rho^2 \rd{M_\lam}_{22}(\rho)+2\rho^{-1}m_\lam(\rho)^{-1}}-\log\rd{c^{(1)}_\lam\rho^{1+2\lam}}
},\nonumber \\
&\vb{
\zeta\pd_\zeta \log \rd{2\rho^2 \rd{M_\lam}_{22}(\rho)+2\rho^{-1}m_\lam(\rho)^{-1}} - \rd{1+2\lam}/2}\le C_3 e^{-c_3 \rho}\,. \label{eq:origest}
\end{align}
For $(i,j)=(0,1)$ resp. (0,2) in (\ref{eq:offdiagexpdecay}), by (\ref{eq:parabweightasympform}), there are $C_4$, $c_4>0$ independent of $\lam\in I$ such that for $\rho \gg 1$,
\begin{align*}
&\vb{\rd{M_\lam}_{22}(\rho)-\rho^{-3}m_\lam\rd{\rho}^{-1}}\le C_4 c^{(1)}_\lam \rho^{2\lam-1} e^{-c_4 \rho}\\
&\vb{\rd{M_\lam}_{12}(\rho)}\le C_4 \rd{c^{(1)}_\lam c^{(2)}_\lam}^{1/2}\rho^{-\lam-1} e^{-c_4 \rho}\,.
\end{align*}

Using $|x-1|\le \tx{(const)}\vb{\log x}$ for small $|x|$, by (\ref{eq:origest}) there are $C_5$ and $c_5>0$ depending continuously on $c^{(1)}(\lam)$ and $c^{(2)}(\lam)$, such that for $\rho\gg 1$,
\ben
\vb{M_\lam(\rho)-\tx{diag}\rd{c^{(2)}_\lam\rho^{-4\lam-1}/2, c^{(1)}_\lam\rho^{2\lam-1}/4}}, \,\, \vb{m_\lam(\rho)-4\rd{c^{(1)}_\lam}^{-1}\rho^{-2\lam-2}}\le C_5e^{-c_5\rho}\,. \label{eq:C0bd}
\een
In the following, denote $c_\lam=c^{(1)}_\lam$. By $m_\lam=\det M_\lam$ we have $c^{(2)}_\lam=32 c_\lam^{-2}$. For $M_{\infty,\lam}$ and $\mu_\lam(\rho)$ in Prop \ref{prop:asympsum}, there are $C_5'$, $c_5'>0$ depending continuously on $c_\lam$ such that for $\rho\gg 1$,
\be \label{eq:C0bd1}
\vb{M_\lam(\rho)-M_{\infty,\lam}(\rho)}, \,\, \vb{\rho^2m_\lam(\rho)-\mu_\lam(\rho)}\le C_5'e^{-c_5'\rho}\,.
\ee
In the rest of \S \ref{sec:asympval}, constants in an inequalities not in lemmas will be depending continuously on $c_\lam$ unless stated otherwise. 

Since $M_{\infty,\lam}$ is diagonal with diagonal entries of the form $c\rho^{\alpha}$, $\alpha\in\mbr$ and $c>0$, it follows that for $\rho\gg 1$, $|M_\lam M_{\infty,\lam}^{-1}-I|=O\rd{e^{-\tx{(const)}\rho}}$. For $\rho \gg 1$ we also have $C_5''$, $c_5''>0$ such that $|M_{\infty,\lam}M_\lam^{-1}-I|=\vb{\sum_{k=1}^\infty\rd{I-M_{\infty,\lam}^{-1}M_\lam}^k}\le C_5''e^{-c_5''\rho}$. With a similar but easier argument for $m_\lam$, we have that for $\rho\gg 1$ there are $C_5'''$, $c_5'''>0$ such that
\be \label{eq:C0bd2}
\vb{M_\lam(\rho)^{-1}-M_{\infty,\lam}(\rho)^{-1}},\,\, \vb{\rho^2 m_\lam(\rho)^{-1}-\mu_\lam(\rho)^{-1}}\le C_5''' e^{-c_5'''\rho}
\ee

\subsubsection{Asymptotics of the first and second derivative of the metric}
\label{sec:asympderv}

We will need the following elementary lemma.

\begin{lem} \label{lem:logderv}
Suppose $f$ is a function on $\cl{\rho>1}$ and $C_j>0$, $j=1,2,3$, $c_1$, $c_2$ and $c_3>0$ satisfy $\vb{\rho\pd_\rho \rd{\log f-\log g}}\le C_1 e^{-c_1 \rho}$ where $g(\rho)=C_2\rho^{c_2}$ and $\vb{f-g}\le C_3 e^{-c_3 \rho}$. Then there are $C_4,c_4,\rho_0>0$ depending on the previous constants that for $\rho\gg 1$, 
\ben
\vb{\rho\pd_\rho\rd{f-g}}\le C_4 e^{-c_4\rho}\,.
\een
\end{lem}

Note that $\zeta\pd_\zeta=\ov{2}\rho\pd_\rho$ on radial functions and $\rho\pd_\rho\log\rho^\alpha=\alpha$. By Lemma \ref{lem:logderv} and (\ref{eq:origest}), we have $C_6,c_6>0$ with
\begin{align}
& \vb{
\rho\pd_\rho \rd{\rd{M_\lam}_{22}(\rho)+\rho^{-3}m_\lam(\rho)^{-1}-c_\lam\rho^{-1+2\lam}/2}}\,, \nonumber \\
& \vb{
\rho\pd_\rho \rd{\rd{M_\lam}_{11}(\rho)-16c_\lam^{-2}\rho^{-1-4\lam}}}\le C_6 e^{-c_6\rho}\,. \label{eq:Mlamdervest1}
\end{align}

In order to bound the first derivatives of off-diagonal elements, we apply Lemma \ref{lem:convcurvbd} to $\vb{v_j}_{\widetilde h}$, and given that $|\lam|<1/4$, we have for $|\zeta|\gg 1$ there is $C_7>0$ independent of $\lam\in I$ with $\vb{\pd_\zeta\log\vb{v_j}_{\widetilde h}^2}\le C_7\vb{\zeta}^{-1}$. This provides a bound for the first term on the right-hand side of the inequality in the statement of Lemma \ref{lem:pdhvj}. The other term is bounded by $e^{-(\tx{const})\rho}$ and is certainly bounded by the first term for $\rho\gg 1$. It follows that for $\rho\gg 1$, there is $C_8>0$ independent of $\lam\in I$ with
\ben
\vb{\pd_{\widetilde h}  v_j}_{\widetilde h_{1,\lam},g_\mbc}\le C_8\rho^{-1}\vb{ v_j}_{\widetilde h}
\een
Since $\vb{v_j}_{\widetilde h}$ satisfies (\ref{eq:powlawinlem}), by (\ref{eq:pdhvjeq2}), (\ref{eq:pdhvjeq3}) in proof of Lemma \ref{lem:pdhvj} and the above there are $C_9$ and $c_9>0$ such that $\vb{z\pd_z \widetilde h\rd{ v_j, v_i}}=\vb{z\widetilde h \rd{\pd_{\widetilde h} v_j, v_i}}\le C_9 e^{-c_9 |z|^2}$ for $i\neq j$. For $(i,j)=(0,1)$ and $(0,2)$, by (\ref{eq:C0bd1}) there are $C_9'$ and $c_9'>0$ such that
\be \label{eq:offdiagdervbd1}
\vb{\rho\pd_\rho \rd{M_\lam}_{12}(\rho)},\,\, \vb{\rho\pd_\rho\rd{\rd{M_\lam}_{22}(\rho)-\rho^{-3}m_\lam(\rho)^{-1}}}\le C_9'e^{-c_9'\rho}\,.
\ee
At this point, we showed that there are $C_9'',c_9''>0$ such that
\be \label{eq:C1bd1}
\vb{\pd_\rho\rd{M_\lam(\rho)-M_{\infty,\lam}(\rho)}}<C_9''e^{-c_9''\rho}\,.
\ee
We have
\begin{align*}
&\tr\rd{M_\lam^{-1}\pd_\rho\rd{M_\lam-M_{\infty,\lam}}}+\tr\rd{\rd{M_\lam^{-1}-M_{\infty,\lam}^{-1}}\pd_\rho M_{\infty,\lam}} \\
&=m_\lam^{-1}\pd_\rho\rd{m_\lam-\rho^{-2}\mu_\lam}+\rd{\rho^2\mu_\lam^{-1}-m_\lam^{-1}}\pd_\rho\rd{\rho^{-2}\mu_\lam}\,.
\end{align*}
By (\ref{eq:C0bd2}), (\ref{eq:C1bd1}) there are $C_9'''$, $c_9'''>0$ such that for $\rho\gg 1$,
\ben
\vb{\pd_\rho\rd{\rho^2m_\lam(\rho)-\mu_\lam(\rho)}}\le C_9'''e^{-c_9'''\rho}\,.
\een

We next bound the difference in the second derivatives. For $i\neq j$, we have
\begin{align*}
&\vb{\pd_{\bar z}\pd_z \widetilde h\rd{ v_i, v_j}}\le \vb{h\rd{\pd_h v_i,\pd_h v_j}}+\vb{h\rd{F_{\nabla_h}v_i,v_j}} \\
&\le \vb{\widetilde h\rd{\rd{\pd_{\widetilde h}\pi_i v_i},\rd{\pd_{\widetilde h}\pi_j v_j}}}+\vb{\widetilde h\rd{\rd{\pd_{\widetilde h}\pi_i}v_i, \pi_j\rd{\pd_{\widetilde h}v_j}}} \\
&+\vb{\widetilde h\rd{\pi_i\rd{\pd_{\widetilde h}v_i},\rd{\pd_{\widetilde h}\pi_j}v_j}}+\vb{\widetilde h\rd{\pi_i\rd{\pd_{\widetilde h}v_i},\pi_j\rd{\pd_{\widetilde h}v_j}}}\,,
\end{align*}
where $\pi_j$ is the holomorphic projection to $\mco'v_j$. Let $\pd_{\widetilde h}v_i=\sum_\ell c_\ell v_\ell$ we have by the proof of Lemma \ref{lem:pdhvj}, 
\ben
\vb{\pi_i\rd{\pd_{\widetilde h}v_i}}_{\widetilde h}=\vb{c_i}\vb{v_i}_{\widetilde h}\le \sqrt{2}\vb{\pd_{\widetilde h}v_i}_{\widetilde h}\,.
\een
By (\ref{eq:curvbd}), (\ref{eq:hsisjest}) there are $C_{10}$ and $c_{10}>0$ such that for $\vb{z}\gg 1$,
\ben
\vb{\pd_{\bar z}\pd_z \widetilde h\rd{ v_i, v_j}}\le C_{10} e^{-c_{10}|z|^2}\,.
\een

For $(i,j)=(0,1)$, $(0,2)$, by (\ref{eq:offdiagdervbd1}) there are $C_{11}$, $c_{11}>0$ such that for $\rho\gg 1$,
\be \label{eq:offdiagdervbd2}
\vb{\pd_{\bar\zeta}\pd_\zeta\rd{M_\lam}_{12}(\rho)}, \,\, \vb{\pd_{\bar\zeta}\pd_\zeta\rd{\rd{M_\lam}_{22}(\rho)-\rho^{-3}m_\lam(\rho)^{-1}}} \le C_{11}e^{-c_{11}\rho}\,.
\ee
By the estimates of the curvature along $\mco'v_j$ for $j=0,2$ in (\ref{eq:curvbd}) there are $C_{12}$, $c_{12}>0$ independent of $\lam$ such that for $\rho\gg 1$,
\ben
\vb{\pd_{\bar\zeta}\pd_\zeta\log\rd{\frac{\rd{M_\lam}_{22}(\rho)+\rho^{-3}m_\lam(\rho)^{-1}}{c_\lam\rho^{2\lam-1}/2}}},\,\, \vb{\pd_{\bar\zeta}\pd_\zeta\log\rd{\frac{\rd{M_\lam}_{11}(\rho)}{16 c_\lam^{-2}\rho^{-4\lam-1}}}}\le C_{12}e^{-c_{12}\rho}\,.
\een
We will need the following elementary lemma.

\begin{lem} \label{lem:logderv2}
Let $f(\rho)$ on $\rho>1$, $g(\rho)=C_1\rho^{c_1}$, $C_2$, $c_2>0$ be such that 
\ben
\vb{f-g}, \vb{\pd_\rho\rd{f-g}}\le C_2 e^{-c_2\rho}\,.
\een
Let $D:=\rho^{-1}\pd_\rho\rd{\rho\pd_\rho}$ and suppose $\vb{D\rd{\log f-\log g}}\le C_2 e^{-c_2\rho}$. Then there are $C_3$, $c_3>0$ depending on previous constants such that for $\rho\gg 1$,
\ben
\vb{D\rd{f-g}}\le C_3 e^{-c_3\rho}\,.
\een
\end{lem}

By (\ref{eq:C0bd1}), (\ref{eq:Mlamdervest1}) and Lemma \ref{lem:logderv2}, there are $C_{13}$ and $c_{13}>0$ such that for $\rho\gg 1$,
\begin{align*}
& \vb{\pd_{\bar\zeta}\pd_\zeta\rd{\rd{M_\lam}_{22}(\rho)+\rho^{-3}m_\lam(\rho)^{-1}-c_\lam\rho^{-1+2\lam}/2}}\,, \\
& \vb{\pd_{\bar\zeta}\pd_\zeta\rd{\rd{M_\lam}_{11}(\rho)-16 c_{\lam}^{-2}\rho^{-1-4\lam}}}\le C_{13}e^{-c_{13}\rho}\,.
\end{align*}
With (\ref{eq:offdiagdervbd2}), there are $C_{13}'$ and $c_{13}'>0$ such that for $\rho\gg 1$,
\ben
\vb{\pd_{\bar\zeta}\pd_\zeta \rd{\rd{M_\lam}_{22}(\rho)-c_\lam\rho^{1-2\lam}/4}}\le C_{13}' e^{-c_{13}'\rho}\,.
\een
Combining (\ref{eq:offdiagdervbd2}),there are $C_{13}'$ and $c_{13}'>0$ such that 
\ben
\vb{\pd_{\bar\zeta}\pd_\zeta\rd{\rd{M_\lam}_{22}(\rho)-c_\lam\rho^{1-2\lam}/4}}\le C_{13}'e^{-c_{13}'\rho}\,.
\een
At this point, we have proved the following:

\begin{prop} \label{prop:asympsum0}
All the estimates in Prop \ref{prop:asympsum} hold with $C_1$, $C_2>0$ depending continuously on $c_\lam$.
\end{prop}

\subsubsection{Uniform boundedness of family of Hermitian-Yang-Mills metrics}
\label{sec:3.3.4}

In this part we finish proof of Prop \ref{prop:asympsum}, which follows from Prop \ref{prop:asympsum0} and:

\begin{prop} \label{prop:clamcont}
The function $\lam\mapsto c_\lam$ is continuous.
\end{prop}

This is the continuity of the next-to-leading order in $\lam\mapsto \widetilde h_{1,\lam}$. We build on the proof of Prop 3.15 in \cite{Moc16} for the rank-two case. Note that the conclusion will follow once we can show that for any $\lam_0\in (-1/4,1/4)$ on $\cl{\rho\ge 1}$, $\rho^{2(\lam_0-\lam)m_{\lam_0}m_\lam^{-1}-1}\to 1$ uniformly as $\lam\to \lam_0$. We have
\begin{align*}
&\vb{c_{\lam_0}c_\lam^{-1}-1}\le \vb{\rho^{2(\lam-\lam_0}+2\mu_{\lam_0}^{-1}\rd{\rho^{-2}\mu_\lam-m_\lam}}+\vb{\rho^2\mu_{\lam_0}^{-1}m_{\lam_0}\rd{\rho^{2(\lam-\lam_0}m_\lam m_{\lam_0}^{-1}-1}} \\
&+\vb{\rho^2\mu_{\lam_0}^{-1}m_{\lam_0}-1}
\end{align*}
Denote the three terms by I, II and III. By Prop \ref{prop:asympsum0}, there are $C_{\tx{I}}(\lam)$, $c_{\tx{I}}(\lam)$, $C_{\tx{II}}$, $C_{\tx{III}}$, $c_{\tx{III}}>0$ such that $\tx{I}\le C_{\tx{I}}(\lam)e^{-c_{\tx{I}}\rho}$, $\tx{III}\le C_{\tx{III}}e^{-c_{\tx{III}}\rho}$ and that $\vb{\rho^2\mu_{\lam_0}^2m_{\lam_0}}\le C_{\tx{II}}$. Therefore by first taking $\rho\to\infty$, we have that $c_\lam \to c_{\lam_0}$ as $\lam\to\lam_0$.

Let $I_0\Subset (-1/4,1/4)$ containing $\lam_0$. All constants in inequalities below will depend only on $I_0$ unless stated otherwise. Consider an auxiliary family $\widetilde h^0_\lam$ for $\lam\in I_0$ such that (1) $\widetilde h^0_{\lam_0}=\widetilde h_{1,\lam_0}$, (2) $\det \widetilde h^0_\lam=\det \widetilde h_{1,\lam}\equiv 1$ the standard metric on trivial bundle, (3) $\widetilde h^0_{\lam}\to \widetilde h_{1,\lam_0}$ uniformly on any $K\Subset \mbp^1-\cl{[1:0]}$ in $C^\infty$ sense as $\lam\to \lam_0$, and (4) on $\cl{|z|\ge 1}$
\ben
\widetilde H^0_\lam=\Gamma_\lam^\ast \widetilde H_{1,\lam_0}\Gamma_\lam
\een
where $\widetilde H_{1,\lam}=\rd{\widetilde h_{1,\lam}}_{\widetilde{\usig}}$, $\usig$ is defined in (\ref{eq:Tesig}), $\widetilde H^0_\lam=\rd{\widetilde h^0_\lam}_{\widetilde{\usig}}$ and 
\ben
\Gamma_\lam=\tx{diag}\rd{|z|^{2(\lam-\lam_0)},|z|^{4(\lam_0-\lam)},|z|^{2(\lam-\lam_0)}}\,.
\een
i.e. $\widetilde h_\lam^0$ have the same leading order asymptotic behavior as $h_{1,\lam}$ near $[1:0]$. Fix $g_{\mbp^1}$ a K\"ahler metric on $\mbp^1$.

\begin{lem}
Let $\mbf_\lam=F_{\nabla_{h}}+[\theta,\theta^{\ast_h}]$ where $h=\widetilde h_\lam^0$. There is $C>0$ independent of $\lam$ such that
\ben
\vb{\mbf_\lam}_{h,g_{\mbp^1}}\le C
\een
\end{lem}

\begin{proof}
Let $G_\lam\in \tx{Aut}(\mce)$ such that $\rd{G_\lam}_{\widetilde{\usig}}=\Gamma_\lam$. Under $\widetilde{\usig}$ for $|z|\ge 1$, $[\theta,\theta_{\widetilde h^0_\lam}^\ast]=\Gamma_\lam^{-1}[\theta,\theta_{\widetilde h_{1,\lam_0}}^\ast]\Gamma_\lam$. For $\psi\in\tx{End}(\mce)$ with $\Psi=\psi_{\widetilde{\usig}}$, we have $\vb{\psi}_{\widetilde h^0_\lam}^2=\tr\rd{\Psi\rd{\widetilde H_\lam^0}^{-1}\Psi^\ast \widetilde H^0_\lam}=\vb{\Gamma_\lam \psi\Gamma_\lam^{-1}}_{\widetilde h_{1,\lam_0}}^2$. 
By Theorem \ref{thm:mochizuki}, there is $C_1>0$ such that for $|z|\gg 1$, 
\ben
\vb{\sq{\theta,\theta_{\widetilde h^0_\lam}^\ast}}_{\widetilde h^0_\lam, g_{\mbp^1}}^2=\vb{\Gamma_\lam^{-1}\sq{\theta,\theta_{\widetilde h_{1,\lam_0}}^\ast}\Gamma_\lam}_{\widetilde h^0_\lam,g_{\mbp^1}}^2=\vb{\sq{\theta,\theta_{\widetilde h_{1,\lam_0}}^\ast}}_{\widetilde h_{1,\lam_0},g_{\mbp^1}}\le C_1\,.
\een
We have 
\begin{align*}
&\rd{F_{\nabla_{\widetilde h^0_\lam}}}_{\widetilde{\usig}}=\pd_{\bar z}\rd{\rd{\widetilde H^0_\lam}^{-1}\pd_z \widetilde H^0_\lam}d\bar z\wedge dz \\
&=\pd_{\bar z}\rd{\Gamma_\lam^{-1}\rd{\widetilde H_{1,\lam_0}}^{-1}\rd{\Gamma_\lam^\ast}^{-1}\pd_z\rd{\Gamma_\lam^\ast}\widetilde H_{1,\lam_0}\Gamma_\lam}d\bar z\wedge dz \\
&+\pd_{\bar z}\rd{\Gamma_\lam^{-1}\rd{\widetilde H_{1,\lam_0}}^{-1}\pd_z\rd{\widetilde H_{1,\lam_0}}\Gamma_\lam}d\bar z\wedge dz+\pd_{\bar z}\rd{\Gamma_\lam^{-1}\pd_z \Gamma_\lam}d\bar z\wedge dz\,.
\end{align*}

Note that $\pd_{\bar z}\rd{\Gamma_\lam^{-1}\pd_z \Gamma_\lam}=0$ and using 
\ben
\pd(A^{-1}BA)=[A^{-1}BA,A^{-1}\pd A]+A^{-1}\rd{\pd B}A
\een 
where $\pd=\pd_z$ or $\pd_{\bar z}$, we get
\begin{align*}
&\rd{F_{\nabla_{\widetilde h^0_\lam}}}_{\widetilde{\usig}}=\Gamma_\lam^{-1}\pd_{\bar z}\rd{\widetilde H_{1,\lam_0}^{-1}\pd_z \widetilde H_{1,\lam_0}}\Gamma_\lam+[\Gamma_\lam^{-1}\widetilde H_{1,\lam_0}^{-1}\rd{\pd_z \widetilde H_{1,\lam_0}}\Gamma_\lam,\Gamma_\lam^{-1}\rd{\pd_{\bar z}\Gamma_\lam}] \\
&+[\Gamma_\lam^{-1}\widetilde H_{1,\lam_0}^{-1}\rd{\Gamma_\lam^{\ast}\pd_z \Gamma_\lam^{\ast}}\widetilde H_{1,\lam_0}\Gamma_\lam,\Gamma_\lam^{-1}\widetilde H_{1,\lam_0}^{-1}\pd_{\bar z}\rd{\widetilde H_{1,\lam_0}\Gamma_\lam}]\,.
\end{align*}
Recall we have
\ben
\widetilde H_{1,\lam_0}=\pmt{|z|^{-4}m_{\lam_0}\rd{|z|^2}^{-1} & \\ & M_{\lam_0}\rd{|z|^2}}\,.
\een
To bound the commutator terms it suffice to bound the off-diagonal elements. By Prop \ref{prop:asympsum0}, off-diagonal elements in $\widetilde H_{1,\lam_0}^{-1}\rd{\pd_z \widetilde H_{1,\lam_0}}$ and $\Gamma_\lam^{-1}\widetilde H_{1,\lam_0}^{-1}\rd{\pd_z \widetilde H_{1,\lam_0}}\Gamma_\lam$ are bounded by $C_2e^{-c_2\rho}$ with $C_2$, $c_2>0$. On the other hand, $|\pd_{\bar z}\rd{\widetilde H_{1,\lam_0}^{-1}\pd_z \widetilde H_{1,\lam_0}}|\le C_3e^{-c_3\rho}$ with $C_3$, $c_3>0$.
\end{proof}

Let $k_\lam$ be given by $\widetilde h_{1,\lam}=\widetilde h^0_\lam k_\lam$. We have that $k_\lam$ is self-adjoint (with respect to both metrics) and positive-definite. It follows that $\vb{k_\lam}\le C \tr{k_\lam^2}^{1/2}\le C\tr\rd{k_\lam}$ for some constant $C>0$. Since $\det k_\lam=1$, $\tr\, k_\lam\ge 3$ where equality happens iff $k_\lam=\tx{Id}$. By design of $\widetilde h^0_\lam$ we have for all $\lam$, $\sup_{\mbp^1} |k_\lam|<\infty$.

The rest of the proof is almost identical to that of \cite[Prop 3.17]{Moc16}. We still include it here for completeness. The following identity is from Prop 3.1 in \cite{Sim88} expressing the difference between curvatures of two metric connections associated with $\bar\pd+\theta$ for two different metrics:
\be \label{eq:weitzenbock}
i\Lambda_{g_{\mbp^1}}\bar\pd \pd\, \tr\, k_\lam=i\Lambda_{g_{\mbp^1}} \tr\rd{k_\lam \mbf_\lam}-\vb{\rd{\rd{\bar\pd+\theta}k_\lam}k_\lam^{-1/2}}_{\widetilde h^0_\lam,g_{\mbp^1}}^2\,.
\ee

We will also need the following result from Prop 2.1 in \cite{Sim88}.

\begin{lem} \label{lem:sim2.1}
Let $(X,g)$ be a compact K\"ahler manifold and $b\in L^p(X,g)$ for some $p>\dim X$. Then there is $C=C(b)>0$ such that for any $f>0$ bounded on $X$ with $\Delta_g f\le b$, we have $\sup_X |f|\le C\dbv{f}_{L^1(X,g)}$.
\end{lem}

\begin{lem}
$k_\lam\to 1$ uniformly on $\mbp^1$
\end{lem}

\begin{proof}
Denote by $\dbv{\cdot}_{L^q}=\dbv{\cdot}_{\widetilde h^0_\lam, g_{\mbp^1},L^q}$ and fix $p>2$. As remarked above, it suffices to show that $\tr\, k_\lam\to 3$ uniformly. For fixed $\lam$, $\mbf_\lam$ and $k_\lam$ are bounded on $\mbp^1$, therefore $|\Lambda_{g_{\mbp^1}}\tr\rd{F_\lam k_\lam}|\in L^p$ where the $L^p$-norm may depend on $\lam$. Set $s_\lam=k_\lam/\dbv{k_\lam}$ where $\dbv{k_\lam}=\dbv{k_\lam}_{L^p}$. By (\ref{eq:weitzenbock})
\be \label{eq:weitzenbockineq}
i\Lambda_{g_{\mbp^1}}\bar\pd \pd \,\tr\, s_\lam\le \vb{\Lambda_{g_{\mbp^1}}\tr\rd{\mbf_\lam s_\lam}}\,.
\ee
The right-hand side now has a uniformly-in-$\lam$ $L^p$ bound. There are $\phi_\lam\in L_2^p\subset C^0$, $C_6>0$ with $\sup|\phi_\lam|\le C_6$, and $C_7>0$ such that
\ben
i\Lambda_{g_{\mbp^1}}\bar\pd \pd \rd{\tr\, s_\lam-\phi_\lam}\le C_7.
\een
By Lemma \ref{lem:sim2.1} there are $C_8$, $C_8'>0$ such that
\ben
\sup_{\mbp^1}|s_\lam|\le C_8' |s_\lam|_{L^1,\widetilde h^0_\lam,g_{\mbp^1}}\le C_8\,.
\een
We have that $\mbf_\lam\to 0$ uniformly on compact sets as $\lam\to \lam_0$. By (\ref{eq:weitzenbock}), (\ref{eq:weitzenbockineq}), and the uniform boundedness of $\sup|s_\lam|$ we have as $\lam\to\lam_0$
\ben
\dbv{\rd{\bar\pd+\theta}s_\lam}_{L^2}\to 0\,.
\een
It follows that $\dbv{\bar\pd s_\lam}_{L^2}=\dbv{\pd_h s_\lam}_{L^2}\to 0$ where $h=\widetilde h^0_\lam$. Therefore $\dbv{s_\lam}_{L_1^2}$ is bounded uniformly in $\lam$. Choose $s_{\lam_n}\rightharpoonup s_\infty$ weakly in $L_1^2$. We have $\dbv{s_\infty}_{L^p}=1>0$, $s_\infty$ is a nonzero holomorphic endomorphism commuting with $\theta$. By stability of $\rd{\mcp_\ast^{\ulc}\widetilde \mce,\theta}$, $s_\infty$ must be a nonzero multiple of the identity. Since $\det k_\lam=1$, we have $\lim_{n\to\infty}\det s_{\lam_n}=\lim_{n\to\infty}\dbv{k_{\lam_n}}^{-3}=\det s_\infty\neq 0$.

Suppose $\dbv{k_\lam}$ is not uniformly bounded below near $\lam_0$. There is a sequence $\dbv{k_{\lam_n}}^{-3}\to 0$ as $n\to\infty$, which also holds for any further subsequence, leading to a contradiction. Therefore, $\dbv{k_\lam}$ is bounded uniformly in $\lam$. It follows that $k_\lam=\dbv{k_\lam} s_\lam$ is uniformly bounded in $L_1^2$. For a sequence $\lam_n\to \lam_0$ as $n\to\infty$, there is a weakly $L_1^2$-convergent subsequence $k_{\lam_{n_\ell}}\to k_\infty$ as $\ell\to\infty$. By the same argument, $k_\infty$ is a nonzero multiple of the identity. Furthermore, $k_\infty=\tx{Id}$ since $\det k_\lam\equiv 1$. It follows that $k_\lam \to \tx{Id}$ weakly in $L_1^2$.

We have 
\ben
i\Lambda_{g_{\mbp^1}}\bar\pd\pd \rd{\tr\, k_\lam-3}=i\Lambda_{g_{\mbp^1}}\bar\pd\pd\,k_\lam\le \vb{\Lambda_{g_{\mbp^1}}\tr\rd{\mbf_\lam\, s_\lam}}\,\dbv{k_\lam}
\een
which is bounded in $L^p$ uniformly in $\lam$. By Lemma \ref{lem:sim2.1}, there is $C_{10}>0$ such that $\sup \rd{\tr\, k_\lam-3}\le C_{10}\dbv{\tr\, k_\lam-3}_{L^1}$. For any sequence $k_{\lam_n}$ with $\lam_n\to\lam_0$ as $n\to\infty$, we may take subsequence $k_{\lam_{n_k}}$ as above such that $\tr\,k_{\lam_{n_k}}\to 3$ in $L^2$ hence $L^1$. The above inequality implies that $k_{\lam_{n_k}}\to 3$ uniformly, hence $k_{\lam_{n_k}}\to \tx{Id}$ uniformly. Therefore, we have that $k_\lam\to \tx{Id}$ as $\lam\to\lam_0$ uniformly on $\mbp^1$.
\end{proof}

\begin{cor} \label{cor:metriccont}
The $\lam$-family of matrix-valued function $H_{1,\lam}$ defined in Prop \ref{prop:locmodsym} is continuous with respect to $\lam$ in $C^0(\Omega)$ for any $\Omega\subset \mbc$ bounded. Furthermore, for $I_0\Subset (-1/4,1/4)$ we have $C=C(k,I_0,\Omega)>0$ such that
\ben
\sup_K \vb{\pd_\zeta^\ell \pd_{\bar \zeta}^m H_{1,\lam}}\le C
\een
for any $\ell+m=k$ and $\lam\in I_0$.
\end{cor}

\begin{proof}
For $\lam_0\in I_0$, it follows from the proof of Prop \ref{prop:clamcont} that:
\begin{itemize}
\item $\widetilde h^0_{\lam}\to\widetilde h_{1,\lam_0}$ as $\lam\to\lam_0$ uniformly, in $C^\infty$ sense, on $\Omega\Subset \mbp^1-\cl{[1:0]}$,
\item $k_\lam\to \tx{Id}$ as $\lam\to\lam_0$ uniformly on $\mbp^1$.
\end{itemize}
It follows that $\widetilde h_{1,\lam}=\widetilde h^0_\lam k_\lam\to\widetilde h_{1,\lam_0}$ uniformly as $\lam\to\lam_0$. Since $H_{1,\lam}(z^2)=K_{1,\lam}(z)$, we have that the map $(-1/4,1/4)\to C^0(\Omega,i\mfu(2))$ given by $\lam\mapsto H_{1,\lam}$ is continuous.

Recall from Lemma \ref{lem:locmod}, $\widetilde K_{1,\lam}=\rd{\widetilde h_{1,\lam}}_{\ule}$ satisfies the local form of Hitchin equation
\ben
\pd_{\bar z}\rd{\widetilde K_{1,\lam}^{-1}\pd_z \widetilde K_{1,\lam}}=\left[\theta,\widetilde K_{1,\lam}^{-1}\theta^\ast \widetilde K_{1,\lam}\right],\,\,\, \theta=\pmt{ & \beta_1 \\ \gam_1 & }\,.
\een
We have that on a bounded domain $\Omega\subset \mbp^1-\cl{[1:0]}$, $\widetilde h_{1,\lam}$ is uniformly bounded in $L_1^2$ and $C^0$. Therefore, $\dbv{\widetilde K_{1,\lam}^{\pm 1}}_{L^\infty(\Omega)}$ and $\dbv{\nabla \widetilde K_{1,\lam}}_{L^2(\Omega)}$ are uniformly bounded. By interior elliptic estimate of the Cauchy-Riemann operator, on $\Omega'\Subset \Omega$ there is $C>0$ such that
\ben
\dbv{\widetilde K_{1,\lam}^{-1}\pd_z \widetilde K_{1,\lam}}_{L_1^2(\Omega')}\le C\rd{\dbv{[\theta,\widetilde K_{1,\lam}^{-1}\theta^\ast \widetilde K_{1,\lam}]}_{L^2(\Omega)} +\dbv{\widetilde K_{1,\lam}^{-1}\pd_z \widetilde K_{1,\lam}}_{L^2(\Omega')} }\,.
\een
By Sobolev embedding $L_1^2\subset L^4$, we have that $\dbv{\nabla \widetilde K_{1,\lam}}_{L^4(\Omega')}$ is uniformly bounded. On $\Omega''\Subset \Omega$ by interior elliptic estimate of Laplacian and
\ben
\pd_{\bar z}\pd_z \widetilde K_{1,\lam}=\rd{\pd_{\bar z}\widetilde K_{1,\lam}}\widetilde K_{1,\lam}^{-1}\rd{\pd_z \widetilde K_{1,\lam}}+\widetilde K_{1,\lam}\sq{\theta,\widetilde K_{1,\lam}^{-1}\theta^\ast \widetilde K_{1,\lam}}\,,
\een
there are $C'$ and $C''>0$ such that
\begin{align*}
&\dbv{\widetilde K_{1,\lam}}_{L_2^2(\Omega'')}\le C'\rd{\dbv{\rd{\pd_{\bar z}\widetilde K_{1,\lam}}\widetilde K_{1,\lam}^{-1}\rd{\pd_z \widetilde K_{1,\lam}}}_{L^2(\Omega')}+\dbv{\widetilde K_{1,\lam}}_{L^2(\Omega'')}} \\
&\le C''\rd{ \dbv{\widetilde K_{1,\lam}}_{L^4(\Omega')}^2 + \dbv{\widetilde K_{1,\lam}}_{L^2(\Omega'')} }\,,
\end{align*}
uniformly bounded on $\lam\in I_0$. Then arguments similar to the proof of Prop \ref{prop:locmodsym} take over, and the second statement follows from Sobolev embedding theorems.
\end{proof}

\subsection{Gluing}
\label{sec:approxsoln}

In this part, we construct a $t$-family of approximate solutions to the Hitchin equation for the familly of stable SU(1,2) Higgs bundle $(F,t\bet,t\gam)$. We will use the notations in Def \ref{def:disc} as well as (\ref{eq:defDjprimes}) and let $\ulb=\rd{b_p}_{p\in D_r}$ be the corresponding Hecke parameter as in Def \ref{def:threesets}. Let $g_X$ be a metric on $X$ which restricts to Euclidean metric on $(\mbd_j; \zeta_j)$ with K\"ahler form $\omega=\frac{i}{2}d\zeta_j\wedge d\bar\zeta_j$.

\subsubsection{Gluing in disks}
\label{sec:glueondisk}

We will define a metric on $F|_{\mbd_j}$ for each $j$ which interpolates smoothly between local model solutions and a decoupled solution. This will be done on disk $\mbd_j'$ by pulling back the function $M_\lam$ via a diffeomorphism
\ben
\mbd_j'\lxrightarrow{\sim}\mbr^2
\een
which restricts to identity on $\mbd_j''$. This metric will be singular at $\pd \mbd_j'$, which is then compensated by a gauge transformation connecting it continuously to a decoupled solution on $\mbd_j-\mbd_j'$.

Let $\ulam$ be an admissible weight with respect to $\ulD$ (see \S \ref{sec:localmodel}). As in previous sections, we fix $I\Subset (-1/4,1/4)$ and constants will be independent of $\lam\in I$ unless otherwise stated. We say an inequality holds for $t\gg 1$ if there is $t_0\ge 1$ such that it holds for all $t\ge t_0$. Recall we defined for matrices $|M|=\tx{max}_{i,j}|M_{ij}|$. We have $|AB|\le n|A|\,|B|$ for $A$, $B$ $n\times n$.

Define on $\mbd_j$
\begin{align}
& \widetilde H_{t,\lam}^{\tx{int}}:=t^{2/3}G_\lam(\chi)^\ast M_\lam\rd{t^{2/3}\chi^{-1}\rho}G_\lam(\chi),\,\, \tx{ on }\mbd_j' \label{eq:tilHtint}\\
& \widetilde H_{t,\lam}^{\tx{ext}}:=t^{2/3}M_{\infty,\lam}\rd{t^{2/3}\rho}\,\, \tx{ on }\mbd_j-\mbd_j' \label{eq:tilHtext}
\end{align}
where $M_{\infty,\lam}$ is as in Prop \ref{prop:asympsum} with $\lam=\lam_j$, 
\be \label{eq:defGlam}
G_\lam(y)=\tx{diag}\rd{y^{-2\lam-1/2},y^{\lam-1/2}}\,,
\ee
and $\chi(\zeta):=\chi_0(|\zeta|/R)$ with $\chi_0\in C^\infty(\mbr_{\ge 0})$ such that $\chi_0\equiv 1$ on $[0,1/3]$ and $\chi_0\equiv 0$ on $[2/3,+\infty)$. Note that there are $C_1$, $C_2>0$ with $|\pd_\rho \chi|<C_1 R^{-1}$ and $|\pd_\rho^2 \chi|<C_2R^{-2}$. Let 
\be \label{eq:tilHtapp}
\widetilde H_{t,\lam}^{\tx{app}}(\zeta_j):=\begin{cases}
\widetilde H_{t,\lam}^{\tx{int}}(\zeta_j) & |\zeta_j|< 2R/3\\
\widetilde H_{t,\lam}^{\tx{ext}}(\zeta_j) & 2R/3\le |\zeta_j|\le R
\end{cases}\,.
\ee
Note that $\widetilde H_t^{\tx{ext}}(\zeta_j)$ can be naturally extended to $0<|\zeta_j|<2R/3$ where
\ben
\widetilde H_{t,\lam}^{\tx{ext}}(\zeta_j)=t^{2/3}M_{\infty,\lam}\rd{t^{2/3}\rho}=t^{2/3} G_\lam(\chi)^\ast M_{\infty,\lam}\rd{t^{2/3}\chi^{-1}\rho} G_\lam(\chi)\,.
\een
Let $a_\lam=\max(2\lam,-\lam)$. We have $|G_\lam(\chi)|<\chi^{-\rd{a_\lam+\ov{2}}}$ and
\begin{align*}
& \vb{\pd_\rho G_\lam(\chi)}\le \rd{a_\lam+\ov{2}}C_1R^{-1}\chi^{-a_\lam-\frac{3}{2}}, \\
& \vb{\pd_\rho^2 G_\lam(\chi)}\le \rd{a_\lam+\ov{2}}\rd{a_\lam+\frac{3}{2}}C_1^2 R^{-2}\chi^{-a_\lam-\frac{5}{2}}+\rd{a_\lam+\ov{2}}C_2 R^{-2}\chi^{-a_\lam-\frac{3}{2}}\,.
\end{align*}
Note that $\chi^{-1}\ge 1$. By Prop \ref{prop:asympsum}, on a compact set in $\mbd_j'-\cl{p_j}$, there are $C$, $c>0$, such that for $t\gg 1$,
\begin{align}
& \vb{\widetilde H_{t,\lam}^{\tx{int}}(\zeta)-\widetilde H_{t,\lam}^{\tx{ext}}(\zeta)},\,\, \vb{\pd_\rho \rd{\widetilde H_{t,\lam}^{\tx{int}}(\zeta)-\widetilde H_{t,\lam}^{\tx{ext}}(\zeta)}}\,,\nonumber \\
& \vb{\pd_{\zeta}\pd_{\bar\zeta}\rd{\widetilde H_{t,\lam}^{\tx{int}}(\zeta)-\widetilde H_{t,\lam}^{\tx{ext}}(\zeta)}}\le Ce^{-c t^{2/3}}\,. \label{eq:Htappbd}
\end{align}
Similarly on a compact set in $\mbd_j'-\cl{p_j}$ there are $C'$, $c'>0$ such that for $t\gg 1$,
\begin{align}
& \vb{\widetilde H_{t,\lam}^{\tx{int}}(\zeta)-\widetilde H_{t,\lam}^{\tx{ext}}(\zeta)},\,\, \vb{\pd_\rho \rd{\widetilde H_{t,\lam}^{\tx{int}}(\zeta)-\widetilde H_{t,\lam}^{\tx{ext}}(\zeta)}}\,, \nonumber \\
& \vb{\pd_{\zeta}\pd_{\bar\zeta}\rd{\widetilde H_{t,\lam}^{\tx{int}}(\zeta)-\widetilde H_{t,\lam}^{\tx{ext}}(\zeta)}}\le C'e^{-c' \chi^{-1}}\,. \label{eq:Htappcont}
\end{align}
By regularities of $M_{\infty,\lam}$ and (\ref{eq:Htappcont}), $\widetilde H_{t,\lam}^{\tx{app}}\in C^0(\mbd_j)\cap C^2(\mbd_j^\times)$. (It will be singular at $p_j$ since the frame $\usig$ is singular there.) We will need the following elementary estimate of matrix norms.

\begin{lem} \label{lem:matrest}
Let $A$, $B$ be such that $|A|$, $|A-B|\le M$. Then there are $C_1$, $C_1'>0$ such that
\begin{align} 
& \vb{\det A-\det B}\le C_1 M\vb{A-B} \label{eq:matrineq1} \\
& \vb{\rd{\det A}A-\rd{\det B}B}\le C_1' M^2\vb{A-B} \label{eq:matrineq2}
\end{align} 
Let $A$, $B$ be as above and $\vb{\det A}\ge \epsilon>0$ and $\vb{A-B}<\epsilon/(2C_1M)$. Then there is $C_2>0$ such that
\begin{align}
& \vb{A^{-1}-B^{-1}}\le C_2 \epsilon^{-1}\rd{1+M\epsilon^{-2}}\vb{A-B} \label{eq:matrineq3} \\
& \vb{\rd{\det A}^{-1}A^{-1}-\rd{\det B}^{-1}B^{-1}} \le C_2 \epsilon^{-2}\rd{1+M^4 \epsilon^{-2}}\vb{A-B} \label{eq:matrineq4}
\end{align}
\end{lem}

Set
\be \label{eq:defnmcht}
\mch_{t,\bet_0,\gam_0}\rd{H}:=\pd_{\bar \zeta}\rd{ H^{-1}\pd_\zeta H}-t^2\gam_0\gam_0^\ast H\rd{\det H}+t^2\rd{\det H}^{-1}H^{-1}\bet_0^\ast\bet_0\,.
\ee
$\mch_{t,\bet_0,\gam_0}(H)=0$ gives the local form of the Hitchin equation with $\bet=\bet_0d\zeta$, $\gam=\gam_0d\zeta$.

\begin{lem} \label{lem:appsol1}
For $p_j\in D_r$, there are $C$, $c>0$ such that on $\mbd_j'-\mbd_j''$ for $t\gg 1$,
\ben
\vb{\mch_{t,\bet_0,\gam_0}\rd{\widetilde H_{t,\lam}^{\tx{int}}}}\le C e^{-c t^{2/3}}\,,
\een
where $\bet_0=\pmt{0 & \zeta_j^{-1}}\,\,\gam_0=\pmt{0 & \zeta_j^2}^T$ are the local forms under $\usig$ on $\mbd_j^\times$ in (\ref{eq:stdform4}).
\end{lem}

\begin{proof}
Write $\mch_t=\mch_{t,\bet_0,\gam_0}$, $H_1:=\widetilde H_{t,\lam}^{\tx{int}}$, $H_2:=\widetilde H_{t,\lam}^{\tx{ext}}$. Note that $H_2$ as well as its extension solves the local form of Hitchin equation, i.e. $\mch_t\rd{H_2}=0$. We have $C_0$, $c_0>0$ such that for $t\gg 1$,
\ben
\vb{H_1-H_2}, \,\, \vb{\pd_{\bar\zeta}\rd{H_1-H_2}}=\vb{\pd_\zeta\rd{H_1-H_2}},\,\, \vb{\pd_{\bar\zeta}\pd_\zeta\rd{H_1-H_2}}\le C_0e^{-c_0 t^{2/3}}
\een
On the other hand there is $C_1>0$ such that for $t\gg 1$, $\vb{H_2}$, $\vb{\pd_{\bar\zeta} H_2}=\vb{\pd_\zeta H_2}$, $\vb{\pd_{\bar\zeta}\pd_\zeta H_2}\le C_1 t^{2a_\lam'/3}$ where $a_\lam'=\max\rd{-2\lam,\lam}\ge 0$. By Lemma \ref{lem:matrest} for $t\gg 1$, $\vb{\pd_{\bar\zeta}H_1}\le 2\vb{\pd_{\bar\zeta}H_2}$ and $\vb{H_1^{-1}}\le 2\vb{H_2^{-1}}$. Note also that
\ben
\pd_{\bar \zeta}\rd{H_j^{-1}\pd_\zeta H_j}=-H_j^{-1}\rd{\pd_{\bar\zeta}H_j}H_j^{-1}\rd{\pd_\zeta H_j}+H_j^{-1}\pd_{\bar\zeta}\pd_\zeta H_j\,.
\een
For $n\ge 0$ and $x\gg 1$ we have $x^n e^{-a x}\le e^{-a' x}$ for some $a'<a$. By (\ref{eq:Htappbd}) there are $C_2$, $c_2>0$ such that for $t\gg 1$,
\be \label{eq:curvbd2}
\vb{\pd_{\bar\zeta} \rd{H_1^{-1}\pd_\zeta H_1 }-\pd_{\bar\zeta} \rd{H_2^{-1}\pd_\zeta H_2 }}\le C_2e^{-c_2 t^{2/3}}
\ee
We have that there is $C_3>0$ such that for $t\gg 1$, $|\det H_2|\ge C_3 t^{-4\lam/3}$. By Lemma \ref{lem:matrest}, there are $C_4$, $c_4>0$ such that on $\mbd_j'-\mbd_j''$ with $t\gg 1$,
\begin{align*}
&\vb{\rd{\det H_1}^{-1}H_1^{-1}\bet_0^\ast\bet_0-\rd{\det H_1}^{-1}H_1^{-1}\bet_0^\ast\bet_0}\le C_4e^{-c_4 t^{2/3}} \\
&\vb{\gam_0\gam_0^\ast H_1\rd{\det H_1}-\gam_0\gam_0^\ast H_2\rd{\det H_2}}\le C_4e^{-c_4t^{2/3}}
\end{align*}
The conclusion follows from this and (\ref{eq:curvbd2}).
\end{proof}

Define
\be \label{eq:HtintDbetgam}
\begin{cases}
H_t^{\tx{int},\bet}=\tx{diag}\rd{\rho^{1/2}\exp\rd{\chi\rd{\rho}\psi_P\rd{\frac{8}{3}t\rho^{3/2}}}, 1} & \tx{ for }p_j\in D_\bet\\
H_t^{\tx{int},\gam}=\tx{diag}\rd{\rho^{-1/2}\exp\rd{-\chi\rd{\rho}\psi_P\rd{\frac{8}{3}t\rho^{3/2}}},1} & \tx{ for }p_j\in D_\gam
\end{cases}
\ee
where $\rho=\vb{\zeta_j}$ and $\psi_P$ is the Painlev\'e function. Note that for any $\ell$ we have for $\zeta_0\in \pd \mbd_j'$, $\lim_{\zeta\to \zeta_0}\pd^\ell \chi=0$ where $\vb{\zeta_0}=2R/3$, $\pd=\pd_\zeta$ or $\pd_{\bar\zeta}$.

\begin{lem} \label{lem:appsol2}
For $p_j\in D_\bet$ or $D_\gam$ there are $C$, $c>0$ such that on $\mbd_j'-\mbd_j''$ with $t\gg 1$,
\ben
\vb{\mch_{t,\bet_0,\gam_0}\rd{ H_t^{\tx{int},\bet}}} \,\, \tx{ or }\,\, \vb{\mch_{t,\bet_0',\gam_0'}\rd{ H_t^{\tx{int},\gam}}}\le Ce^{-c t}\,,
\een
where $\bet_0=\pmt{\zeta_j & 0}$, $\bet_0'=\pmt{1 & 0}$, $\gam_0=\pmt{1 & 0}^T$, $\gam_0'=\pmt{\zeta_j & 0}^T$
\end{lem}

\begin{proof}
For $p_j\in D_\gam$, define $H^{\tx{ext},\gam}=\tx{diag}\rd{\rho^{-1/2}, 1}$ on $\mbd_j''-\mbd_j'$ with $\rho=\vb{\zeta_j}$ solving the local form of decoupled equation with $\bet=\bet_0'd\zeta_j$, $\gam=\gam_0'd\zeta_j$. By properties of $\psi=\psi_P$ in Lemma \ref{lem:painleveasymp} there are $C_3$, $C_4>0$ such that for $t\gg 1$,
\ben
\vb{H_t^{\tx{int},\gam}-H^{\tx{ext},\gam}}, \,\,\vb{\pd_\rho\rd{H_t^{\tx{int},\gam}-H^{\tx{ext},\gam}}},\,\, \vb{\pd_{\bar\zeta}\pd_\zeta\rd{H_t^{\tx{int},\gam}-H^{\tx{ext},\gam}}}\le C_3 e^{-C_4 t}\,.
\een
The rest of the estimate follows from the proof of Lemma \ref{lem:appsol1} using Lemma \ref{lem:matrest}. The case $p_j\in D_\bet$ is similar.
\end{proof}

\begin{lem} \label{lem:h1lamintcont}
Let $p_j\in D_r$. The family of functions $H_{1,\lam}^{\tx{int}}$ is continuous in $C^0(\mbd_j')$ with respect to $\lam\in I$.
\end{lem}

\begin{proof}
We have $H_{1,\lam}^{\tx{int}}=H_{1,\lam}$ on $\mbd_j''$ continuous in $C^0(\mbd_j'')$ with respect to $\lam\in I$ by Cor \ref{cor:metriccont}. We will focus on the annulus $\Omega=\mbd_j'-\mbd_j''$.

Let $D_\lam=\tx{diag}\rd{4 c_\lam^{-1},c_\lam^{1/2}/2}$. We have $M_{\infty,\lam}(\rho)=\rd{D_\lam G_\lam(\rho)}^\ast \rd{D_\lam G_\lam(\rho)}$ with $G_\lam$ as in (\ref{eq:defGlam}). Let 
\ben
F_\lam(\rho):= \rd{D_\lam^\ast}^{-1}\rd{G_\lam(\rho)^\ast}^{-1}M_\lam(\rho)G_\lam(\rho)^{-1}D_\lam^{-1}\,.
\een
Since $G_\lam(a)G_\lam(b)=G_\lam(ab)$, we have that
\ben
\widetilde H_{1,\lam}^{\tx{int}}=G_\lam(\rho)^\ast D_\lam^\ast F_\lam\rd{\chi^{-1}\rho}D_\lam G_\lam(\rho)\,.
\een
We have $H_{1,\lam}^{\tx{int}}=S^\ast \widetilde H_{1,\lam}^{\tx{int}} S$ with 
\ben
S=\ov{\sqrt{2}}\pmt{\zeta & -1 \\ \zeta & 1}
\een
whose entries are bounded on $\Omega$. Furthermore, note that $G_\lam(\rho)^{-1}D_\lam^{-1}$ is continuous with respect to $\lam\in I$ in $C^0(\Omega)$. Therefore, the conclusion follows if $\zeta\mapsto F_\lam(|\zeta|)$ is continuous on $C^0(\mbr^2-B(0,R/3))$.

For $\eps>0$, by Prop \ref{prop:asympsum}, there is $\rho_1=\rho_1(\eps)\ge \max\rd{\rho_0,R/3}$ such that for $\rho\ge \rho_1$ and all $\lam,\lam_0\in I$, $\vb{F_\lam(\rho)-F_{\lam_0}(\rho)}\le \epsilon$. On the other hand, by Cor \ref{cor:metriccont}, there is $\delta=\delta(\rho_1)>0$ such that for $\vb{\lam-\lam_0}<\delta$ and $R/3\le \rho\le \rho_1$, $\vb{F_\lam(\rho)-F_{\lam_0}(\rho)}\le \epsilon$. Note that $\delta$ depends only on $\eps>0$. Therefore we have that $F_\lam\to F_{\lam_0}$ in $C^0(\mbr^2-B(0,R/3))$.
\end{proof}

\subsubsection{Global construction}
\label{sec:htapp}

In this part, we will assemble a metric $h_{t,\ulam}^{\tx{app}}$ on $F$ for a choice of weight $\ulam$ by gluing metric defined on $\left.F\right|_{\mbd_j}$ in \S \ref{sec:glueondisk} and a decoupled solution on $X-\coprod_j \mbd_j$. It is no surprise that for continuity at $\pd \mbd_j$ we need to choose the decoupled solution, i.e. the weight $\ulam$ carefully. We will be using notations $\varphi_{\ulam}$, $h_{L,\ulam}^0$, $h_{L,\tx{HE}}$ and $g_{\ell j}$ in Def \ref{def:someharmfcns}.

For a stable SU(1,2) Higgs bundle $(F,\bet,\gam)$ let $\ulb$ be corresponding Hecke parameters as in Def \ref{def:threesets}. For each $p_j\in D$ we fix once and for all a trivializing section $\cl{s_{0,j}^{(0)}}$ of $\atp{L}{\mbd_j}$ with $\rd{s_{0,j}^{(0)}(0)}^{\otimes 3}=b_{p_j}$ if $p_j\in D_r$. Let $\ulam$ be an admissible weight with respect to $\ulD$ (see \S \ref{sec:localmodel}).

\begin{defn} \label{def:tcompatible}
For $t\ge 1$, $\ulam$ is called $t$-compatible with $\ulb$ if there is a trivializing $s_{0,j}$ of $\left.L\right|_{\mbd_j}$ for each $p_j\in D_r$ such that
\begin{itemize}
\item (1) $s_{0,j}^{\otimes 3}(0)=b_{p_j}$,
\item (2) there is a harmonic metric $h_L$ adapted to $(L,\ulam)$ such that on $\mbd_j$,
\be \label{eq:tcomplocform1}
\rd{h_L^{-2}h_K\oplus h_Lh_K}_{\usig}=t^{2/3}M_{\infty,\lam_j}(t^{2/3}\rho) \tx{ for all }p_j\in D_r
\ee
where $\rho=|\zeta_j|$ and $\usig$ is induced by $s_{0,j}$ as in Theorem \ref{thm:hecke}.
\end{itemize}
\end{defn}

Note that (\ref{eq:tcomplocform1}) is exactly $\widetilde H_{t,\lam_j}^{\tx{ext}}$ in (\ref{eq:tilHtext}). Furthermore, the condition (2) is equivalent to
\begin{itemize}
\item (3) given any choice of $\{s_{0,j}\}_{p_j\in D_r}$ with $s_{0,j}\in \left.L\right|_{\mbd_j}$ and $s_{0,j}(0)^3=b_{p_j}$, there is a harmonic metric $h_L$ adapted to $\rd{L,\ulam}$,
\be \label{eq:tcomplocform2}
\vb{s_{0,j}}^2_{h_L}/\rd{4 c_{\lam_j}^{-1} t^{4\lam_j/3}\vb{\zeta_j}^{2\lam_j}}\lxrightarrow{\zeta_j\to 0} 1
\ee
\end{itemize}
Note that (2) $\Rightarrow$ (3) by a direct calculation. We next show (3)$\Rightarrow$(2) and introduce some notations.

Suppose (3) holds and denote the harmonic metric by $h_{L,\ulam,t}$. For each $p_j\in D_r$,
\be \label{eq:lognorms00j1}
\log\vb{s_{0,j}^{(0)}}_{h_{L,\ulam,t}}^2=\log\rd{c_{\lam_j}/4}+(4/3)\rd{\log t}\lam_j+2\lam_j \log\vb{\zeta_j}+o(1)
\ee
as $\zeta_j\to 0$. We have 
$\log\vb{s_{0,j}^{(0)}}_{h_{L,\ulam}^0}^2=\log\vb{s_{0,j}^{(0)}}_{h_{L,\tx{HE}}}^2+\varphi_{\ulam}$. By the uniqueness there is $\eta_{\ulam,t}\in \mbr$ such that 
\be \label{eq:defhLulamt}
h_{L,\ulam,t}=h_{L,\ulam}^0e^{\eta_{\ulam,t}}=h_{L,\tx{HE}}e^{\varphi_{\ulam}+\eta_{\ulam,t}}
\ee
therefore $\log \vb{s_{0,j}^{(0)}}_{h_{L,\ulam,t}}^2=\eta_{\ulam,t}+\log\vb{s_{0,j}^{(0)}}_{h_{L,\ulam}^0}^2$. Thus we have for all $j$
\be \label{eq:lognorms00j2}
\log\vb{s_{0,j}^{(0)}}_{h_{L,\ulam,t}}^2=\eta_{\ulam,t}+\log\vb{s_{0,j}^{(0)}(0)}_{h_{L,\tx{HE}}}^2+\sum_{\ell=1}^N\rd{\sum_{k=1}^\ell \lam_k}g_{\ell j}+2\lam_j \log\vb{\zeta_j}
\ee
Compare (\ref{eq:lognorms00j1}) and (\ref{eq:lognorms00j2}) we get for all $p_j\in D_r$
\be \label{eq:culamt}
\eta_{\ulam,t}=(4/3)\rd{\log t}\lam_j+\log\rd{c_{\lam_j}/4}-\log\vb{s_{0,j}^{(0)}(0)}_{h_{L,\tx{HE}}}^2-\sum_{\ell=1}^N\rd{\sum_{k=1}^\ell \lam_k}g_{\ell j}(0)
\ee
Since $\mbd_j$ is simply connected, there are unique holomorphic functions $\xi_j$, $f_{\ell j}$ on $\mbd_j$ such that 
\begin{align}
&\tx{Re} \xi_j=\log\vb{s_{0,j}^{(0)}}_{h_{L,\tx{HE}}}^2-\log \vb{s_{0,j}^{(0)}(0)}_{h_{L,\tx{HE}}}^2 \nonumber \\
&\tx{Re} f_{\ell j}=g_{\ell j}-g_{\ell j}(0) \label{eq:xiandfellj}
\end{align} 
and $\xi_j(0)=f_{\ell j}(0)=0$. Set
\be \label{eq:Fulamj}
F_{\ulam,j}=\xi_j+\sum_{\ell=1}^N\rd{\sum_{k=1}^\ell \lam_k}f_{\ell j}\,.
\ee

\begin{defn} \label{def:weightdependentframes}
For each $p_j\in D_r$, let
\be \label{eq:weightdeps0}
s_{0,j}^{\ulam}=e^{-F_{\ulam,j}/2}s_{0,j}^{(0)}\,.
\ee
Let $\usig_j^{(0)}=\cl{\sig_{1,j}^{(0)},\sig_{2,j}^{(0)}}$ and $\uls_j^{(0)}=\cl{s_{1,j}^{(0)},s_{2,j}^{(0)}}$ for all $j$ (resp. $\usig_j^{\ulam}=\cl{\sig_{1,j}^{\ulam},\sig_{2,j}^{\ulam}}$ and $\uls_j^{\ulam}=\cl{s_{1,j}^{\ulam},s_{2,j}^{\ulam}}$ for $p_j\in D_r$) be the frames induced by the frame $\cl{s_{0,j}^{(0)}}$ (resp. the frame $\cl{s_{0,j}^{\ulam}}$) as in Theorem \ref{thm:hecke}. Note that the frame $\usig=\usig_j^{\ulam}$ satisfies (\ref{eq:tcomplocform1}).
\end{defn}

From the above, condition (3) is also equivalent to
\begin{itemize}
\item (4) There is $\eta\in \mbr$ such that for all $p_j\in D_r$, we have
\be \label{eq:tcomplocform3}
\eta+\psi_j\rd{\ulam}=\log \rd{c_{\lam_j}^{1/2}/2}+(2/3)\rd{\log t}\lam_j
\ee
where $\psi_j\rd{\ulam}=\log\vb{s_{0,j}^{(0)}(0)}_{h_{L,\tx{HE}}}+(1/2)\sum_{\ell=1}^N\rd{\sum_{k=1}^\ell \lam_k}g_{\ell j}(0)$ is an affine function in $\lam_j$ with $p_j\in D_r$.
\end{itemize}

For the example in (\ref{eq:specialSU12}), $L\cong \mco_X$, $\ulb=\rd{1,\ldots,1}$ and $h_{L,\tx{HE}}$ may be taken to be the trivial metric on $\mco_X$. Recall that $c_{\lam=0}=4$. It follows that $\ulam=\rd{0,\ldots,0}$ is $t$-compatible with $\ulb$ for all $t$. Note also that $c_{\ulzero,t}$ in (\ref{eq:culamt}) is independent of $t$ in this case.

\begin{prop} \label{prop:tcompfamily}
There are $C>0$, $t_0\ge 1$ such that for $t\gg 1$ there is an admissible weight $\ulam(t)\in \msp_{\ulD}$, $t$-compatible with $\ulb$ and such that
\be \label{eq:lammigrate}
\vb{\ulam(t)-\ulami}\le \frac{C}{\log t}
\ee
where $\ulami$ is defined in Theorem \ref{thm:main}.
\end{prop}

\begin{proof}
Let $\lam_c=-\rd{\deg L+(d_\bet-d_\gam)/4}/d_r$ which is the value $\lam_{\infty,j}$ for $p_j\in D_r$. Set
\ben
\widetilde\msp_{\ulD}=\left\{\umu=\rd{\mu_1,\ldots,\mu_{d_r}}\middle| \mu_k+\lam_c\in(-1/4,1/4),\,\, \sum_k\mu_k=0 \right\}\subset H\subset \mbr^{d_r}\,,
\een
an open subset of $H=\cl{\umu|\sum_k\mu_k=0}$, a hyperplane in $\mbr^{d_r}$. By Prop \ref{prop:su12stab2}, $\ulzero=(0,\ldots,0)\in \widetilde\msp_{\ulD}$. Let $\pi: \mbr^{d_r}\to H,\,\, \umu\mapsto \rd{\mu_k-d_r^{-1}\sum_\ell \mu_\ell}_{k=1}^{d_r}$. We have that $\pi(\umu-\umu')=0$ iff there is a constant $c$ such that $\mu_k=\mu_k'+c$ for all $k$.

Set
\ben
G_t:\widetilde\msp_{\ulD}\to\mbr^{d_r}, \,\, G_t\rd{\umu}=(3/2)(\log t)^{-1}\rd{\psi_{i_k}\rd{\ulam}-\log\rd{c_{\lam_c+\mu_k}^{1/2}/2}}_{k=1}^{d_r}\,,
\een
where $1\le i_1<\ldots<i_{d_r}\le 4g-4$ are the indices such that $p_{i_j}\in D_r$ and $\ulam$ is the tuple with $\lam_{i_k}=\lam_c+\mu_k$ and set $F_t=\pi\circ G_t$. We can rewrite (\ref{eq:tcomplocform3}) as 
\ben
\mu_k=G_t\rd{\ulam}+\rd{(3/2)(\eta/\log t)-\lam_c}\,.
\een
i.e., $\ulam$ is $t$-compatible with $\ulb$ iff 
\ben
F_t\rd{\umu}=\umu\,.
\een

$F_t$ is continuous by Prop \ref{prop:clamcont}. There is $r_0>0$ such that $H\cap \overline{B(0,r_0)}\Subset \widetilde\msp_{\ulD}$ and that $F_t: H\cap \overline{B(0,r_0)}\to H\cap \overline{B(0,r_0)}$. The conclusion follows then from Brouwer's fixed point theorem.
\end{proof}

Let $r_0>0$ be as in the proof and set $I=[\lam_c-r_0,\lam_c-r_0]$. Let 
\be \label{eq:mci}
\mci=\left\{\ulam\in\msp_{\ulD} \middle| \lam_j\in I\tx{ for }p_j\in D_r\right\}\,.
\ee
We fix a family $\ulam(t)\in \mci$ for $t\gg 1$, $t$-compatible with $\ulb$.

\begin{defn} \label{def:htapp}
Let $\usig_j^{(0)}$ (resp. $\uls_j^{(0)}$) be the frames of $\atp{V}{\mbd_j}$ (resp. $\atp{F}{\mbd_j}$) defined in Def \ref{def:weightdependentframes}. Let $h_t^{\tx{app}}$ be a Hermitian inner product on $F$ defined by a piecewise expression
\ben
\begin{cases}
\rd{h_t^{\tx{app}}}_{\usig_j^{(0)}}=T_{\ulam(t),j}^\ast \widetilde H_{t,\lam_j(t)}^{\tx{app}} T_{\ulam(t),j} & \tx{ on }\mbd_j,\,\, p_j\in D_r \\
\rd{h_t^{\tx{app}}}_{\uls_j^{(0)}}=\rd{T_{\ulam(t),j,t}'}^\ast H_t^{\tx{int},\bet} T_{\ulam(t),j,t}' & \tx{ on }\mbd_j\,\, p_j\in D_\bet\\
\rd{h_t^{\tx{app}}}_{\uls_j^{(0)}}=\rd{T_{\ulam(t),j,t}'}^\ast H_t^{\tx{int},\gam} T_{\ulam(t),j,t}' & \tx{ on }\mbd_j\,\, p_j\in D_\gam\\
h_t^{\tx{app}}=h_{\infty,t} & \tx{ on }X-\coprod_j \mbd_j
\end{cases} 
\een
where 
\ben
h_{\infty,t}=\iota^\ast\rd{h_{L,\ulam(t),t}^{-2}h_K\oplus h_{L,\ulam(t),t}h_K}
\een
with $\iota: F\to V=L^{-2}K_X\oplus LK_X$ the Hecke modification corresponding to $(F,\bet,\gam)$ as in \S \ref{sec:setup}
and $\widetilde H_{t,\lam}^{\tx{app}}$ (resp. $H_t^{\tx{int},\bet}$, $H_t^{\tx{int},\gam}$) is defined in (\ref{eq:tilHtapp}) (resp. (\ref{eq:HtintDbetgam})) and
\begin{align}
& T_{\ulam,j}=\tx{diag}\rd{e^{-F_{\ulam,j}},e^{F_{\ulam,j}/2}},\nonumber \\
& T_{\ulam,j,t}'=\tx{diag}\rd{\tau_{\ulam(t),t}^{-1/2},\tau_{\ulam(t),t}}\,, \label{eq:defTulamjtprime}
\end{align}
where $F_{\ulam,j}$ is defined in (\ref{eq:Fulamj}) and 
\ben
\tau_{\ulam,t}=4 c_{\lam_{j_0}}^{-1}t^{-4\lam_{j_0}/3}e^{-\widetilde F_{\ulam,j}}
\een
where $j_0$ is a fixed index such that $p_{j_0}\in D_r$ (recall by Prop \ref{prop:su12stab2} we have $D_r\neq\emptyset$) and 
\begin{align*}
&\widetilde F_{\ulam,j}=F_{\ulam,j}-\log \widetilde \eta_{\ulam,j}\,, \\
&\widetilde \eta_{\ulam,j}=\frac{\vb{s_{0,j_0}^{(0)}}_{h_{L,\tx{HE}}}^2}{\vb{s_{0,j}^{(0)}}_{h_{L,\tx{HE}}}^2} \exp\rd{
\sum_{\ell=1}^{4g-4}\rd{\sum_{k=1}^\ell \lam_k}\rd{g_{\ell j_0}(0)-g_{\ell j}(0)}\,.
}
\end{align*}
\end{defn}

In view of (\ref{eq:weightdeps0}) and the form of $h_t^{\tx{app}}$ on $\mbd_j$ with $p_j\in D_r$, it is not hard to verify that for $p_j\in D_r$, we have that
\ben
\rd{h_t^{\tx{app}}}_{\usig_j^{\ulam}}=t^{2/3}M_{\lam_j}(t^{2/3}\rho)\,\,\tx{ on }\mbd_j''\tx{ for }p_j\in D_r\,.
\een
Therefore, we have by Def \ref{def:metrics} and (\ref{eq:htlaminetafrm}) that on $\mbd_j''$, $\rd{h_t^{\tx{app}}}_{\uls_j^{\ulam(t)}}=H_{t,\lam_j(t)}$ the local model solution in Prop \ref{prop:locmodsym}. In particular we see that $h_t^{\tx{app}}$ is smooth and non-singular at $D_r$. 

Define $\uls_{j,t}^{\ulam}=\cl{s_{1,j,t}^{\ulam},s_{2,j,t}^{\ulam}}$ on $\atp{F}{\mbd_j}$ by
\be \label{eq:defsijt}
\begin{cases}
s_{1,j,t}^{\ulam}=\tau_{\ulam,t}^{1/2}s_{1,j}^{(0)}\\
s_{2,j,t}^{\ulam}=\tau_{\ulam,t}s_{2,j}^{(0)}
\end{cases}\,. 
\ee
By a direct calculation for $\ulD$-admissible weight $\ulam$,
\be \label{eq:localformundersjtulam}
\rd{h_t^{\tx{app}}}_{\uls_{j,t}^{\ulam}}=\begin{cases}
H_t^{\tx{int},\bet}\,\, \tx{ for }p_j\in D_\bet,\\
H_t^{\tx{int},\gam}\,\, \tx{ for }p_j\in D_\gam
\end{cases}\,.
\ee
We have
\ben
\rd{h_{\infty,t}}_{\uls_{j,t}^{\ulam}}=
\begin{cases}
\tx{diag}\rd{\rho^{1/2},1} & p_j\in D_\bet \\
\tx{diag}\rd{\rho^{-1/2},1} & p_j\in D_\gam
\end{cases}\,.
\een
By the definitions of $H_t^{\tx{int},\bet}$ and $H_t^{\tx{int},\gam}$ in (\ref{eq:HtintDbetgam}) as well as the asymptotic properties $\psi_P$ in Prop \ref{prop:locmodpainleve}, we have that $h_t^{\tx{app}}$ is $C^2$ at $\pd \mbd_j$ for $p_j\in D_\bet$ and $D_\gam$. On the other hand, we saw that $\widetilde H_{t,\lam_j(t)}^{\tx{app}}$ is in $C^2\rd{\mbd_j^\times}$ in \S \ref{sec:glueondisk}. By condition (2) in Def \ref{def:tcompatible}, we see that $h_t^{\tx{app}}$ on $\mbd_j-\mbd_j'$ is identical to $h_{\infty,t}$. Therefore it is also $C^2$ at $\pd \mbd_j$ for $p_j\in D_r$. This proves the first statement in the following proposition.

\begin{prop} \label{prop:htapp}
$h=h_t^{\tx{app}}$ as above is a $C^2$ metric on $F$ over $X$. There are $C$, $c>0$ such that for $t\gg 1$,
\ben
\vb{F_{\nabla_h}+t^2\gam\wedge \gam^{\ast_h}+t^2\bet^{\ast_h} \wedge\bet}_{h_{t_0}^{\tx{app}}}\le C e^{-c t^{2/3}}
\een
\end{prop}

\begin{proof}
By definition, $h$ solves the Hitchin equation everywhere except on $\mbd_j'-\mbd_j''$. Note that $\xi$, $f_{\ell j}$ in (\ref{eq:xiandfellj}) are independent of $t$ and $\ulam$. There is $C_0>0$ such that for all $\ulam\in \mci$, $\vb{F_{\ulam,j}}\le C_0$ on $\mbd_j$ for $p_j\in D_r$. Therefore, there is $C_1>0$ such that $\vb{T_{\ulam,j}}$, $\vb{T_{\ulam,j}^{-1}}<C_1$. We have $C_2>0$ such that for all $\ulam\in\mci$,
\ben
\vb{\mch_{t,\bet_0,\gam_0}\rd{T_{\ulam,j}^\ast \widetilde H_{t,\lam_j}^{\tx{app}}T_{\ulam,j}}}=\vb{T_{\ulam,j}^{-1}\mch_{t,\bet_0,\gam_0}\rd{\widetilde H_{t,\lam_j}^{\tx{app}}}T_{\ulam,j}}\le C_2 \vb{\mch_{t,\bet_0,\gam_0}\rd{\widetilde H_{t,\lam_j}^{\tx{app}}}}\,,
\een
with $\bet_0=\ov{\sqrt{2}}\pmt{\zeta_j & 1}$, $\gam_0=\ov{\sqrt{2}}\pmt{1 & \zeta_j}^T$. On the other hand since $T_{\ulam,j,t}'$, $\mch_{t,\bet_0',\gam_0'}$, and $\mch_{t,\bet_0'',\gam_0''}$ (with $\bet_0'=\pmt{1 & 0}$, $\bet_0''=\pmt{\zeta_j & 0}$, $\gam_0'=\pmt{\zeta_j & 0}^T$, and $\gam_0''=\pmt{1 & 0}^T$) are all diagonal we have
\ben
\mch_{t,\bet,\gam}\rd{\rd{T_{\ulam,j,t}'}^\ast \widetilde H_{t,\lam_j}^{\tx{app}}T_{\ulam,j,t}'}=\rd{T_{\ulam,j,t}'}^{-1}\mch_{t,\bet,\gam}\rd{\widetilde H_{t,\lam_j}^{\tx{app}}}T_{\ulam,j,t}'=\mch_{t,\bet,\gam}\rd{\widetilde H_{t,\lam_j}^{\tx{app}}}
\een
for $\bet=\bet_0'$ (resp. $\bet_0''$) and $\gam=\gam_0'$ (resp. $\gam_0''$). The conclusion follows by combining Lemmas \ref{lem:appsol1}, \ref{lem:appsol2}, and the fact that $h_{t_0}^{\tx{app}}$ is independent of $t$.
\end{proof}

The following will be a consequence of (\ref{eq:Htappbd}) and Lemma \ref{lem:bddifffj} in the next section.

\begin{prop} \label{prop:htappclosetoext}
On $X_0\subset X-D$ compact, there are $C$, $c>0$ such that for all $t\gg 1$,
\ben
\dbv{\,g_t-\tx{Id}\,}_{L_2^2\rd{X_0},h_{t_0}^{\tx{app}}}\le C e^{-c t^{2/3}}
\een
where $g_t$ is given by $h_t^{\tx{app}}=h_{\infty,t}\cdot g_t$ and $\dbv{\,\cdot\,}_{L_2^2\rd{X_0},h}$ is defined in (\ref{eq:defLk2norm1}).
\end{prop}

For future reference, we gather below the local forms of $h_t^{\tx{app}}$ in various regions and frames.

\begin{center}
{\tabulinesep=1.2mm 
\begin{tabu}{|l|l|l|}
\hline
Region & Frame & Local form of $h_t^{\tx{app}}$ \\
\hline
$\mbd_j-\cl{p_j}$, $p_j\in D_r$ & $\sig_j^{\ulam(t)}$ & $\widetilde H_{t,\lam_j(t)}^{\tx{app}}$, see (\ref{eq:tilHtapp}) \\
\hline
$\mbd_j-\cl{p_j}$ , $p_j\in D_r$ & $\sig_j^{(0)}$ & $T_{\ulam(t),j}^\ast \widetilde H_{t,\lam_j(t)}^{\tx{app}} T_{\ulam(t),j}$ see Def \ref{def:htapp} \\
\hline
$\mbd_j$, $p_j\in D_r$ & $s_j^{\ulam(t)}$ & $S^\ast \widetilde H_{t,\lam_j(t)}^{\tx{app}} S$ where $S$ is defined in (\ref{eq:defS2}) \\
\hline
$\mbd_j$, $p_j\in D_r$ & $s_j^{(0)}$ & $S^\ast T_{\ulam(t),j}^\ast \widetilde H_{t,\lam_j(t)}^{\tx{app}} T_{\ulam(t),j}S$ \\
\hline
$\mbd_j'-\cl{p_j}$, $p_j\in D_r$ & $\sig_j^{\ulam(t)}$ & $\widetilde H_{t,\lam_j(t)}^{\tx{int}}$, see (\ref{eq:tilHtint}) \\
\hline
$\mbd_j''-\cl{p_j}$, $p_j\in D_r$ & $\sig_j^{\ulam(t)}$ & $t^{2/3}M_{\lam_j(t)}\rd{t^{2/3}\vb{\zeta_j}}$, see (\ref{eq:defMlamfirst}) \\
\hline
$\mbd_j''$, $p_j\in D_r$ & $s_j^{\ulam(t)}$ & $H_{t,\lam_j(t)}$ see Prop \ref{prop:locmodsym} \\
\hline
$\mbd_j-\mbd_j'$, $p_j\in D_r$ & $\sig_j^{\ulam(t)}$ & $t^{2/3}M_{\infty,\lam_j}\rd{t^{2/3}\vb{\zeta_j}}$, see Prop \ref{prop:asympsum} \\
\hline
$\mbd_j$, $p_j\in D_\bet$ & $s_{j,t}^{\ulam(t)}$ & $H_t^{\tx{int},\bet}$, see (\ref{eq:HtintDbetgam}) \\
\hline
$\mbd_j$, $p_j\in D_\bet$ & $s_j^{(0)}$ & $\rd{T_{\ulam(t),j,t}'}^\ast H_t^{\tx{int},\bet} \rd{T_{\ulam(t),j,t}'}$, see Def \ref{def:htapp} \\
\hline
$\mbd_j$, $p_j\in D_\gam$ & $s_{j,t}^{\ulam(t)}$ & $H_t^{\tx{int},\gam}$, see (\ref{eq:HtintDbetgam}) \\
\hline
$\mbd_j$, $p_j\in D_\gam$ & $s_j^{(0)}$ & $\rd{T_{\ulam(t),j,t}'}^\ast H_t^{\tx{int},\gam} \rd{T_{\ulam(t),j,t}'}$, see Def \ref{def:htapp}\\
\hline
\end{tabu}}
\end{center}

\section{Proof of the main theorem}
\label{sec:proof}

\subsection{Linearization}
\label{sec:linear}

In this part, we linearize the operator assocated to the Hitchin's equation at $h_{t,\ulam}^{\tx{app}}$ and use contraction mapping principle to find a solution near it. We begin by defining various norms that will be used later. Let $g_X$ be the K\"ahler metric defined in \S \ref{sec:approxsoln}. For a matrix-valued $L^2$ function $M$ and an open subset $\Omega\subset X$, we use two equivalent definitions of $L^2$-norms interchangably when no ambiguity arises:
\begin{align}
& \|M\|_{L^2(\Omega)}:=\rd{\int_\Omega \tr\rd{M^\ast M} \tx{dvol}_{g_X}}^{1/2}\,,\label{eq:L2normdef1} \\
& \|M\|_{L^2(\Omega)}:=\rd{\int_\Omega \sup_{i,j}\vb{M_{ij}}^2 \tx{dvol}_{g_X}}^{1/2}\,,\label{eq:L2normdef2}
\end{align}
When $\Omega=X$, we write $\dbv{\ldots}_{L^2}$. Let $F$ be the rank-two bundle in a stable SU(1,2) Higgs bundle $(F,\bet,\gam)$. Fix on $X$ a finite atlas $(U_\alpha;z_\alpha)_{\alpha\in\mca}$ over which $F$ is trivialized by holomorphic frames $\uls_\alp$. Fix a smooth partition of unity $0\le \rho_\alpha\le 1$, with supp$\rho_\alpha\subset U_\alpha$ and $\sum_\alpha \rho_\alpha=1$. The difference in metrics on $F$ will be measured by an automorphism of the form $g=e^u$ with $u\in \tx{End}(F)$. For $u$ with local forms $u_{\uls_\alp}\in C^k(U_\alp)$, define
\be \label{eq:Cknorm}
\dbv{u}_{C^k}=\sum_{j=0}^k \max_\alp \, \sup_{U_\alp} \, \max_{i_1,\ldots,i_k=z_\alp, \bar z_\alp}\vb{\pd_{i_1}\ldots \pd_{i_k} u_{\uls_\alp}}
\ee
and for $u_{\uls_\alp}\in L^2(U_\alp)$, $h$ a metric on $F$, define
\be \label{eq:L2normdef3}
\dbv{u}_{L^2,h}=\rd{\int_X \vb{u}_{h}^2 \tx{dvol}_{g_X}}^{1/2}=\rd{\int_X \tr\rd{u u^{\ast_h}} \tx{dvol}_{g_X}}^{1/2}\,.
\ee
We note that for a fixed $h$, the norm
\ben
\dbv{u}_{L^2}^2=\sum_\alpha \int_{U_\alpha} \rho_\alpha \vb{u_{\uls_\alp}}^2 dx_\alpha dy_\alpha
\een
is equivalent to $\dbv{\cdot}_{L^2,h}$. Define inductively two equivalent norms,
\begin{align}
& \dbv{u}_{L_k^2,h}^2=\sum_{\alp\in\mca} \int_{U_\alpha} \rho_\alpha \sum_{\ell+m=k}\vb{\pd_{\bar z_\alpha}^\ell\pd_{z_\alpha}^m u_{\uls_\alp}}_{h_{\uls_\alp}}^2 \tx{dvol}_{g_X}+\dbv{u}_{L_{k-1}^2,h}^2\,, \label{eq:defLk2norm1}\\
& \dbv{u}_{L_k^2}^2=\sum_\alpha \int_{U_\alpha} \rho_\alpha \sum_{\ell+m=k}\vb{\pd_{\bar z_\alpha}^\ell\pd_{z_\alpha}^m u_{\uls_\alpha}}^2 \tx{dvol}_{g_X}+\dbv{u}_{L_{k-1}^2}^2\,. \label{eq:defLk2norm2}
\end{align}
Given a metric $h$ on $F$ and $u\in \tx{End}(F)$, define
\begin{align}
& \mch_{t,h}: \tx{Herm}\rd{F,h}\lto \Omega^{1,1}\rd{\tx{End}(F)} \nonumber \\
& u \lmapsto F_{\nabla_{h\cdot g}}+t^2\gam\wedge\gam^{\ast_{h\cdot g}}+t^2\bet^{\ast_{h\cdot g}}\wedge \bet \label{eq:defnmchth}
\end{align}
where $g=e^u$. Note that by Sobolev embedding, if $u\in L^2_2(\Omega)$, then $g=e^u\in L^2_2$ as well. We have
\ben
\mch_{t,h}: L_2^2\rd{\tx{Herm}(F,h)}\lto L^2\rd{\Omega^{1,1}\rd{\tx{End}(F)}}\,.
\een
This has a Fr\'echet derivative at $u=0$ given by
\ben
D\mch_{t,h}: u\lmapsto \bar\pd\pd_h u+t^2\rd{\gam\wedge \gam^{\ast_h}\, \hat u-\hat u\, \bet^{\ast_h}\wedge\bet}\,,
\een
where $\hat u=u+\rd{\tr\, u}\tx{Id}$. Note that $\mch_{t,h}(u)$ is Hermitian with respect to $h\cdot e^u$ instead of $h$. We instead consider $u\mapsto g^{1/2}\mch_{t,h}\rd{u}g^{-1/2}$ where $g^{1/2}=e^{u/2}$.

\begin{defn} \label{def:Lth}
For $h$ a $L_2^2$ metric on $F$ and $t\ge 1$, let
\begin{align*}
& L_t^{(h)}: L_2^2\rd{\tx{Herm}(F,h)}\lto L^2\rd{\tx{Herm}(F,h)}\\
& u \lmapsto i\Lambda \rd{2 D\mch_{t,h}(u)+[u ,\mch_{t,h}(0)]}\,.
\end{align*}
\end{defn}

It follows from direct calculations that $L_t(u)$ is Fr\'echet derivative of 
\ben
u\mapsto 2i\Lambda g^{1/2}\mch_{t,h}(u)g^{-1/2}
\een
at $u=0$. We have a more convenient form of $L_t$ which follows from the K\"ahler identities $[i\Lambda,\pd_h]=-\rd{\bar\pd}^\ast$, $[i\Lambda,\bar\pd]=\rd{\pd_h}^\ast$ and the Kodaira-Nakano identity $\Delta''-\Delta'=[iF_{\nabla_h},\Lambda]$ (see e.g. \cite{Dem86})

\begin{prop} \label{prop:Ltform}
We have
\be \label{eq:Ltformweuse}
L_t^{(h)}(u)=\Delta_h u+t^2\cl{\psi_{\bet,\gam,h},\hat u}
\ee
where $\cl{A,B}:=AB+BA$, $\hat u=u+\rd{\tx{tr} u}\tx{Id}$, $\Delta_h=d_h \rd{d_h}^\ast+\rd{d_h}^\ast d_h$ with $d_h=\bar\pd+\pd_h$ ($\rd{\cdot}^\ast$ is the formal adjoint with respect to $h$ and $g_X$) and
\be \label{eq:defnpsibetagamma}
\psi_{\bet,\gam,h}:=i\Lambda\rd{\gam\wedge\gam^{\ast_h}-\bet^{\ast_h}\wedge\bet}\,.
\ee
\end{prop}

\begin{defn} \label{def:agag}
Let $h$ be a metric on $F$. This induces a pointwise pairing $\ag{u,v}_h=\tr\rd{v^{\ast_h} u}$ for $u$, $v\in \tx{End}(F)$. Define
\ben
\ag{\ag{u,v}}_h:=\ag{u,\hat v}_h=\ag{\hat u,v}_h
\een
where $\hat u=u+\rd{\tr\, u}\tx{Id}$ and denote $\rd{\rd{u,v}}_h:=\int_X \ag{\ag{u,v}}_h\omega$.
\end{defn}

The following lemma is easy to verify by a direct calculation.

\begin{lem} \label{lem:quadform001}
We have for $u\in L_2^2\rd{\tx{Herm}(F,h)}$,
\be \label{eq:defnQt}
Q_t^{(h)}(u):=\rd{\rd{L_t u,u}}_h=\|d\rd{\tr\, u}\|^2+2 \|\bar\pd u\|^2+2t^2\|\hat u\circ\gam\|^2+2t^2\|\bet\circ\hat u\|^2
\ee
where $\dbv{\cdot}=\dbv{\cdot}_{L^2,h,g_X}$.
\end{lem}

As a consequence of the above lemma ker $L_t=0$. To see this suppose $L_t^{(h)} u =0$, then $Q_t^{(h)}(u)=0$. As a result $\tr\, u=c$ a constant. As noted in the proof of Lemma \ref{lem:decoupledsoln}, we have on $X-D$, $(F,\bet,\gam)\cong (V,\pmt{0 & q^{-1}}, \pmt{0 & q^2}^T)$ under Hecke modification $\iota$ of Theorem \ref{thm:hecke}. Since $\hat u\gam$, $\bet\hat u=0$ we have $u=\tx{diag}(2c, -c)$. By stability of $(F,\bet,\gam)$ from Prop \ref{prop:su12stab2}, $d_r>0$. A calculation using local forms from e.g. (\ref{eq:betagammaform01}), (\ref{eq:iotaform01}) over $\mbd_j$ with $p_j\in D_r$ shows that $c=0$. Note that this last step is crucial. In fact when restricted to $X-D$, the operator $L_t$ does have a nontrivial kernel spanned by diag$(2c,-c)$.

Let $h=h_t^{\tx{app}}$ be the approximate solution as in Def \ref{def:htapp}. In the following we will denote $L_t=L_t^{(h)}$ and omit the subscripts to write $\ag{\ag{u,v}}$, $((u,v))$.

The proof of the main theorem will proceed as follows. First, we prove a $t$-dependent bound for $L_t$ from below for $t\gg 1$. In particular, we show there are $C$, $p>0$ such that $\|L_t u\|_{L^2}\ge Ct^p \|u\|_{L_2^2}$. Write $h=h_t^{\tx{app}}$ and $g=e^u$ and define the remainder term $R_t$ by
\be \label{eq:defnRtu}
2i\Lambda g^{1/2}\mch_{t,h}\rd{u}g^{-1/2}=2i\Lambda g^{1/2}\mch_{t,h}\rd{0}g^{-1/2}+L_t\rd{u}+R_t\rd{u}\,.
\ee
A fixed point of
\ben
u\lmapsto -L_t^{-1}\rd{2i\Lambda g^{1/2}\mch_{t,h}(0)g^{-1/2}+R_t(u)}\,.
\een
is a solution of the Hitchin equation. We prove relevant upper bounds for $R_t$ and use the contraction mapping principle to show convergence of the corresponding iteration sequence to a fixed point.

It is worth remarking that even though the outline the proof resembles the work in \cite{MSWW16}, there are some significant differences. For instance, the global estimates for the analogous operator $L_t$ in \cite{MSWW16} are by a combination of local estimates and the domain monotonicity principle. As we saw above this no longer works for our context since on $X-\coprod_j \mbd_j$ the operator $L_t$ has a nontrivial kernel. Furthermore in contrast to Prop 5.2 (i) and Lemma 6.3 of \cite{MSWW16}, there is no $t$-independent $L^2\to L^2$ lower bound of $L_t$. We are able to give a $t$-dependent lower bound in Prop \ref{prop:LtL2lowbd} which $\to 0$ as $t\to \infty$. The analysis is further complicated by t-dependency of the weight $\ulam(t)$ in $h_t^{\tx{app}}$.

\subsection{\texorpdfstring{$L^2$}{L2} lower bound for \texorpdfstring{$L_t$}{Lt}} \label{sec:L2lowerbound}

We will need the following technical lemma. Let $\mbd$ be the unit disk and define on $L_1^2(\mbd)$,
\ben
Q_t^{(0)}(u)=\|d u\|_{L^2(\mbd)}^2+\rd{F_t u,u}_{L^2(\mbd)}
\een
where $F_t(\zeta)=t^2 F(t \zeta)$, and $F\in L^\infty(\mbr^2)$ is a non-negative function such that there are $A>0$ and $0<\delta<1$ such that $F\ge A>0$ on $B(0,\delta)$. An ineqaulity is said to hold for $t\gg 1$ if there is $t_0\ge 1$ such that it holds for $t\ge t_0$.

\begin{lem} \label{lem:schroedingerneumann}
There are $C>0$ and for all $u\in L_1^2(\mbd)$ and $t\gg 1$, we have
\ben
Q_t^{(0)}(u)\ge \frac{C}{\log t}\|u\|_{L^2(\mbd)}^2
\een
\end{lem}

\begin{proof}
Let $G_t(\zeta)=t^2 G(t \zeta)$ with $G(\zeta)=A$ for $|\zeta|<\delta$ and zero elsewhere. Let 
\ben
P_t(u):=\|d u\|_{L^2(\mbd)}^2+\rd{G_t u,u}_{L^2(\mbd)}=\|d u\|_{L^2(\mbd)}^2+A^2t^4\|u\|_{L^2(B(0,t^{-1}\delta))}^2
\een
We have $F_t\ge G_t$, therefore $Q_t^{(0)}(u)\ge P_t(u)$, and it suffices to bound $P_t$ from below. By Rayleigh's theorem,
\be \label{eq:lam1variational}
\lam_1=\lam_1(t):=\inf_{u\in L_1^2(\mbd)-\cl{0}}\frac{P_t(u)}{\|u\|_{L^2(\mbd)}^2}
\ee
is the first eigenvalue of the self-adjoint operator $u\lmapsto -\Delta u +G_t u$ on $L^2(\mbd)$ with the Neumann boundary condition $\pd_r u=0$ on $\pd \mbd$.

For $t>0$, we have $\lam_1(t)>0$. To see this, suppose there is a sequence $u_n\in L_1^2(\mbd)$ with $\|u_n\|_{L^2(\mbd)}=1$ and $P_t(u_n)\to 0$ as $n\to \infty$. It follows that $du_n\to 0$ in $L^2(\mbd)$ as $n\to\infty$. Thus $u_n$ is bounded in $L_1^2(\mbd)$-norm. Denote again by $u_n$ a subsequence such that $u_n\to u_\infty$ weakly in $L_1^2(\mbd)$ and strongly in $L^2(\mbd)$. We have $\dbv{u_\infty}_{L^2(\mbd)}=1$ and  
\ben
\dbv{u_\infty}_{L_1^2(\mbd)}\le \liminf_{n\to\infty}\dbv{u_n}_{L_1^2(\mbd)}=1
\een
therefore $du_\infty=0$, and $u_\infty$ is a nonzero constant. We would have $0<P_t(u_\infty)\le \liminf_{n\to\infty} P_t(u_n)=0$ which leads to contradiction.

Next note that $\lam_1(t)\to 0$ as $t\to\infty$. This is a consequence of the existence of an $L_1^2(\mbd)$ function which is unbounded at the origin, e.g. $\zeta\mapsto \log\log\rd{2+1/|\zeta|}$. In fact there is $C>0$ such that
\be \label{eq:lam1tupperbound}
\lam_1(t)\le \frac{C}{\rd{\log\log \frac{t}{2\delta}}^2}\,.
\ee
To see this, let $\alp_t=1/\log\log\rd{2+\delta^{-1}t}\le 1/\log\log\rd{t/(2\delta)}$ and consider the family of functions
\ben
u_t\rd{\zeta}=\max\rd{1-\alp_t\log\log\rd{2+\rho^{-1}},0}\,,
\een
where $\rho=\vb{\zeta}$. We have that $u_t=0$ on $B(0,t^{-1}\delta)$, thus $G_t u_t\equiv 0$. Since $u_t$ is Lipschitz, it is in $L_1^2(\mbd)$. Since $1=u_t+(1-u_t)\le u_t+\alp_t \log\log(2+1/\rho)$, we have for $t$ large enough
\ben
\dbv{u_t}_{L^2(\mbd)}\ge \dbv{1}_{L^2(\mbd)}-\alp_t \dbv{\log\log\rd{2+1/\rho}}_{L^2(\mbd)}\ge \ov{2}\dbv{1}_{L^2(\mbd)}\,.
\een
On the other hand, we have $C>0$ such that
\ben
\dbv{du_t}_{L^2(\mbd)}\le \alp_t \dbv{d\log\log\rd{2+1/\rho}}_{L^2(\mbd)}\le C\alp_t\,.
\een
The estimate (\ref{eq:lam1tupperbound}) follows from the above and (\ref{eq:lam1variational}).

Suppose $u\in L_1^2(\mbd)$ satisfies
\ben
-\Delta u+G_t u=\lam_1 u
\een
By the elliptic estimate for $-\Delta+G_t$ (see e.g. \cite[Theorem 9.11]{GT83}) and the Sobolev inequalities, we have that for any $\Omega\Subset \mbd$, there are $C$ and $C'>0$ such that
\ben
\|u\|_{L_2^4(\Omega)}\le C\|\lam_1 u\|_{L^4(\mbd)}+\|u\|_{L^4(\mbd)} \le C' \|u\|_{L_1^2(\mbd)}<\infty
\een
We have that $u\in L_2^4(\Omega)\subset C^1(\Omega)$ by Sobolev embedding theorem (see \cite[(A.9) Appendix IV]{DK90}). By repeated application of the interior elliptic regularity estimates in $D-\pd B(0,\delta)$ where $G_t$ is smooth, we have that $u\in C^1(\mbd)\cap C^\infty(\mbd-\pd B(0,\delta))$.

For some $t\ge 1$, suppose $u\in C^1(\mbd)\cap C^\infty(\mbd-\pd B(0,\delta))$ is a real-valued eigenfunction of the first eigenvalue for the self-adjoint operator $-\Delta+G_t$ with Neumann boundary condition. By Courant's nodal domain theorem (see \cite[\S VI.6]{CH53}), $u$ has no node in $\mbd$, i.e., we may assume that $u>0$. It is easy to see that any function orthogonal to $u$ in $L^2(\mbd)$ cannot have a definite sign. Therefore the eigenspace of $\lam_1$ is one-dimensional. As a consequence, it follows from the rotational symmetry of $G_t$ that $u$ is a radial function. Let $v$ be given by
\ben
u(\zeta)=v(|\zeta|)\,.
\een
We have that $v$ solves the boundary value problem on $\rho\in [0,1]$:
\be \label{eq:SchroedingerODE}
\begin{cases}
v''+\rho^{-1}v'-G_t u=-\lam u \tx{ on }[0,1]\\
v'(1)=0
\end{cases}
\ee
As a result, the lowest eigenvalue of (\ref{eq:SchroedingerODE}) is $\lam_1(t)$. Assume without loss of generality $\delta=1$. Note for $t>1$, (\ref{eq:SchroedingerODE}) is given piecewise by 
\begin{align*}
& v''+\rho^{-1}v'-\rd{A t^2-\lam }u=0 \tx{ on }[0,t^{-1}]\\
& v''+\rho^{-1}v'+\lam u = 0\tx{ on }[t^{-1},1]
\end{align*}
For $t\gg 1$, we have $\lam_1<A t^2$, and there are $c_0$, $c_1$, $c_2$ such that
\begin{align*}
& v(\rho)=c_0 I_0\rd{\sqrt{A t^2-\lam}\rho} \tx{ for } 0\le \rho\le t^{-1}\\
& v(\rho)=c_1 J_0\rd{\sqrt{\lam}\rho}+c_2 Y_0\rd{\sqrt{\lam}\rho} \tx{ for }t^{-1}\le \rho\le 1
\end{align*}
where $J_n$ and $Y_n$ are Bessel functions of the first and second kind, and $I_n$ is the modified Bessel function of the first kind. We refer to \cite[Chapter 9]{AS72} for their definition and properties. Note that $I_0$ and $K_0$ form a basis of the space of solutions on $(0,t^{-1}]$, and $K_0$ is unbounded at the origin. The continuity of $v$, $v'$ at $\rho=1/t$, as well as the boundary condition $v'(1)=0$, implies
\begin{align*}
& c_0 I_0\rd{\sqrt{A t^2-\lam} t^{-1}}=c_1 J_0\rd{\sqrt{\lam}t^{-1}}+c_2 Y_0\rd{\sqrt{\lam}t^{-1}} \\
& c_0\sqrt{A t^2-\lam}I_1\rd{\sqrt{A t^2-\lam}t^{-1}}=-c_1\sqrt{\lam}J_1\rd{\sqrt{\lam}t^{-1}}-c_2\sqrt{\lam}Y_1\rd{\sqrt{\lam}t^{-1}} \\
& c_1 J_1\rd{\sqrt{\lam}}+c_2 Y_1\rd{\sqrt{\lam}}=0\,.
\end{align*}
The coefficient matrix for $c_0$, $c_1$, $c_2$ is singular iff $\lam$ is an eigenvalue. Thus the eigenvalues of (\ref{eq:SchroedingerODE}) are precisely the zeros of
\begin{align*}
&\delta_t(\lam):=\det\pmt{
0 & J_1\rd{\sqrt{\lam}} & Y_1\rd{\sqrt{\lam}} \\
-I_0\rd{\sqrt{A t^2-\lam}\,t^{-1}} & J_0\rd{\sqrt{\lam}\,t^{-1}} & Y_0\rd{\sqrt{\lam}\,t^{-1}} \\
\sqrt{A t^2-\lam}I_1\rd{\sqrt{A t^2-\lam}\,t^{-1}} & \sqrt{\lam}J_1\rd{\sqrt{\lam}\,t^{-1}} & \sqrt{\lam}Y_1\rd{\sqrt{\lam}\,t^{-1}}
} \\
&=\sqrt{\lam}I_0\rd{\sqrt{A-\lam t^{-2}}} g_t(\lam)+t\sqrt{A-\lam t^{-2}}I_1\rd{\sqrt{A-\lam t^{-2}}}f_t(\lam)
\end{align*}
where
\begin{align*}
& g_t(\lam):=J_1\rd{\sqrt{\lam}}Y_1\rd{\sqrt{\lam}\,t^{-1}}-J_1\rd{\sqrt{\lam}\,t^{-1}}Y_1\rd{\sqrt{\lam}} \\
& f_t(\lam):=J_1\rd{\sqrt{\lam}}Y_0\rd{\sqrt{\lam}\,t^{-1}}-J_0\rd{\sqrt{\lam}\,t^{-1}}Y_1\rd{\sqrt{\lam}}
\end{align*}

Using asymptotics of the relevant Bessel functions (see \cite[Chapter 9, (9.1.7), (9.1.8) and (9.1.9)]{AS72})
\ben
J_1(x)\sim \frac{x}{2},\,\, Y_0(z)\sim \frac{2}{\pi}\log\frac{x}{2},\,\, J_0(x)\sim 1,\,\, Y_1(x)\sim -\frac{2}{\pi}\ov{x}
\een
where $f_1(x)\sim f_2(x)$ if $f_1/f_2\to 1$ as $x\to 0$. We have that for $\lam>0$ small enough and $t\ge 1$,
\begin{align*}
J_1\rd{\sqrt{\lam}}Y_0\rd{\sqrt{\lam}t^{-1}}\ge 2\rd{\frac{\sqrt{\lam}}{2}}\rd{\frac{2}{\pi}\log\frac{\sqrt{\lam}t^{-1}}{2}}=2\frac{\sqrt{\lam}}{\pi}\log\frac{\sqrt{\lam}}{2t} \\
 -J_0\rd{\sqrt{\lam}t^{-1}}Y_1\rd{\sqrt{\lam}}\ge \ov{2}\cdot 1\cdot\rd{\frac{2}{\pi}\ov{\sqrt{\lam}}}=\ov{\pi\sqrt{\lam}} \\
 J_1\rd{\sqrt{\lam}}Y_1\rd{\sqrt{\lam}t^{-1}}\ge -2\rd{\frac{\sqrt{\lam}}{2}}\rd{\frac{2}{\pi}\ov{\sqrt{\lam}t^{-1}}}=-\frac{2t}{\pi} \\
 -J_1\rd{\sqrt{\lam}t^{-1}}Y_1\rd{\sqrt{\lam}}\ge \ov{2}\rd{\frac{\sqrt{\lam}t^{-1}}{2}}\rd{\frac{2}{\pi}\ov{\sqrt{\lam}}}=\ov{2\pi t}\,,
\end{align*}
therefore
\begin{align*}
& f_t(\lam)\ge \ov{\pi\sqrt{\lam}}\rd{1+2\lam\log\rd{\frac{\sqrt{\lam}}{2 t}}} \\
& g_t(\lam)\ge -\frac{2}{\pi}t+\ov{2\pi t}\ge -t\,.
\end{align*}
It is also easy to see that for $\lam>0$ small enough, and $t\ge 1$ we have some $C,C'>0$ with
\ben
I_0\rd{\sqrt{A-\lam/t^2}}\le C,\,\, \sqrt{A-\lam/t^2}I_1\rd{\sqrt{A-\lam/t^2}}\ge C'
\een
Therefore for $t\ge 1$ and small enough $\lam$, we have
\begin{align*}
& \delta_t(\lam)\ge -C t\sqrt{\lam}+ C'\frac{t}{\pi\sqrt{\lam}}\rd{1+2\lam \log\frac{\sqrt{\lam}}{2 t}} \\
&=\frac{2 C' t\sqrt{\lam}}{\pi}\rd{\ov{2\lam}+\ov{2}\log \lam-\log(2t)-\frac{\pi C}{2 C'}} \\
&\ge \frac{2 C' t\sqrt{\lam}}{\pi}\rd{\ov{2\lam}-\log \rd{2 C'' t}}
\end{align*}
where $C''=e^{-\pi C/(4 C')}$. Form (\ref{eq:lam1tupperbound}), we see that for $t\gg 1$, the above inequality holds with $\lam=\lam_1(t)$, therefore
\ben
\ov{2\lam_1(t)}-\log\rd{2 C''t}\le \frac{\pi}{2C't\sqrt{\lam_1(t)}}\delta_t\rd{\lam_1\rd{t}}=0\,,
\een
and we have that for $t$ large enough,
\ben
\lam_1\rd{t}\ge \ov{2\log \rd{2C'' t}}\,.
\een
\end{proof}


We next prove the $L^2$ lower bound of $Q_t=Q_t^{(h)}$ with $h=h_t^{\tx{app}}$ is defined in (\ref{eq:defnQt}). For this we will need the following three elementary lemmas. We omit the proofs except for one of them.

\begin{lem} \label{lem:singularvalueineq}
Let $x_1,\ldots,x_{N_1}\in L^2(\Omega)$ and $A=\rd{A_{ij}}$ $N_2\times N_1$ with entries in $L^\infty$ such that for a.e.~$x\in \Omega$, the singular values of $A$ satisfy $\lam_j\le 1$ for all $j$. Then
\ben
\sum_{j=1}^{N_1}\|x_j\|_{L^2(\Omega)}^2\ge \sum_{j=1}^{N_2}\left\|\sum_{\ell=1}^{N_1}A_{j\ell}x_\ell\right\|_{L^2(\Omega)}^2
\een
\end{lem}

\begin{lem} \label{lem:boundsqrt}
There is $C>0$ such that for $A$, $B>0$, $2\times 2$ Hermitian matrices with $|A-B|\le|A|/2$, $|A|\le L$, we have $|A^{1/2}-B^{1/2}|\le C|A-B|$.
\end{lem}

Recall that for $A$ positive-definite Hermitian, we denote by $A^{1/2}$ the unique positive-definite Hermitian square root. We have:

\begin{lem} \label{lem:offdiagofsqrt}
Let $A$ be $2\times 2$ positive-definite Hermitian and $B=A^{1/2}$. Then,
\ben
B_{12}=\frac{A_{12}}{\tr\, B}
\een
\end{lem}

\begin{proof}
Let $a=a_0I+a_1\sigma_1+a_2\sigma_2+a_3\sigma_3$ where $\sig_j$ is defined in (\ref{eq:pauli}) and $a_j\in \mbr$, such that 
\ben
A=e^a=e^{a_0}\rd{\rd{\cosh r}I+\frac{\sinh r}{r}\rd{a_1\sigma_1+a_2\sigma_2+a_3\sigma_3}}
\een
where $r=\sqrt{a_1^2+a_2^2+a_3^2}$. In particular, $A_{12}=e^{a_0}\rd{a_1-ia_2}\sinh(r)/r$. We also have $B=e^{a/2}$ and $\tr\, B=2 e^{a_0/2}\cosh(r/2)$. The conclusion follows since $B_{12}=e^{a_0/2}(a_1-ia_2)\sinh(r/2)/r$.
\end{proof}

\begin{prop} \label{prop:LtL2lowbd}
There is $C>0$ such that for $u\in L_1^2$, $t\gg 1$
\ben
Q_t(u)\ge \frac{C}{\log t}\|u\|_{L^2,h_t^{\tx{app}}}^2\,.
\een
\end{prop}

\begin{proof}
Recall the $t$-family of tuple of weights $\ulam=\ulam(t)\in \mci$ in Def \ref{def:htapp} defined for $t$ large enough with $\ulam\in \mci$ and $\lam_j(t)\in I_j$. We say a constant in an inequality is uniform in $\lam$ if the inequality holds for all $\lam\in \cup_j I_j$.

We first focus on the terms of $Q_t(u)$ related to the Higgs field. Set
\be \label{eq:QtHiggs}
Q_t^{\tx{Higgs}}(u)=4t^2\|\hat u\circ \bet\|_{L^2,h_t^{\tx{app}}}^2+4t^2\|\gam\circ\hat u\|_{L^2,h_t^{\tx{app}}}^2\,.
\ee
Write $Q_{t,\Omega}^{\tx{Higgs}}(u)$ for the above with $L^2$ replaced by $L^2(\Omega)$. We have
\ben
Q_t^{\tx{Higgs}}(u)=Q_{t,X-\coprod_{j=1}^{4g-4} \mbd_j'}(u)+\sum_{j=1}^{4g-4}Q_{t,\mbd_j'}(u)\,.
\een
Consider first $\mbd_j'$ for $p_j\in D_r$. We construct an $h_t^{\tx{app}}$-unitary frame and decompose $u$ with respect to it. We then bound $Q_{t,\mbd_j'}^{\tx{Higgs}}(u)$ from below by terms of the form $\dbv{f_\ell u_\ell}_{L^2}^2$ with $u_\ell$ components in the decomposition. For $t\gg 1$, the lower bound for $f_\ell$ will be $t$-independent for all but one term. Similarly, $t$-independent lower bounds will be obtained for $Q_{t,\mbd_j'}^{\tx{Higgs}}(u)$ with $p_j\in D_\bet$, $D_\gam$ as well as $Q_{t,X-\coprod_j {\mbd_j'}}^{\tx{Higgs}}(u)$ for all but one term. The lower bounds for the remaining terms and the $L^2$-norm of its derivative will follow from Lemma \ref{lem:schroedingerneumann}.

Let $p_j\in D_r$ and set $\zeta=\zeta_j$, $\lam=\lam_j$. Let $N_\lam$ (resp. $N_{\infty,\lam}$) be the unique positive definite Hermitian matrix satisfying
\be \label{eq:NlamNinftylam}
N_\lam^2=M_\lam,\,\, \tx{ resp. }N_{\infty,\lam}^2=M_{\infty,\lam}\,.
\ee
where $M_\lam$ (resp. $M_{\infty,\lam}$) are defined in (\ref{eq:htlaminetafrm}) (resp. Prop \ref{prop:asympsum}). Similar to $\widetilde H_{t,\lam}^{\tx{int}}$ and $\widetilde H_{t,\lam}^{\tx{ext}}$ in (\ref{eq:tilHtint}) and (\ref{eq:tilHtext}), set
\begin{align}
& E_{t,\lam}^{\tx{int}}=t^{-1/3}G_\lam(\chi)^{-1}N_{\lam}(t^{2/3}\chi \rho) \tx{ on }\mbd_j' \nonumber \\
& E_{t,\lam}^{\tx{ext}}=t^{-1/3}N_{\infty,\lam}(t^{2/3} \rho) \tx{ on }\mbd_j-\mbd_j' \label{eq:defEtlam}
\end{align}
where $\rho=\vb{\zeta}$ and $\chi$, $G_\lam$ are defined in \S \ref{sec:approxsoln}. Note that $E_{t,\lam}^{\tx{ext}}$ can be extended to $\mbd_j'$ where we have $E_{t,\lam}^{\tx{ext}}\rd{\zeta}=t^{-1/3}G_\lam(\chi)N_{\infty,\lam}(t^{2/3}\chi \rho)$. The piece-wise function
\be \label{eq:defElamapp}
E_{t,\lam_j(t)}^{\tx{app}}=\begin{cases}
E_{t,\lam_j(t)}^{\tx{int}} &\tx{ on }\mbd_j'\\
E_{t,\lam_j(t)}^{\tx{ext}} &\tx{ on }\mbd_j-\mbd_j'
\end{cases}
\ee
is $C^2$ at $\pd \mbd_j'$. Let $e_{i,j,t}=\sum_{k=1}^2 \rd{E_{t,\lam_j(t)}^{\tx{app}}}_{ki}\sig_{k,j}^{\ulam(t)}$ for $1\le j\le 4g-4$ and $i=1,2$ where the frame $\usig_j^{\ulam(t)}$ is given in Def \ref{def:weightdependentframes}. Note that on $\mbd_j-\mbd_j'$ we have
\begin{align*}
& e_{1,j,t}=(1/4)c_{\lam_j(t)}\rho^{2\lam_j(t)+1/2}t^{4\lam_j/3}\sig_{1,j}^{\ulam(t)} \\
& e_{2,j,t}=2 c_{\lam_j(t)}^{-1/2}\rho^{-\lam_j(t)+1/2}t^{-2\lam_j/3}\sig_{2,j}^{\ulam(t)}\,.
\end{align*}
We have that $\ule_{j,t}$ is an $h_t^{\tx{app}}$-unitary frame of $F$ over $\mbd_j$. Let $\widetilde\rho=t^{2/3}\chi^{-1}\rho$ and set
\begin{align}
& n_{ij,\lam}(\widetilde \rho)=\rd{N_\lam}_{ij}\rd{\widetilde \rho},\,\, j=1,2\nonumber \\
& n_\lam(\widetilde\rho)=\widetilde\rho^{-1}\det N_\lam\rd{\widetilde\rho}\,. \label{eq:defnij}
\end{align}
Let $\bet=\bet_0d\zeta$ and $\gam=\gam_0 d\zeta$. On $\mbd_j'$ with $\lam_j=\lam_j(t)$,
\begin{align}
&\rd{\bet_0}_{\ule_{j,t}}=t^{-1/3}\chi^{1/2}n_{\lam_j}\pmt{\overline{n_{12,\lam_j}} & n_{22,\lam_j}}\,,\nonumber \\
&\rd{\gam_0}_{\ule_{j,t}}=t^{-1/3}\chi^{1/2}n_{\lam_j}^{-2}\pmt{-n_{12,\lam_j} & n_{11,\lam_j}}^T\,. \label{eq:betgamonf}
\end{align}
For $u\in \tx{Herm}(F,h_t^{\tx{app}})$, we write
\be \label{eq:decompofherm}
\rd{u}_{\ule_{j,t}}=u_0\sigma_0+u_1\sigma_1+u_2\sigma_2+u_3\sigma_3
\ee
where
\be \label{eq:pauli}
\sigma_0=\pmt{ 2 & \\ & -1},\,\, \sigma_1=\pmt{ & 1 \\ 1 & },\,\, \sigma_2=i\pmt{ & -1 \\ 1 & },\,\, \sigma_3=\pmt{1 & \\ & -1}\,.
\ee
Note that $\sig_1$, $\sig_2$, $\sig_3$ are the Pauli matrices (a basis of the space of $2\times 2$ trace-free Hermitian matrices) and by comments at the end of \S \ref{sec:linear}, $\sig_0$ spans the nontrivial ker $L_t$ when restricted to $X-D$.

Recall that we have $g_X=d\zeta_j\cdot d\bar\zeta_j$ on $\mbd_j$, therefore $\vb{d\zeta}^2=2$. Thus $\|\bet \hat u\|_{h_t^{\tx{app}}}^2=2\|\bet_0\hat u\|_{h_t^{\tx{app}}}^2$ and $\|\hat u \gam\|_{h_t^{\tx{app}}}^2=2\|\hat u\gam_0\|_{h_t^{\tx{app}}}^2$. We have from the above decomposition,
\be \label{eq:defnQtHiggs}
Q_{t,\mbd_j'}^{\tx{Higgs}}(u)=4t^2\rd{\|\hat u\bet\|_{L^2(\mbd_j'),h_t^{\tx{app}}}^2+\|\gam\hat u\|_{L^2(\mbd_j'),h_t^{\tx{app}}}^2}=4t^{4/3}\rd{\sum_{\ell=1}^4\|S_{\ell,j}\|_{L^2(\mbd_j')}^2}
\ee
where
\begin{align*}
& S_{1,j}=\chi^{1/2}\sq{\rd{-n_{\lam_j}^{-2}n_{12,\lam_j}}\rd{3u_0+u_3}+\rd{n_{\lam_j}^{-2}n_{11,\lam_j}}\rd{u_1-iu_2}}  \\
& S_{2,j}=\chi^{1/2}\sq{\rd{n_{\lam_j} n_{12,\lam_j}}\rd{3u_0+u_3}+\rd{n_{\lam_j} n_{22,\lam}}\rd{u_1-iu_2}}  \\
& S_{3,j}=\chi^{1/2}\sq{\rd{n_{\lam_j}^{-2}n_{11,\lam_j}}\rd{-u_3}+\rd{-n_{\lam_j}^{-2}\overline{n_{12,\lam_j}}}\rd{u_1+iu_2}}  \\
& S_{4,j}=\chi^{1/2}\sq{\rd{n_{\lam_j} n_{22,\lam_j}}\rd{-u_3}+\rd{n_{\lam_j} \overline{n_{12,\lam_j}}}\rd{u_1-iu_2}}
\end{align*}
Let
\ben
B=\pmt{
\chi^{-1/2} n_\lam^3 & \chi^{-1/2} & 0 & 0 \\
0 & 0 & \chi^{-1/2} n_\lam^3 & \chi^{-1/2} \\
-n_\lam^3 n_{22,\lam} & n_{11,\lam} & n_\lam^3 n_{12,\lam} & n_{12,\lam}
}
\een
The matrix $B^\ast B$ has characteristic polynomial
\ben
\tx{char}_{ B^\ast B}(t)=t\rd{t-\chi^{-1}\rd{1+n_\lam^6}}\rd{t^2-\nu_\lam^{-2} t+\chi^{-1}\rd{n_{11,\lam}+n_{22,\lam}}^2n_{\lam}^6}
\een
where
\be \label{eq:defnulam}
\nu_\lam=\rd{\chi^{-1}\rd{1+n_\lam^6}+\rd{n_{11,\lam}^2+|n_{12,\lam}|^2+n_\lam^6\rd{n_{22,\lam}^2+|n_{12,\lam}|^2}}}^{-1/2}
\ee
The largest eigenvalue of $B^\ast B$ is bounded above by $1/\nu_\lam^2$. Let $A=\nu_\lam B$, by direct calculation,
\ben
\sum_{\ell=1}^4A_{k \ell}S_\ell=
\begin{cases}
\nu_\lam n_\lam(n_{11,\lam}+n_{22,\lam})(u_1-iu_2) & k = 1,\\
\sum_{\ell=1}^4A_{2\ell}S_\ell=-\nu_\lam n_\lam(n_{11,\lam}+n_{22,\lam})u_3 & k =2,\\
\sum_{\ell=1}^4A_{3\ell}S_\ell=3\nu_\lam \chi^{1/2}n_\lam n_{12,\lam}\rd{n_{11,\lam}+n_{22,\lam}} u_0 & k=3\,.
\end{cases}
\een
Note that $\chi\equiv 1$ on $\mbd_j''$, we have by Lemma \ref{lem:singularvalueineq}
\begin{align}
& Q_{t,\mbd_j'}^{\tx{Higgs}}(u)\ge 4t^{4/3}\left\| \nu_\lam n_\lam (n_{11,\lam}+n_{22,\lam})(u_1-iu_2)\right\|_{L^2(\mbd_j')}^2 \nonumber \\
& +4t^{4/3}\left\| \nu_\lam n_\lam (n_{11,\lam}+n_{22,\lam}) u_3\right\|_{L^2(\mbd_j')}^2 + 4t^{4/3}\left\| 3n_{12,\lam} \nu n_\lam \rd{n_{11,\lam}+n_{22,\lam}}u_0\right\|_{L^2(\mbd_j'')}^2 \label{eq:QtHigbd01}
\end{align}
where $\lam=\lam_j(t)$.

Next, we provide a lower bound for the positive function $\nu_\lam n_\lam \rd{n_{11,\lam}+n_{22,\lam}}$ on $\mbd_j'$ uniform in $\lam$. Fix $\rho_1\in [R/3, 2R/3)$ and consider $\rho\in (\rho_1,2R/3)$ and $(0,\rho_1]$ separately. For the first region, we use the asymptotics in Pro.\ref{prop:asympsum}. For the second region, relatively compact in $\mbd_j'$, we use continuity in $\lam$ in Cor \ref{cor:metriccont}. By Prop \ref{prop:asympsum} and Lemmas \ref{lem:matrest}, \ref{lem:boundsqrt}, there are $C$, $c>0$ such that for $\rho\gg 1$ and $\lam\in I$,
\be \label{eq:nlamasymp}
\vb{N_\lam-N_{\infty,\lam}}\le C e^{-c\rho}
\ee
Note that $N_\infty\rd{\rho}=\tx{diag}\rd{c_{\lam}\rho^{2\lam+1/2}/4,2 c_{\lam}^{1/2}\rho^{-\lam+1/2}}$. For $\rho\gg 1$, $t\ge 1$, we have $\chi^{-1}=t^{-2/3}\rho^{-1}\widetilde\rho\le \rho_0^{-1}\widetilde \rho$. Therefore from (\ref{eq:defnulam}),
\ben
\nu_\lam^{-2}\le \rho_0^{-1}\widetilde\rho\rd{1+n_\lam^6}+\rd{n_{11,\lam}^2+|n_{12,\lam}|^2+n_\lam^6\rd{n_{22,\lam}^2+|n_{12,\lam}|^2}}
\een
By (\ref{eq:nlamasymp}), there are $\widetilde \rho_0$, $C_1$, $C_2>0$ uniform in $\lam$ such that for all $\widetilde\rho\ge \widetilde \rho_0$,
\ben
\nu_\lam\ge C_1\widetilde\rho^{-\ov{2}-3a_\lam},\,\,\, n_\lam(n_{11,\lam}+n_{22,\lam})\ge C_2\widetilde\rho^{\ov{2}+3a_\lam} 
\een
where $a_\lam=\max(0,\lam)$. For $\lam\in I$ we have 
\be \label{eq:nunn11n22bound01}
\nu_\lam n_\lam(n_{11,\lam}+n_{22,\lam})\ge C_1C_2>0\,\,\tx{ for }\,\,\widetilde\rho\ge \widetilde\rho_0\,.
\ee

Note that $\rho\mapsto \widetilde\rho=t^{2/3}\rho\chi\rd{\rho}^{-1}$ gives a diffeomorphism $(0,2R/3)\xrightarrow{\sim}\mbr_{>0}$. Let $\rho_1$ be such that $\widetilde\rho\rd{\rho_1}=\widetilde\rho_0$, for $0\le \rho\le \rho_1$, there is a constant $\chi_0>0$ such that $\chi\ge \chi_0$. We have
\be \label{eq:nm2bd}
\nu_\lam^{-2}\le \chi_0^{-1}\rd{1+n_\lam^6}+\rd{n_{11,\lam}^2+|n_{12,\lam}|^2+n_\lam^6\rd{n_{22,\lam}^2+|n_{12,\lam}|^2}}\,.
\ee
By a direct calculation using (\ref{eq:Mlam2H1lam}), $\tr\, M_\lam(\rho)=\rho^{-2}\rd{H_{1,\lam}}_{11}+\rd{H_{1,\lam}}_{22}$, $\det M_\lam(\rho)=\rho^{-2}\det H_{1,\lam}$. Let $x_1,x_2>0$ be the two eigenvalues of $M_\lam$, then
\begin{align}
&\rd{n_{11,\lam}(\rho)+n_{22,\lam}(\rho)}^2=\rd{x_1^{-1/2}+x_2^{-1/2}}^2\nonumber \\
&=\frac{x_1+x_2}{x_1x_2}+\frac{2}{\sqrt{x_1x_2}}=\ov{\det H_{1,\lam}}\rd{\rd{H_{1,\lam}}_{11}+2\rd{\det H_{1,\lam}}^{1/2}\rho+\rd{H_{1,\lam}}_{22}\rho^2} \label{eq:thetwoeigvals}
\end{align}
and $n_\lam(\rho)=\rd{\det H_{1,\lam}}^{-1/2}$. By Cor \ref{cor:metriccont}, there is $C_3>0$ uniform in $\lam$ such that $n_\lam\rd{n_{11,\lam}+n_{22,\lam}}\ge C_3$ for $\widetilde\rho\le \widetilde\rho_0$. By the same corollary, $n_\lam$ is bounded above uniformly for $\lam\in I=I_j$. Therefore there is $C_3'>0$ uniform in $\lam$ such that on $\widetilde\rho\le \widetilde\rho_0$, $\nu_\lam^{-2}\le C_3'(n_{11,\lam}^2+n_{22,\lam}^2+|n_{12,\lam}|^2+1)$. Furthermore, by $\det N_\lam>0$, we have $|n_{12,\lam}|^2<n_{11,\lam}n_{22,\lam}\le(n_{11,\lam}^2+n_{22,\lam}^2)/2$. Therefore there is $C_4>0$ uniform in $\lam$ such that $\nu_\lam\ge C_4$ for $\widetilde\rho\le \widetilde\rho_0$. Therefore
\ben
\nu_\lam n_\lam\rd{n_{11,\lam}+n_{22,\lam}}\ge C_3C_4\,\, \tx{ for }\,\,\widetilde\rho\le \widetilde\rho_0\,.
\een
Consequently, on $\mbd_j'$
\ben
\nu_\lam n_\lam\rd{n_{11,\lam}+n_{22,\lam}}\ge C_5=\min\cl{C_1C_2,C_3C_4}
\een
where $C_5$ uniform in $\lam$.

By (\ref{eq:QtHigbd01}) and the above estimate, we have for $t\ge 1$, 
\be \label{eq:QtHigbd02}
Q_{t,\mbd_j'}^{\tx{Higgs}}\ge 4C_5^2\rd{\|u_1\|_{L^2(\mbd_j')}^2+\|u_2\|_{L^2(\mbd_j')}^2+\|u_3\|_{L^2(\mbd_j')}^2}+4C_5^2t^{4/3}\|3 n_{12,\lam_j(t)}u_0\|_{L^2(\mbd_j'')}^2
\ee
From (\ref{eq:Htpol}) and (\ref{eq:defMlamfirst}), we have
\ben
\rd{M_\lam^{-1}}_{12}(\rho)=\ov{2\det H_{1,\lam}}\rd{\rd{H_{1,\lam}}_{11}-2\rho\re \rd{f_{3,\lam}(\rho)}-\rho^2\rd{H_{1,\lam}}_{22}}
\een
By Cor \ref{cor:metriccont}, this gives a continues family in $C^0(\mbd_j')$ in $\lam$. As $H_{1,\lam}$ is regular at the origin, $f_{3,\lam}(\rho)\to 0$ as $\rho\to 0$. We have
\ben
\rd{M_{\lam}^{-1}}_{12}(\rho)\to \frac{\rd{H_{1,\lam}}_{11}(0)}{2\det H_{1,\lam}(0)}>0
\een
as $\rho\to 0$. By (\ref{eq:thetwoeigvals}) and Cor \ref{cor:metriccont}, $\tr \,N_\lam=n_{11,\lam}+n_{22,\lam}$ is bounded near the origin uniformly in $\lam$. By Lemma \ref{lem:offdiagofsqrt} and Cor \ref{cor:metriccont}, there are $A>0$ and $0<\delta<R/3$ uniform in $\lam$ such that for $\rho\le\delta$,
\ben
\vb{n_{12,\lam}(\rho)}=\frac{\vb{\rd{M_{\lam}^{-1}}_{12}}}{n_{11,\lam}+n_{22,\lam}}\ge A\,.
\een
For $\rho\le R/3$, $\chi\rd{\rho}=1$. Therefore with $t\ge 1$, $\widetilde\rho\le \delta$ iff $\rho\le t^{-2/3}\delta$. Define function $y_j$ with $y_j(\zeta_j)=A$ for $|\zeta_j|\le \delta$ and zero elsewhere. From (\ref{eq:QtHigbd02}), we have $C_6$ uniform in $\lam$,
\be \label{eq:QtDjprimebound1}
Q_{t,\mbd_j'}^{\tx{Higgs}}\ge C_6\rd{\|u_1\|_{L^2(\mbd_j')}^2+\|u_2\|_{L^2(\mbd_j')}^2+\|u_3\|_{L^2(\mbd_j')}^2}+C_6\sum_{p_j\in D_r}\rd{y_{j,t}u_0,u_0}_{L^2(X)}^2
\ee
where $y_{j,t}=t^{4/3}y_j\rd{t^{2/3}\zeta_j}$ on $\mbd_j'$ and zero elsewhere.

Next, we consider the term $Q_{t,\mbd_j'}^{\tx{Higgs}}$ with $p_j\in D_\gam$. Let $\uls_{j,t}^{\ulam}$ be the frame defined in (\ref{eq:defsijt}) and let $e_{1,j,t}=s_{2,j,t}^{\ulam(t)},\,\,\, e_{2,j,t}=a s_{1,j,t}^{\ulam(t)}$ where $a=\rho^{1/4}e^{\chi \psi_P/2}$ and $\psi_P$ is Painlev\'e function defined in \S \ref{sec:localmodel}. By (\ref{eq:HtintDbetgam}) and (\ref{eq:localformundersjtulam}), the frame $\ule_{j,t}=\cl{e_{1,j,t},e_{2,j,t}}$ is $h_t^{\tx{app}}$-unitary and $\rd{\bet_0}_{\ule_t}=\pmt{0 & a^2}$, $\rd{\gam_0}_{\ule_t}=\pmt{0 & a^{-2}\zeta}^T$. Write $u_{\ule_t}=u_0\sig_0+u_1\sig_1+u_2\sig_2+u_3\sig_3$ as in (\ref{eq:decompofherm}), we have
\ben
\rd{\bet_0\cdot\hat u}_{\ule_t}=\pmt{a^2\rd{u_1+iu_2}& -a^2 u_3},\,\, \rd{\hat u\cdot\gam_0}_{\ule_t}=\pmt{a^{-2}\zeta \rd{u_1-iu_2} & -a^{-2}\zeta u_3}^T
\een
It follows from a direct calculation that
\ben
Q_{t,\mbd_j'}^{\tx{Higgs}}(u)=4t^2\rd{ \dbv{f\,u_1}^2+\dbv{f\,u_2}^2 +\dbv{f\,u_3}^2}
\een
where $\|\cdot\|=\|\cdots\|_{L^2(\mbd_j')}$ and $f=\rd{2\rho\cosh\rd{2\chi \psi_P}}^{1/2}$. By \cite{MTW77}, there is an asymptotic expansion of the form
\ben
e^{-\psi(x)}\sim x^{1/3}\sum_{j=0}^\infty a_j x^{4j/3}\,\,\tx{ as }x\to 0\,.
\een
In particular there is $C_0>0$ such that $e^{2\psi_P(x)}\ge C_0x^{-2/3}$ for $x$ small enough. On the other hand, $\psi_P\to 0$ exponentially as $t\to\infty$. It follows that there is $C>0$ such that for $t\gg 1$, $f\ge Ct^{-1/3}$, therefore $t^2\dbv{f u_j}^2\ge C^2t^{4/3}\dbv{u_j}^2\ge C^2\dbv{u_j}^2$. There is $C'>0$ with
\be \label{eq:QtDjprimebound2}
Q_{t,\mbd_j'}^{\tx{Higgs}}(u)\ge C'\rd{\dbv{u_1}_{L^2(\mbd_j')}^2+\dbv{u_2}_{L^2(\mbd_j')}^2 +\dbv{u_3}_{L^2(\mbd_j')}^2}\,.
\ee
The same estimate holds for $p_j\in D_\bet$ case with an similar argument.

For the last region, consider $Q_{t,X-\coprod_j \mbd_j'}^{\tx{Higgs}}$. Let $\cl{\rd{W_\alpha;w_\alpha}}$ be a finite atlas of $X-\coprod_j \mbd_j'$ by charts with $q=\rd{dw_\alpha}^2$. As in the proof of Lemma \ref{lem:decoupledsoln} after identifying $F$ with $V=L^{-2}K_X\oplus LK_X$ via the Hecke modification, we have $\bet=\pmt{0 & q^{-1}}$, $\gam=\pmt{0 & q^2}^T$ over $W_\alpha$. On this region, $h_t^{\tx{app}}=\iota^\ast \rd{h_L^{-2}h_K}\oplus \rd{h_Lh_K}$ where $h_L=h_{L,\ulam(t),t}$ and $h_K$ both have flat Chern connections. It is easy to see that there is $\sig\in \mco_{W_\alp}$ such that $|\sigma|_{h_L}\equiv 1$. Let $e_1=\rd{\sigma^{-2}dw_\alpha,0}$, $e_2=\rd{0,\sigma dw_\alpha}$. The frame $\ule=\cl{e_1,e_2}$ is both holomorphic and $h_t^{\tx{app}}$-unitary. We write $u_{\ule}=u_0\sig_0+u_1\sig_1+u_2\sig_2+u_3\sig_3$ as in (\ref{eq:decompofherm}). From a direct calculation,
\ben
\rd{\bet_0\cdot\hat u}_{\ule}=\pmt{u_1+iu_2 & -u_3},\,\, 
\rd{\hat u\cdot\gam_0}_{\ule}=\pmt{u_1-iu_2 & -u_3}^T\,.
\een
We have
\be \label{eq:Qtbound3}
Q_{t,W_\alpha}^{\tx{Higgs}}(u)=8t^2\rd{\dbv{u_1}_{L^2(W_\alpha)}^2+\dbv{u_2}_{L^2(W_\alpha)}^2+\dbv{u_3}_{L^2(W_\alpha)}^2}\,.
\ee
and that
\be \label{eq:Qtbound4}
Q_{t,X-\coprod_j \mbd_j'}^{\tx{Higgs}}(u)\ge \sum_\alpha Q_{t,W_\alpha}^{\tx{Higgs}}(u)\,.
\ee 

From the above discussion on regions covering the entirety of $X$, we defined $t$-dependent frames $\ule_{j,t}$ over $\mbd_j'$ for $p_j\in D$ and $\ule$ over $W_\alpha\subset X-\mbd_j'$ providing a smooth decomposition of $F$ 
\be \label{eq:decomposef}
F\cong L_1\oplus L_2
\ee
such that over $X-\mbd_j'$, $L_1$ (resp. $L_2$) coincides with $L^{-2}K$ (resp. $LK$) summand of $V$ under the identification $\iota$ as in Theorem \ref{thm:hecke}. Let $\sig_0=\tx{diag}(2,-1)$ corresponding to (\ref{eq:decomposef}). There is an induced decomposition
\ben
\tx{End}F=\ag{\sig_0}\oplus \tx{End}_0 F\,.
\een
For $u\in \tx{Herm}(F,h_t^{\tx{app}})$, we write
\ben
u=\rd{\tr\,u }\sig_0+v
\een
with $\tr\,v=0$ which is compatible with (\ref{eq:decompofherm}) and similar decompositions on each region. By the estimates in (\ref{eq:QtDjprimebound1}), (\ref{eq:QtDjprimebound2}), (\ref{eq:Qtbound3}), and (\ref{eq:Qtbound4}), there is $C_7>0$ such that for all $t\gg 1$,
\ben
Q_t^{\tx{Higgs}}(u)\ge C_7\dbv{v}_{L^2,h_t^{\tx{app}}}^2+C_7\sum_{j=1}^{d_r}\rd{y_{j,t} \tr\,u,\tr\,u}_{L^2,h_t^{\tx{app}}}
\een
where $y_{j,t}$ is defined below (\ref{eq:QtDjprimebound1}). Recall from Prop \ref{prop:su12stab2} that the stability of $(F,\bet,\gam)$ implies that $D_r>0$. Fix an index $j_0$ with $p_{j_0}\in D_r$. For $t\gg 1$ there is $C_8>0$ such that
\be \label{eq:L2normbdsecondtolast}
Q_t(u)\ge C_8\rd{Q_t^{(1)}(\tr\, u)+\dbv{v}_{L^2,h_t^{\tx{app}}}^2}
\ee
with
\be \label{eq:Qt1}
Q_t^{(1)}(f)=\dbv{df}_{L^2}^2+ \rd{y_t f,f}_{L^2}
\ee
where $y_t=y_{j_0,t}$.

By the uniformization theorem, $X=\Gamma\backslash \mbh$ for $\Gamma$ a Fuchsian group, denote the quotient map by $p:\mbh\to X$. Furthermore, there exists a fundamental polygon $\Pi\in \mbh$ for the action of $\Gamma$ on $\mbh$ with finitely many edges and positive angles at vertices (see e.g. \cite{Sti92}). We can arrange so that a lift of $p_{j_0}$ is in the interior of $\Pi$. As a result there are finitely many $g\in \Gamma$ such that $g(\Pi)$ is adjacent to $\Pi$, denote the union of these with $\Pi$ by $\widetilde\Pi$. Consider $\Omega$ with $\pd\Omega$ smooth and $\Pi\Subset \Omega\Subset \widetilde\Pi$. By the Riemann mapping theorem, there is biholomorphism $r: \mbd\xrightarrow{\sim} \Omega$ mapping the origin to the unique point corresponding to $p_{j_0}$ under quotient in $\Pi$. By a result of Painlev\'e (see \cite{Bell90}), $r$ extends smoothly to $\pd \mbd$. The above identification corresponds to a constant curvature metric $g_X'$ on $X$. Note that both $g_X$ and $g_X'$ are independent of $t$, we will use them interchangeably to define norms. We use $\gtrsim$ (resp. $\sim$) to denote (in)equalities up to constants independent of $t$ and $\lam$ omitted. Let $\widetilde F_t:=\rd{p\circ r}^\ast y_t$ and let
\ben
Q_t^{(0)}(\phi)=\dbv{d\phi}_{L^2\rd{\mbd}}^2+\rd{\widetilde F_t\phi,\phi}_{L^2\rd{\mbd}}^2
\een
There is a non-negative function $F\in L^\infty\rd{\Omega}$ such that $\widetilde F_t\ge F_t$ with $F_t(\zeta)=t^2 F(t\zeta)$. By Lemma \ref{lem:schroedingerneumann} for $t\gg 1$
\ben
Q_t^{(0)}\rd{\phi}\gtrsim \rd{\log t}^{-1}\dbv{\phi}_{L^2\rd{\mbd}}^2
\een
Therefore,
\begin{align*}
&Q_t^{(1)}(f)\sim \dbv{d\rd{p^\ast f}}_{L^2\rd{\widetilde\Omega}}^2+\rd{\rd{p^\ast y_t \,f},p^\ast f}_{L^2(\widetilde\Omega)} \\
&\gtrsim \dbv{d\rd{p^\ast f}}_{L^2\rd{\Omega}}^2+\rd{\rd{p^\ast y_t \,f},p^\ast f}_{L^2(\Omega)} \gtrsim Q_t^{(0)}(\phi) \\
&\gtrsim (\log t)^{-1}\dbv{\phi}_{L^2(\mbd)}^2\gtrsim \rd{\log t}^{-1}\dbv{p^\ast f}_{L^2(\Omega)}^2 \\
&\gtrsim \rd{\log t}^{-1}\dbv{f}_{L^2}^2
\end{align*}
where $\phi=\rd{p\circ r}^\ast f$. Let $\zeta$ (resp. $z$) be a coordinate on $\mbh$ as the upper half-plane (resp. $\mbd$ as the unit disk) we have also used above the fact that for the map $r^{-1}: \zeta\mapsto z$, $\vb{\pd_\zeta z}^2$ is bounded above and below on $\overline{\mbd}$.

Combining (\ref{eq:L2normbdsecondtolast}) we have for $t\gg 1$,
\begin{align*}
& Q_t(u)\gtrsim \rd{\log t}^{-1}\rd{\dbv{\tr\, u}_{L^2}^2+\dbv{v}_{L^2,h_t^{\tx{app}}}^2}  \\
& \gtrsim \rd{\log t}^{-1}\rd{\dbv{\rd{\tr\, u}\sig_0}_{L^2,h_t^{\tx{app}}}^2+\dbv{v}_{L^2,h_t^{\tx{app}}}^2}\gtrsim \rd{\log t}^{-1}\dbv{u}_{L^2,h_t^{\tx{app}}}^2\,.
\end{align*}
\end{proof}

Note that the pointwise norm associated given by $\ag{\ag{\cdot,\cdot}}$ and $\ag{\cdot,\cdot}$ are mutually bounded. We have the following.

\begin{cor} \label{cor:mainL2est}
There is $C>0$ such that for $u\in L_2^2\rd{\tx{Herm}\rd{F,h_t^{\tx{app}}}}$ and $t\gg 1$
\ben
\dbv{L_t u}_{L^2,h_t^{\tx{app}}}\ge \frac{C}{\log t}\dbv{u}_{L^2,h_t^{\tx{app}}}\,.
\een
\end{cor}

\subsection{\texorpdfstring{$L_2^2$}{L\_2\^{}2} lower bound for \texorpdfstring{$L_t$}{Lt}}

Building on the inequality in Cor \ref{cor:mainL2est}, we will prove the following $t$-dependent elliptic estimate for $L_t$.

\begin{prop} \label{prop:L22toL2lowerbound}
Fix $t_0\ge 1$. There are $C>0$ such that for $u\in L_2^2\rd{\tx{Herm}\rd{F,h_t^{\tx{app}}}}$ and $t\gg 1$,
\ben
\dbv{L_t u}_{L^2,h_{t_0}^{\tx{app}}}^2\ge Ct^{-38} \dbv{u}_{L_2^2,h_{t_0}^{\tx{app}}}^2
\een
\end{prop}

The proof combines elliptic estimate of $\Delta_{h_{t_0}^{\tx{app}}}$ and bounds on $\dbv{\Delta_{h_t^{\tx{app}}}-\Delta_{h_{t_0}^{\tx{app}}}}$ and $\dbv{\cl{\psi,\hat u}}$. Note that in contrast to Cor \ref{cor:mainL2est}, the norms in the above inequality are given by fixed metric $h_{t_0}^{\tx{app}}$. We begin by proving a comparison result (Lemma \ref{lem:htappcomp}) between the norms induced by $\dbv{\cdot}_{h_{t_0}^{\tx{app}}}$ and $\dbv{\cdot}_{h_t^{\tx{app}}}$. For an $r\times r$ matrix $M$ and positive-definite Hermitian matrix $H$ of the same dimension, recall the norm $\vb{M}_H^2=\tx{tr}\rd{MH^{-1}M^\ast H}$ defined in (\ref{eq:Hnorm}). We will need the following three comparison lemmas of matrix norms associated to different Hermitian matrices.

\begin{lem} \label{lem:hermcomp0}
Let $H>0$ be a $2\times 2$ Hermitian matrix with eigenvalues $0<\alp_1\le \alp_2$, then for $L>0$ the following are equivalent
\begin{itemize}
\item For any $A$, $L^{-1}\vb{A}_I^2\le \vb{A}_H^2\le L\vb{A}_I^2$
\item $\alp_2/\alp_1\le L$
\end{itemize}
\end{lem}

\begin{lem} \label{lem:hermcomp1}
There is $\eps_0>0$ and $C\ge 1$ such that for $2\times 2$ Hermitian matrices $H_0$, $H_1>0$ with
\ben
\vb{H_1^{-1/2}\rd{H_0-H_1}H_1^{-1/2}}\le \eps_0
\een
we have for a $2\times 2$ matrix $A$,
\ben
\vb{A}_{H_0}^2\le C \vb{A}_{H_1}^2\,.
\een
\end{lem}

\begin{proof}
We have for $A\neq 0$, $\vb{A}_{H_0}^2/\vb{A}_{H_1}^2=\vb{B}_M/\vb{B}_I^2$ where $B=H_1^{1/2}A H_1^{-1/2}$ and $M=H_1^{-1/2}H_0H_1^{-1/2}$. There are $\eps_0$, $C_0>0$ such that for any $2\times 2$ matrix $X$ with $\vb{X-I}<\eps_0$ we have $\vb{X^{-1}-I}<C_0\eps_0$. Suppose $\vb{M-I}<\eps_0$ and without loss of generality that $\vb{B}_I^2=1$. There is $C_1>0$ such that
\ben
\vb{\vb{B}_M^2-\vb{B}_I^2}\le \vb{\tr\rd{B\rd{M^{-1}-I}B^\ast M}}+\vb{\tr\rd{B B^\ast \rd{M-I}}}\le C_1\eps_0\,.
\een
Thus there is $C_2>0$ such that $\vb{B}_M^2\le C_2$.
\end{proof}

\begin{lem} \label{lem:hermcomp2}
Let $H>0$ be a $2\times 2$ Hermitian matrix and $0<\alp_1\le \alp_2$ its two eigenvalues. There is $C\ge 1$ such that for all $A$ and $T$, $2\times 2$ matrices with $T$ non-singular, we have
\ben
\frac{\alp_1^2}{C\alp_2^2}\frac{\vb{A}_H^2}{\vb{T}^2\vb{T^{-1}}^2}\le \vb{A}_{T^\ast H T}^2\le \frac{C\alp_2^2}{\alp_1^2}\vb{T}^2\vb{T^{-1}}^2\vb{A}_H^2
\een
\end{lem}

\begin{proof}
We have with $B=H^{1/2}A H^{-1/2}$,
\ben
\frac{\vb{A}_{T^\ast H T}^2}{\vb{A}_H^2}=\frac{\vb{H^{1/2}T H^{-1/2} B H^{1/2}T^{-1}H^{-1/2}}_I^2}{\vb{B}_I^2}\le \vb{H^{1/2}}_I^4 \vb{H^{-1/2}}_I^4 \vb{T}_I^2\vb{T^{-1}}_I^2
\een
Note that $\vb{H^{1/2}}_I^2=\tr\, H=\alp_1+\alp_2\le 2\alp_2$ and $\vb{H^{-1/2}}_I^2=\tr\, H^{-1}=\alp_1^{-1}+\alp_2^{-1}\le 2\alp_1^{-1}$. The conclusion follows from similar argument as in Lemma \ref{lem:hermcomp1}.
\end{proof}

\begin{lem} \label{lem:htappcomp}
There is $C\ge 1$ such that for $u\in \tx{End}(F)$ and $t\gg 1$,
\ben
C^{-1} t^{-13}\vb{u}_{h_{t_0}^{\tx{app}}}^2\le \vb{u}_{h_t^{\tx{app}}}^2\le C t^{13}\vb{u}_{h_{t_0}^{\tx{app}}}^2
\een
\end{lem}

\begin{proof}
We consider four types of regions covering $X$ together with local frames of $F$ and prove the claim by studying local form of $h_t^{\tx{app}}$:
\begin{itemize}
\item (1) for $p_j\in D_r$, $\widetilde \mbd_{j,t}=\cl{\rho<\rho_0t^{-2/3}}$ with $\rho_0$ as in Prop \ref{prop:asympsum}; 
\item (2) for $p_j\in D_r$, $\mbd_j'-\widetilde \mbd_{j,t}$, 
\item (3) $\mbd_j'$ for $p_j\in D_\bet$ or $D_\gam$, and 
\item (4) $X-\coprod_j \mbd_j'$.
\end{itemize}
On (1), (2), (3), we use holomorphic frame $\uls_j^{(0)}$ defined in \S \ref{sec:htapp}. We cover (4) by a finite atlas over which $L$ and $K$ are trivialized and take a unitary frame with respect to the metric $\iota^\ast\rd{h_L^{-2}h_K\oplus h_L h_K}$ where $h_L=h_{L,\ulam(t_0)}^0$. By Lemma \ref{lem:hermcomp0}, it suffices bound the ratio $\alp_2(H)/\alp_1(H)$ where $0<\alp_1(H)\le \alp_2(H)$ are the eigenvalues of $H$ the local form of $h_t^{\tx{app}}$.

{\it Region (1)}: Note that for $t\gg 1$, $\rho_0t^{-2/3}<R/3$ and $\widetilde \mbd_{j,t}\subset \mbd_j''$. From (\ref{eq:tscalinglaw}) and Def \ref{def:metrics}, we have
\ben
H_{t,\lam}(\zeta)=\rd{\Gamma_t^\ast}^{-1}H_{1,\lam}\rd{t^{2/3}\zeta}\Gamma_t^{-1}
\een
with $\Gamma_t=\tx{diag}\rd{t^{1/3},t^{-1/3}}$. The matrix-valued function $\rd{h_t^{\tx{app}}}_{\uls_j^{\ulam(t)}}$ on $\widetilde \mbd_{j,t}$ (where the frame $\uls_j^{\ulam(t)}$ is defined in Def \ref{def:weightdependentframes}) is therefore given by $H_{t,\lam_j(t)}(t^{-2/3}\zeta')=\rd{\Gamma_t^\ast}^{-1}H_{1,\lam_j(t)}(\zeta')\Gamma_t^{-1}$, where $\zeta'=t^{2/3}\zeta\in K=\overline{B(0,\rho_0)}$ is a fixed compact set. The eigenvalues are roots $0<\alp_{1,\lam}\le \alp_{2,\lam}$ of
\ben
x^2-\rd{t^{2/3}\rd{H_{1,\lam}}_{11}+t^{-2/3}\rd{H_{1,\lam}}_{22}}x+\det H_{1,\lam}=0
\een
for $\lam=\lam_j(t)$. By Cor \ref{cor:metriccont} there is $C_0>0$ such that for all $\lam\in I$,
\ben
\frac{\alp_{2,\lam}}{\alp_{1,\lam}}\le C_0 t^{4/3}
\een
Therefore there is $C_1>1$ such that for all $A$ on $\widetilde \mbd_{j,t}$,
\ben
C_1^{-1}t^{-4/3}\vb{A}_I^2\le \vb{A}_{\rd{h_t^{\tx{app}}}_{\uls_j^{\ulam}}}\le C_1t^{4/3}\vb{A}_I^2
\een

{\it Region (2)}: let $H_0=\widetilde H_{t,\lam_j(t)}^{\tx{int}}$ and $H_1=\widetilde H_{t,\lam_j(t)}^{\tx{ext}}$. We have by (\ref{eq:Htappbd}) there are $C_2$, $c_2>0$ such that $\vb{H_0-H_1}\le C_2e^{-c_2t^{2/3}}$. By the definition of $\widetilde H_{t,\lam}^{\tx{ext}}$ in (\ref{eq:tilHtext}) and Prop \ref{prop:clamcont}, there is $C_3>0$ such that $\vb{H_1}\le C_3 t^{4a_\lam'/3}$ where $a_\lam'=\max\rd{-2\lam,\lam}$ and $\lam=\lam_j(t)$. For $t\gg 1$, we have $\vb{H_0}\le 2C_3t^{4a_\lam'/3}$. Therefore if $\eps_0$ is as in Lemma \ref{lem:hermcomp1}, for $t\gg 1$,
\ben
\vb{H_0^{-1/2}\rd{H_0-H_1}H_0^{-1/2}}, \vb{H_1^{-1/2}\rd{H_0-H_1}H_1^{-1/2}}\le \eps_0\,,
\een
and there is $C_4\ge 1$ such that $C_4^{-1}\vb{A}_{H_1}^2\le \vb{A}_{H_0}^2\le C_4\vb{A}_{H_1}^2$. By Prop \ref{prop:clamcont}, there is $C_5>0$ such that $\alp_2(H_1)/\alp_1(H_1)\le C_5 t^{4\vb{\lam_j(t)}}\le C_5t$. Combining the above, by Lemma \ref{lem:hermcomp0}, there is $C_6>1$ such that on $\mbd_j'-\widetilde \mbd_{j,t}$ for $t\gg 1$, $C_6^{-1}t^{-1}\vb{A}_I^2\le \vb{A}_{H_0}^2\le C_6t\vb{A}_I^2$ for all $A$.

By (\ref{eq:tilHtint}) and Def \ref{def:htapp}, \ref{def:weightdependentframes}, we have
\be \label{eq:defS2}
\rd{h_t^{\tx{app}}}_{\uls_j}^{\ulam}=S^\ast \widetilde H_{t,\lam_j(t)}^{\tx{int}}S,\,\, S=\ov{\sqrt{2}}\pmt{\zeta_j & -1 \\ \zeta_j & 1}
\ee
We have $C_7>0$ such that on $\mbd_j'-\widetilde \mbd_{j,t}$, $|S|\le C_7$, and $|S^{-1}|<C_7 t^{2/3}$. Therefore by Lemma \ref{lem:hermcomp2}, there is $C_8>1$ such that 
\ben
C_8^{-1}t^{-10/3}\vb{A}_{H_0}^2\le \vb{A}_{S^\ast H_0 S}^2=\vb{A}_{\rd{h_t^{\tx{app}}}_{\uls_j^{(0)}}}^2\le C_8 t^{10/3}\vb{A}_{H_0}^2\,.
\een

{\it Combining results on regions (1) and (2)}: on $\mbd_j'$, there is $C_9>1$ such that for all $A$,
\be \label{eq:mutualbound1}
C_9 t^{-13/3}\vb{A}_I^2\le \vb{A}_{\rd{h_t^{\tx{app}}}_{\uls_j^{\ulam}}}^2\le C_9 t^{13/3}\vb{A}_I^2
\ee
Note that by Def \ref{def:htapp}, $\rd{h_t^{\tx{app}}}_{\uls_j^{(0)}}=\rd{S^{-1}T_{\ulam,j}S}^\ast\rd{h_t^{\tx{app}}}_{\uls_j^{\ulam}} \rd{S^{-1}T_{\ulam,j}S}$. By direct calculation, we have
\be \label{eq:matrixSinvTS}
S^{-1}T_{\ulam,j}^{\pm 1}S=\ov{2}\pmt{e^{\pm F}+e^{\mp F/2} & \zeta^{-1}\rd{-e^{\pm F}+e^{\mp F/2}}\\ \zeta\rd{-e^{\pm F}+e^{\mp F/2}} & e^{\pm F}+e^{\mp F/2}}\,,
\ee
where $F=F_{\ulam(t),j}$ defined in (\ref{eq:Fulamj}) is a linear combination of fixed holomorphic functions $\xi_j$ and $f_{\ell j}$ with coefficients linear in $\ulam\in\mci$. In particular, $\xi_j(0)=f_{\ell j}(0)=0$. Thus there is $C_{10}>0$ such that for $\ulam\in \mci$, 
\be \label{eq:bdSinvTS}
\vb{S^{-1}T_{\ulam,j}^{\pm 1}S}\le C_{10}\,.
\ee
By Lemma \ref{lem:hermcomp2}, there is $C_{11}>1$ such that for any $u\in \tx{End}\rd{F}$ over $\mbd_j'$ and $t\gg 1$,
\ben
C_{11}^{-1}t^{-13}\vb{u}_{h_{t_0}^{\tx{app}}}^2\le \vb{u}_{h_t^{\tx{app}}}^2\le C_{11}t^{13}\vb{u}_{h_{t_0}^{\tx{app}}}^2
\een

{\it Region (3)}: Note that by the properties $\psi_P$ in Prop \ref{prop:locmodpainleve}, there is $C_{12}>0$ such that $\vb{H_t^{\tx{int},\bet}}$, $\vb{H_t^{\tx{int},\gam}}\le C_{12}$. By (\ref{eq:defTulamjtprime}) and Prop \ref{prop:clamcont}, there is $C_{13}>0$ such that
\ben
\alp_2\rd{H}/\alp_1\rd{H}\le C_{13}t^{4\vb{\lam_{j_0}(t)}}\le C_{13}t
\een
where $H=\rd{h_t^{\tx{app}}}_{\uls_j^{(0)}}$. By Lemma \ref{lem:hermcomp0}, there is $C_{14}>1$ such that on $\mbd_j'$ for $t$ large enough,
\be \label{eq:mutualbound2}
C_{14}^{-1}t^{-1}\vb{u}_{h_{t_0}^{\tx{app}}}^2\le \vb{u}_{h_t^{\tx{app}}}^2\le C_{14}t\vb{u}_{h_{t_0}^{\tx{app}}}^2
\ee

{\it Region (4)}: from Def \ref{def:tcompatible}, \ref{def:someharmfcns} we have $h_{L,\ulam,t}=h_{L,\ulam}^0e^{\eta_{\ulam,t}}$ and $h_{L,\ulam}^0=h_{L,\tx{HE}}e^{\varphi_{\ulam}}$ where $\varphi_{\ulam}$ is a linear combination of fixed functions $G_j$ bounded on $X-\coprod_j \mbd_j'$ with coefficients linear in $\ulam$. By Prop \ref{prop:clamcont}, there are $C_{15}$ and $C_{15}'>0$ such that for all $t\gg 1$,
\ben
\vb{\eta_{\ulam,t}-(4/3)\lam_{j_0}\log t}\le C_{15},\,\, \vb{\varphi_{\ulam}}\le C_{15}'
\een
$H$ be a local form of $h_t^{\tx{app}}$. From Def \ref{def:htapp} there is $C_{15}''>0$ such that $\alp_2(H)/\alp_1(H)\le C_{15}''' t^{4\vb{\lam_{j_0}}}\le C_{15}'''t$. Therefore there is $C_{16}>1$ such that on $X-\coprod_j \mbd_j'$ with $t\gg 1$,
\be \label{eq:mutualbound3}
C_{16}^{-1}t^{-1}\vb{u}_{h_{t_0}^{\tx{app}}}^2\le \vb{u}_{h_t^{\tx{app}}}^2\le C_{16}t\vb{u}_{h_{t_0}^{\tx{app}}}^2\,.
\ee
\end{proof}

\begin{lem} \label{lem:compLaplacians}
There is $C>0$ such that for $t\gg 1$, $u\in L_2^2\rd{\tx{Herm}(F,h_t^{\tx{app}})}$ we have
\ben
\dbv{\Delta_{h_t^{\tx{app}}}u-\Delta_{h_{t_0}^{\tx{app}}}u}_{L^2,h_{t_0}^{\tx{app}}}^2\le C t^4\rd{\dbv{d_{h_{t_0}^{\tx{app}}} u}_{L^2,h_{t_0}^{\tx{app}}}^2+\dbv{u}_{L^2,h_{t_0}^{\tx{app}}}^2}\,,
\een
where $d_h=\pd_h+\bar\pd$ the exterior covariant derivative with respect to the Chern connection $\nabla_h$.
\end{lem}

\begin{proof}
Let $H$ (resp. $U$) be the local form of a metric $h$ (resp. an endomorphism $u$) with respect to a holomorphic frame over a chart $(V;z)$. Up to a positive scalar independent of $h$, the local form of $\Delta_h u=i\Lambda_\omega \rd{\bar\pd\pd_h u-\pd_h \bar\pd u}$ is
\be \label{eq:localformdifflaplacian}
-4\pd_{\bar z}\pd_z U-4\sq{H^{-1}\pd_z H,\pd_{\bar z} U}-2\sq{\pd_{\bar z}\rd{H^{-1}\pd_z H},U}\,.
\ee
Fix local holomorphic frames as in the proof of Lemma \ref{lem:htappcomp}. Since the first term $-4\pd_{\bar z}\pd_z U$ is independent of $H$, it suffices to bound the differences in $H^{-1}\pd_\zeta H$, $\pd_{\bar\zeta}\rd{H^{-1}\pd_\zeta H}$ for $H$ given by $h_t^{\tx{app}}$ (resp. $h_{t_0}^{\tx{app}}$) over the four types of regions listed at the beginning of the proof of Lemma \ref{lem:htappcomp}. Note that the second term is proportional to the curvature $F_{\nabla_h}$, which vanishes for both $h_t^{\tx{app}}$ and $h_{t_0}^{\tx{app}}$ outside $\mbd_j'$.

{\it Region (1)}: On $\widetilde \mbd_{j,t}$ with $p_j\in D_r$, denote by $T_1=S^{-1}T_{\ulam(t),j}S$ (see (\ref{eq:defS2}) and Def \ref{def:htapp}), $T_0=S^{-1}T_{\ulam(t_0),j}S$, $H_1=H_{t,\lam_j(t)}$, and $H_0=H_{t_0,\lam_j(t_0)}$. Let $H_\ell'=T_\ell^\ast H_\ell T_\ell$ for $\ell=0,1$. We have $\rd{h_t^{\tx{app}}}_{\uls_j^{(0)}}=H_1'$ and $\rd{h_{t_0}^{\tx{app}}}_{\uls_j^{(0)}}=H_0'$. Note that $\pd_{\bar\zeta}T_\ell=0$, we have by a direct calculation,
\begin{align}
& \rd{H_\ell'}^{-1}\pd_\zeta H_\ell'=T_\ell^{-1} \rd{H_\ell^{-1}\pd_\zeta H_\ell}T_\ell+T_\ell^{-1} \pd_\zeta T_\ell\,, \nonumber \\
& \pd_{\bar\zeta}\rd{\rd{H_\ell'}^{-1}\pd_\zeta H_\ell'}=T_\ell^{-1}\pd_{\bar\zeta}\rd{H_\ell^{-1}\pd_\zeta H_\ell}T_\ell\,. \label{eq:twomatrixforms}
\end{align}
By (\ref{eq:tscalinglaw}) and Def \ref{def:metrics}, we have
\begin{align*}
& H_1^{-1}\pd_\zeta H_1=t^{2/3} \Gamma_t \rd{\atp{H_{1,\lam_j(t)}^{-1}\pd_\zeta H_{1,\lam_j(t)}}{t^{2/3}\zeta}} \Gamma_t^{-1}\,, \\
& \pd_{\bar\zeta}\rd{H_1^{-1}\pd_\zeta H_1}=t^{4/3} \Gamma_t\pd_{\bar\zeta}\rd{\zeta\mapsto \atp{H_{1,\lam_j(t)}^{-1}\pd_\zeta H_{1,\lam_j(t)}}{t^{2/3}\zeta}}\Gamma_t^{-1}\,,
\end{align*}
where $\Gamma_t=\tx{diag}(t^{1/3},t^{-1/3})$. We have $\vb{\Gamma_t A\Gamma_t^{-1}}\lesssim t^{2/3}\vb{A}$. On $\widetilde \mbd_{j,t}$, we have $\vb{\zeta}\le \rho_0t^{-2/3}$, thus $t^{2/3}\zeta\in \overline{B(0,\rho_0)}$ a compact set independent of $t$. By Cor \ref{cor:metriccont}, there are $C_0$, $C_0'$, $C_1$, $C_1'>0$ such that for $t\gg 1$,
\begin{align*}
& \vb{H_1^{-1}\pd_\zeta H_1}\le C_0 t^{4/3},\,\, \vb{H_0^{-1}\pd_\zeta H_0}\le C_0'\,, \\
& \vb{\pd_{\bar\zeta}\rd{H_1^{-1}\pd_\zeta H_1}}\le C_1 t^2,\,\, \vb{\pd_{\bar\zeta}\rd{H_0^{-1}\pd_\zeta H_0}}\le C_1'\,.
\end{align*}
Similar to the proof of Lemma \ref{lem:htappcomp}, there is $C_2>0$ such that $\vb{T_\ell^{\pm 1}}$, $\vb{T_\ell^{-1}\pd_\zeta T_\ell}\le C_2$ for $\ell=0,1$ and $t\gg 1$. Thus there are $C_3$, $C_3'>0$ such that
\begin{align}
& \vb{\rd{H_1'}^{-1}\pd_\zeta H_1'-\rd{H_0'}^{-1}\pd_\zeta H_0'}\le C_3 t^{4/3}\,, \nonumber \\
& \vb{\pd_{\bar\zeta}\rd{\rd{H_1'}^{-1}\pd_\zeta H_1'-\rd{H_0'}^{-1}\pd_\zeta H_0'}}\le C_3't^2\,. \label{eq:compLaplacians1}
\end{align}

{\it Region (2)}: On $\mbd_j'-\widetilde \mbd_{j,t}$ with $p_j\in D_r$, we have by the frame defined at the beginning of \S \ref{sec:htapp}, Def \ref{def:weightdependentframes} and \ref{def:htapp} that $\rd{h_t^{\tx{app}}}_{\uls_j^{(0)}}=S^\ast T_{\ulam(t),j}^\ast \rd{h_t^{\tx{app}}}_{\usig_j^{\ulam(t)}}T_{\ulam(t),j}S$ and that $\rd{h_t^{\tx{app}}}_{\usig_j^{\ulam(t)}}=\widetilde H_{t,\lam_j(t)}^{\tx{int}}$. Let $H_0=\widetilde H_{t_0,\lam_j(t_0)}^{\tx{int}}$, $H_1=\widetilde H_{t,\lam_j(t)}^{\tx{int}}$, $T_0=T_{\ulam(t_0),j}S$, $T_1=T_{\ulam(t),j}S$, and let $H_\ell'=T_\ell^\ast H_\ell T_\ell$. The calculations in (\ref{eq:twomatrixforms}) still apply. Let $H_{\infty,1}=\widetilde H_{t,\lam_j(t)}^{\tx{ext}}$. Note that we have $\rho_0t^{-2/3}\le \rho\le 2R/3$ on $\mbd_j'-\widetilde \mbd_{j,t}$. By a direct calculation using (\ref{eq:tilHtext}), (\ref{eq:defmulam}), and Prop \ref{prop:clamcont}, there are $C_{14}$, $C_{14}'>0$ such that
\ben
C_{14}R^{-1}\le \vb{H_{\infty,1}^{-1}\pd_\zeta H_{\infty,1}}\le C_{14}'t^{2/3} \rho_0^{-1}
\een
By (\ref{eq:Htappcont}), there are $C_5$, $C_5'>0$ such that
\ben
\vb{H_1^{-1}\pd_\zeta H_1-H_{\infty,1}^{-1}\pd_\zeta H_{\infty,1}}\le C_5\le C_5'\vb{H_{\infty,1}^{-1}\pd_\zeta H_{\infty,1}}
\een
Therefore there are $C_6$, $C_6'>0$ such that
\be \label{eq:bdhinvdh}
\vb{H_1^{-1}\pd_\zeta H_1}\le (1+C_6)\vb{H_{\infty,1}^{-1}\pd_\zeta H_{\infty,1}}\le C_6't^{2/3}
\ee
Since $\pd_{\bar\zeta}\rd{H_{\infty,1}^{-1}\pd_\zeta H_{\infty,1}}=0$, again by (\ref{eq:Htappcont}) given , there is $C_7>0$ such that
\ben
\vb{\pd_{\bar\zeta}\rd{H_1^{-1}\pd_\zeta H_1}}\le C_7
\een
By (\ref{eq:defS2}), there are $C_8$ and $C_8'>0$ such that $|S|\le C_8$, $\vb{S^{-1}}\le C_8't^{2/3}$ on $\mbd_j'-\widetilde \mbd_{j,t}$. Therefore there are $C_9$, $C_9'$, $C_9''>0$ such that $\vb{T_1}\le C_9$, $\vb{T_1^{-1}}\le C_9't^{2/3}$ and $\vb{T_1^{-1}\pd_\zeta T_1}\le C_9''t^{2/3}$. Note that $H_0'$ is independent of $t$. By (\ref{eq:twomatrixforms}), there are $C_{10}$, $C_{10}'>0$ such that on $\mbd_j'-\widetilde \mbd_{j,t}$,
\begin{align}
& \vb{\rd{H_1'}^{-1}\pd_\zeta H_1'-\rd{H_0'}^{-1}\pd_\zeta H_0'}\le C_{10}t^{4/3}\,,\nonumber \\
& \vb{\pd_{\bar\zeta}\rd{\rd{H_1'}^{-1}\pd_\zeta H_1'-\rd{H_0'}^{-1}\pd_\zeta H_0'}}\le C_{10}'t^{2/3} \label{eq:compLaplacians11}
\end{align}

{\it Region (3)}: On $\mbd_j'$ with $p_j\in D_\gam$, let $H_0=H_{t_0}^{\tx{int},\gam}$, $H_1=H_t^{\tx{int},\gam}$, $T_0=T_{\ulam(t_0),j,t_0}'$, $T_1=T_{\ulam(t),j,t}'$. Let $H_\ell'=T_\ell^\ast H_\ell T_\ell$ with $\ell=0,1$. We have $H_0'=\rd{h_{t_0}^{\tx{app}}}_{\uls_j^{(0)}}$ and $H_1'=\rd{h_t^{\tx{app}}}_{\uls_j^{(0)}}$. Note that $H_\ell^{\pm 1}$, $\pd_\zeta H_\ell$ and $T_\ell^{\pm 1}$ are diagonal, we have from (\ref{eq:twomatrixforms}) that $\rd{H_\ell'}^{-1}\pd_\zeta H_\ell'=H_\ell^{-1}\pd_\zeta H_\ell+T_\ell^{-1}\pd_\zeta T_\ell$ and $\pd_{\bar\zeta}\rd{\rd{H_\ell'}^{-1}\pd_\zeta H_\ell'}=\pd_{\bar\zeta}\rd{H_\ell^{-1}\pd_\zeta H_\ell}$. The latter has all but (1,1) entry zero. By a direct calculation,
\ben
H_1^{-1}\pd_\rho H_1=\tx{diag}\rd{-\chi'\psi_P-(2\rho)^{-1}\rd{1+2\chi \rho\pd_\rho\psi_P},0}
\een
where $\psi_P$ is the Painlev\'e function (see (\ref{eq:defpsiP})). We have $1+2\rho\pd_\rho\psi_P=8\eta_P$ where
\be \label{eq:defetapainleve}
\eta_P = \eta\rd{(8/3)t\rho^{3/2}},\,\,\, \eta(x)=\rd{1+3x \psi'(x)}/8\,.
\ee
By Lemma 3.4 of \cite{MSWW16}, we have that $\eta(x)\le C_{11}x^{4/3}$ for $x\ll 1$ and $0\le \eta(x)\le 1/8$ for all $x\ge 0$. For $t \gg 1$, let $\rho_1$ be such that $t^{-2/3}\rho_1< R/3$, therefore for $\rho\le \rho_1$, $\chi(\rho)=1$ and there is $C_{11}>0$ such that
\ben
\rd{2\rho}^{-1}\vb{1+2\chi\rho\pd_\rho\psi_P}=4\rho^{-1}\eta_P\le C_{11}t^{2/3}\,.
\een
On the other hand for $t^{-2/3}\rho_1\le \rho\le 2R/3$, we have $2\chi\rho\vb{\pd_\rho\psi_P}\le 2\rho\vb{\pd_\rho\psi_P}=\vb{\eta_P-1}\le 1$. It follows that 
\ben
\rd{2\rho}^{-1}\vb{1+2\chi\rho\pd_\rho\psi_P}\le\rho^{-1}\le \rho_1^{-1}t^{2/3}\,.
\een
Note that $\chi'$ is supported only on $\mbd_j'-\mbd_j''$, where $(8/3)t\rho^{3/2}\ge (8/3)(R/3)^{3/2}$. Since $\psi_P>0$ is decreasing monotonically, there is $C_{12}>0$ such that $\vb{\chi'\psi_P}\le C_{12}$ on $\mbd_j'$. By the form of $T_{\ulam,j,t}'$ in Def \ref{def:htapp}, there is $C_{13}>0$ such that $\vb{T_1^{\pm 1}}$, $\vb{T_1^{-1}\pd_\zeta T_1}\le C_{13}$.

By (\ref{eq:twomatrixforms}) since $H_0'$ is independent of $t$, there is $C_{14}>0$ such that on $\mbd_j'$ for $t\gg 1$,
\be \label{eq:compLaplacians2}
\vb{\rd{H_1'}^{-1}\pd_\zeta H_1'-\rd{H_0'}^{-1}\pd_\zeta H_0'}\le C_{14}t^{4/3}\,.
\ee
On $\mbd_j''$, $\chi\equiv 1$ therefore
\ben
\vb{\rd{\pd_{\bar\zeta}\rd{H_1^{-1}\pd_\zeta H_1}}_{1,1}}=\vb{\rd{4\rho}^{-1}\rd{\pd_\rho\psi_P+\rho\pd_\rho^2\psi_P}}=\vb{2t^2\rho\sinh\rd{2\psi_P}}
\een
By the asymptotic expansion of $\psi_P(x)$ for $x\to 0$, there is $C_{15}>0$ such that $\sinh \rd{2\psi_P}\le C_{15}x^{-2/3}$. Thus there is $C_{15}'>0$ with
\ben
\vb{\pd_{\bar\zeta}\rd{H_1^{-1}\pd_\zeta H_1}}\le \vb{2t^2\rho\sinh\rd{2\psi_P}}\le C_{15}'t^{4/3}
\een
On $\mbd_j'-\mbd_j''$, we have $C_{16},C_{16}'>0$ such that
\begin{align}
&\vb{\rd{\pd_{\bar\zeta}\rd{H_1^{-1}\pd_\zeta H_1}}_{1,1}}\le \rd{\chi/(4\rho)}\vb{\pd_\rho\psi_P+\rho\pd_\rho^2\psi_P}+\vb{\chi'}\rd{4\rho}^{-1}\rd{\vb{\psi_P}+2\rho\vb{\pd_\rho\psi_P}}+(1/4)\vb{\chi''}\vb{\psi_P}\nonumber \\
&\le C_{16}\rd{\rd{4\rho}^{-1}\vb{\pd_\rho\psi_P+\rho\pd_\rho^2\psi_P}+\vb{\psi_P}+2\rho\vb{\pd_\rho\psi_P}}\nonumber \\
&=C_{16}\rd{\vb{2t^2\rho\sinh\rd{2\psi_P}}+\vb{\psi_P}+\vb{8\eta_P-1}}\le C_{16}'t^{4/3}\,, \label{eq:compLaplacians21}
\end{align}
where we used property of $\eta_P$ as well as the fact that $\psi_P$ decreases monotonically.

Combining the above estimates, there is $C_{17}>0$ such that on $\mbd_j'$ with $t\gg 1$,
\be \label{eq:compLaplacians3}
\vb{\pd_{\bar\zeta}\rd{\rd{H_1'}^{-1}\pd_\zeta H_1'-\rd{H_0'}^{-1}\pd_\zeta H_0'}}\le C_{17} t^{4/3}
\ee

Note that $H_t^{\tx{int},\bet} = \rd{H_t^{\tx{int},\gam}}^{-1}$. The same estimates hold on $\mbd_j'$ for $p_j\in D_\bet$.

{\it On region (4)}: Fix a finite atlas $(U_\alp,z_\alp)_{\alp\in \mca}$ trivializing line bundle $L$ of $X-\coprod_j \mbd_j'$. Let $\uls_\alp$ be a holomorphic frame of $\left.V\right|_{U_\alp}$ induced by that of $\left.L\right|_{U_\alpha}$ and $dz_\alp$ and identify $F$ with $V$ via $\iota$ over each $U_\alpha$. Let $H_0$ (resp. $H_1$) be the local form of $h_{t_0}^{\tx{app}}$ (resp. $h_t^{\tx{app}}$) with respect to this frame over $U_\alp$. From Def \ref{def:someharmfcns} and \S \ref{sec:htapp} we have $h_{L,\ulam(t),t}=h_{L,\ulam(t)}^0 e^{\eta_{\ulam(t),t}}=h_{L,\tx{HE}}e^{\varphi_{\ulam(t)}+\eta_{\ulam(t),t}}$. We have
\be \label{eq:compLaplacians31}
\rd{\iota^\ast \rd{h_{L,\tx{HE}}^{-2}h_K\oplus h_{L,\tx{HE}}h_K}}_{\uls_\alp}=\tx{diag}\rd{h_0^{-2},h_0}
\ee
where $h_0>0$ on $U_\alpha$. It follows that 
\begin{align*}
& H_0=\tx{diag}\rd{h_0^{-2}e^{-2\varphi_{\ulam(t_0)}-2\eta_{\ulam(t_0),t_0}}, h_0e^{\varphi_{\ulam(t_0)}+\eta_{\ulam(t_0),t_0}}}\,, \\
& H_1=\tx{diag}\rd{h_0^{-2}e^{-2\varphi_{\ulam(t)}-2\eta_{\ulam(t),t}}, h_0e^{\varphi_{\ulam(t)}+\eta_{\ulam(t),t}}}\,.
\end{align*}
where $\varphi_{\ulam}$ is a linear combination of fixed functions $G_\ell$ with coefficients linear in $\ulam$. Thus there is $C_{18}>0$ such that for all $\ulam\in\mci$ and $\alp$ we have on $U_\alp$,
\ben
\vb{\pd_{z_\alp} \varphi_{\ulam}}\le C_{18}
\een
Therefore there is $C_{18}'>0$ such that
\be \label{eq:compLaplacians4}
\max_\alp \sup_{U_\alp}\vb{H_1^{-1}\pd_{z_\alp} H_1-H_0^{-1}\pd_{z_\alp} H_0}\le C_{18}'
\ee

Combining (\ref{eq:compLaplacians1}), (\ref{eq:compLaplacians2}), (\ref{eq:compLaplacians3}), and (\ref{eq:compLaplacians4}), there are $C$, $C'>0$ such that for $u\in L_2^2\rd{\tx{Herm}\rd{F,h_t^{\tx{app}}}}$ and $t\gg 1$,
\begin{align*}
& \dbv{\Delta_{h_t^{\tx{app}}}u-\Delta_{h_{t_0}^{\tx{app}}}u}_{L^2,h_{t_0}^{\tx{app}}}^2\le C'\sum_{\alp\in\mca}\int_{U_\alp} \vb{\sq{H_1^{-1}\pd_{z_\alp} H_1-H_0^{-1}\pd_{z_\alp} H_0,\pd_{\bar z_\alp} u_{\uls_\alp}}}^2 \\
&+C'\sum_{j=1}^{4g-4}\int_{\mbd_j'}\vb{\sq{H_1^{-1}\pd_{\zeta_j}H_1-H_0^{-1}\pd_{\zeta_j}H_0,\pd_{\bar\zeta_j} u_{\uls_j^{(0)}}}}^2 \\
&+C'\sum_{j=1}^{4g-4}\int_{\mbd_j'}\vb{\sq{
\pd_{\bar\zeta_j}\rd{H_1^{-1}\pd_{\zeta_j}H_1}-
\pd_{\bar\zeta_j}\rd{H_0^{-1}\pd_{\zeta_j}H_0},u_{\uls_j^{(0)}}}
}^2 \\
&\le C t^4 \rd{\dbv{d_{h_{t_0}^{\tx{app}}} u}_{L^2,h_{t_0}^{\tx{app}}}^2+\dbv{u}_{L^2,h_{t_0}^{\tx{app}}}^2}
\end{align*}
where $H_0$ (resp. $H_1$) is the local form of $h_{t_0}^{\tx{app}}$ (resp. $h_t^{\tx{app}}$), and in the last step, we used the fact that there is $C''>0$ such that for $t\gg 1$,
\ben
\int_{U_\alp}\vb{\pd_{\bar z_\alp}u_{\uls_\alp}}^2 ,\,\, \int_{\mbd_j'}\vb{\pd_{\bar \zeta_j}u_{\uls_j^{(0)}}}^2\le C''\rd{\dbv{d_{h_{t_0}^{\tx{app}}} u}_{L^2,h_{t_0}^{\tx{app}}}^2+\dbv{u}_{L^2,h_{t_0}^{\tx{app}}}^2}\,.
\een
\end{proof}

As a consequence of (\ref{eq:compLaplacians1}), (\ref{eq:compLaplacians11}), and (\ref{eq:compLaplacians3}) in the above proof we have:

\begin{lem} \label{lem:htappcurvbound}
There is $C>0$ such that for $t\gg 1$,
\ben
\vb{F_{\nabla}}_{h_{t_0}^{\tx{app}}}^2\le C t^4
\een
where $\nabla=\nabla_{h_t^{\tx{app}}}$.
\end{lem}

\begin{lem} \label{lem:psibetagammaest}
There is $C>0$ such that $u\in L_2^2\rd{\tx{Herm}\rd{F,h_t^{\tx{app}}}}$ and $t$ large enough,
\ben
\dbv{\cl{\psi,\hat u}}_{L^2,h_{t_0}^{\tx{app}}}\le C t^{11/3} \dbv{u}_{L^2,h_{t_0}^{\tx{app}}}
\een
where $\psi=\psi_{\bet,\gam,h_t^{\tx{app}}}$ is defined in Prop \ref{prop:Ltform}.
\end{lem}

\begin{proof}
As in the proof of Lemmas \ref{lem:htappcomp} we use the decomposition of $X$ into four types of regions, use holomorphic frames $\uls_j^{(0)}$ in \S \ref{sec:htapp} over $\mbd_j'-\widetilde \mbd_{j,t}$, $\widetilde \mbd_{j,t}$ and $\uls_\alp$ over $U_\alp$ induced by that of $L$ and $dz_\alp$ of $K$. Let $H$ (resp. $U$) be the local form of $h_t^{\tx{app}}$ (resp. $u$) with respect to respective holomorphic frames and let $\hat U=U+\rd{\tr\,U}\tx{Id}$. The local form of $\cl{\psi,\hat u}$ is given by
\ben
\cl{\gam_0\gam_0^\ast H \det H+\rd{\det H}^{-1}H^{-1}\bet_0^\ast \bet_0, \hat U}
\een

{\it Region (1)}: By (\ref{eq:tscalinglaw}) and Def \ref{def:metrics}, we have that $H_{t,\lam}(\zeta)=\rd{\Gamma_t^\ast}^{-1}H_{1,\lam}\rd{t^{2/3}\zeta}\Gamma_t^{-1}$ with \break $\Gamma_t=\tx{diag}\rd{t^{1/3},t^{-1/3}}$. For $t\gg 1$, $H_{t,\lam}$ on $\widetilde \mbd_{j,t}$ is determined by values of $H_{1,\lam}$ on the fixed compact set $\overline{B(0,\rho_0)}$. We have $\vb{\Gamma_t A\Gamma_t^{-1}}\lesssim t^{2/3}\vb{A}$ for $A$ a $2\times 2$ matrix. Thus there is $C_0>0$ such that for $t\gg 1$
\ben
\vb{H_{t,\lam_j(t)}^{\pm 1}}\le C_0 t^{2/3},\,\, \vb{\det H_{t,\lam_j(t)}^{\pm 1}}\le C_0\,.
\een
Furthermore, there is $C_1>0$ such that $\vb{H_{t,\lam_j(t)}^{\pm 1}\det H_{t,\lam_j(t)}^{\pm 1}}\le C_1 t^{2/3}$. Let $H=\rd{h_t^{\tx{app}}}_{\uls_j^{(0)}}$ we have $H=\rd{S^{-1}T_{\ulam(t),j}S}^\ast H_{t,\lam_j(t)}\rd{S^{-1}T_{\ulam(t),j}S}$ where $S$ is given in (\ref{eq:defS2}). By (\ref{eq:bdSinvTS}), there is $C_2>0$ such that
\ben
\vb{H\det H},\,\, \vb{H^{-1}\det H^{-1}}\le C_2 t^{2/3}
\een

{\it Region (2)}: On $\mbd_j'-\widetilde \mbd_{j,t}$ let $H'=\rd{h_t^{\tx{app}}}_{\uls_j^{(0)}}$ we have 
\ben
H'=\rd{T_{\ulam(t),j}S}^{\ast}\widetilde H_{t,\lam_j(t)}^{\tx{int}}\rd{T_{\ulam(t),j}S}
\een
By arguments below (\ref{eq:defS2}) and (\ref{eq:matrixSinvTS}) there is $C_4>0$ such that for $t\gg 1$, $\vb{T_{\ulam(t),j}S}$, $\vb{\det\rd{T_{\ulam(t),j}S}}\le C_4$, $\vb{S^{-1}T_{\ulam(t),j}^{-1}}$, $\vb{\det\rd{S^{-1}T_{\ulam(t),j}}^{-1}}\le C_4 t^{2/3}$. By arguments in the proof of Lemma \ref{lem:appsol1} and applying Lemma \ref{lem:matrest} and (\ref{eq:Htappbd}), there are $C_5$, $c_5>0$ such that on $\mbd_j'-\widetilde \mbd_{j,t}$ for $t\gg 1$,
\ben
\vb{\rd{\widetilde H_{t,\lam_j(t)}^{\tx{int}}\det \widetilde H_{t,\lam_j(t)}^{\tx{int}}}^{\pm 1}-\rd{\widetilde H_{t,\lam_j(t)}^{\tx{ext}}\det \widetilde H_{t,\lam_j(t)}^{\tx{ext}}}^{\pm 1}}\le C_5e^{-c_5 t^{2/3}}\,.
\een
By a direct calculation and Prop \ref{prop:clamcont}, there is $C_6>0$ such that
\ben
\vb{\rd{\widetilde H_{t,\lam_j(t)}^{\tx{ext}}\det \widetilde H_{t,\lam_j(t)}^{\tx{ext}}}^{\pm 1}}\le C_6 t^{4\vb{\lam_j(t)}}\le C_6 t
\een
Therefore there is $C_6'>0$ such that $\vb{\rd{\widetilde H_{t,\lam_j(t)}^{\tx{int}}\det \widetilde H_{t,\lam_j(t)}^{\tx{int}}}^{\pm 1}}\le C_6' t$. Combining the above, we have that there are $C_7$, $C_7'>0$ such that on $\mbd_j'-\widetilde \mbd_{j,t}$,
\ben
\vb{H'\det H'}\le C_7 t,\,\, \vb{\rd{H'}^{-1}\det \rd{H'}^{-1}}\le C_7' t^{11/3}\,.
\een

{\it Region (3)}: On $\mbd_j'=\cl{\rho\le 2R/3}$ for $p_j\in D_\bet$ or $D_\gam$, we have by direct calculation,
\ben
\rd{H_t^{\tx{int},\gam}}\det H_t^{\tx{int},\gam}=\tx{diag}\rd{\rho^{-1}e^{-2\chi \psi_P},\rho^{-1/2}e^{-\chi\psi_P}}=\rd{H_t^{\tx{int},\bet}}^{-1}\rd{\det H_t^{\tx{int},\bet}}^{-1}\,.
\een
On $\widetilde \mbd_{j,t}$, we have $\rho^{-1}e^{-2\psi_P}=t^{2/3}f(t^{2/3}\rho)$ where $f(x)=x^{-1}e^{-2\psi\rd{\frac{8}{3}x^{3/2}}}$ and $\psi_P$, $\psi$ are as in (\ref{eq:defpsiP}). By the properties of $\psi_P$ in Prop \ref{prop:locmodpainleve}, there is $C_8>0$ such that $\vb{f}$, $\vb{f^{-1}}\le C_8$ for $0\le x\le \rho_0$. Therefore there is $C_8'>0$ such that on $\widetilde \mbd_{j,t}$ we have $\vb{\rho^{-1}e^{-2\chi\psi_P}}\le C_8't^{2/3}$, $\vb{\rho e^{2\chi\psi_P}}\le C_8'$.

On $\mbd_j'-\widetilde \mbd_{j,t}$ for $p_j\in D_\bet$ since $\chi\psi_P>0$, we have $\vb{\rho^{-1}e^{-2\chi\psi_P}}\le \vb{\rho^{-1}}\lesssim t^{2/3}$. On the other hand, since $\psi_P$ is monotonically decreasing, so is $\chi\psi_P$. Therefore there is $C_9>0$ such that $\vb{\rho e^{2\chi\psi_P}}\le C_9$. 

Therefore on $\mbd_j'$ for $p_j\in D_\bet$ or $D_\gam$, we have that there is $C_{10}>0$,
\begin{align*}
& \vb{H_t^{\tx{int},\gam}\det H_t^{\tx{int},\gam}}, \,\, \vb{\rd{H_t^{\tx{int},\bet}\det H_t^{\tx{int},\bet}}^{-1}}\le C_{10} t^{2/3}\,, \\
& \vb{\rd{H_t^{\tx{int},\gam}\det H_t^{\tx{int},\gam}}^{-1}},\,\, \vb{H_t^{\tx{int},\bet}\det H_t^{\tx{int},\bet}}\le C_{10}\,.
\end{align*}

By a direct calculation, there is $C_{11}>0$ such that $\vb{T_{\ulam(t),j,t}'}\le C_{11}t^{2a_{\lam_{j_0}(t)}'/3}$, $\vb{\rd{T_{\ulam(t),j,t}'}^{-1}}\le C_{11}t^{2a_{\lam_{j_0}(t)}/3}$, $\vb{\det T_{\ulam(t),j,t}'}\le C_{11}t^{-2\lam_{j_0}(t)/3}$, $\vb{\rd{\det T_{\ulam(t),j,t}'}^{-1}}\le C_{11}t^{-2\lam_{j_0}(t)/3}$ where $a_\lam=\max\rd{2\lam,-\lam}$, $a_\lam'=\max\rd{-2\lam,\lam}$. Let $H=\rd{h_t^{\tx{app}}}_{\uls_j^{(0)}}$ we have by Def \ref{def:htapp} 
\ben
H=\begin{cases}
\rd{T_{\ulam(t),j,t}'}^\ast H_t^{\tx{int},\bet} T_{\ulam(t),j,t}' & p_j \in D_\bet \\
\rd{T_{\ulam(t),j,t}'}^\ast H_t^{\tx{int},\gam} T_{\ulam(t),j,t}' & p_j \in D_\gam
\end{cases}
\een
Combining the above estimates, there is $C_{12}>0$ such that for $p_j\in D_\gam$,
\ben
\vb{H\det H}\le C_{12}t^{5/3},\,\, \vb{H^{-1}\det H^{-1}}\le C_{12}t
\een
and for $p_j\in D_\bet$,
\ben
\vb{H\det H}\le C_{12}t,\,\, \vb{H^{-1}\det H^{-1}}\le C_{12}t^{5/3}
\een

{\it Region (4)}: On $X-\coprod_j \mbd_j'$, let $H$ be the local form of $h_t^{\tx{app}}$. We have for $s=\pm 1$,
\ben
\rd{H\det H}^s=\tx{diag}\rd{h_0^{3 s}e^{s\rd{ 3\varphi_{\ulam(t)}+ 3\eta_{\ulam(t),t}}},1}\,.
\een
By Prop \ref{prop:clamcont}, there is $C_{13}>0$ such that $\vb{\eta_{\ulam(t),t}-\frac{4}{3}\lam_{j_0}(t)\log t}\le C_{13}$ for $t\gg 1$. It follows that there is $C_{14}>0$ such that for all $\alp$ we have on $U_\alp$
\ben
\vb{\rd{H\det H}^{\pm 1}}\le C_{14}t
\een
The conclusion follows from combining all the above estimates.
\end{proof}

Now we are ready to prove Prop \ref{prop:L22toL2lowerbound}

\begin{proof}
Denote by $f\lesssim g$ if there is constant $C>0$ independent of $t$ such that $f\le C g$ for $t\gg 1$. By Prop \ref{prop:Ltform} and Lemmas \ref{lem:compLaplacians}, and \ref{lem:psibetagammaest}, we have for $u\in L_2^2\rd{\tx{Herm}\rd{F,h_t^{\tx{app}}}}$, $\epsilon>0$, such that for $t\gg 1$,
\begin{align*}
&\dbv{\Delta_{h_{t_0}^{\tx{app}}}u}^2\le \dbv{L_t u}^2+\dbv{\Delta_{h_t^{\tx{app}}}u -\Delta_{h_{t_0}^{\tx{app}}}u}^2+\dbv{t^2\cl{\psi_{\bet,\gam,h_t^{\tx{app}}},\hat u}}^2 \\
&\lesssim \dbv{L_t u}^2+t^4 \dbv{d_{h_{t_0}^{\tx{app}}}u}^2+t^{22/3+4}\dbv{u}^2\le \dbv{L_t u}^2+t^4 \rd{\Delta_{h_{t_0}^{\tx{app}}}u,u}+t^{11}\dbv{u}^2 \\
&\lesssim \dbv{L_t u}^2+\frac{\epsilon t^4}{2}\dbv{\Delta_{h_{t_0}^{\tx{app}}}u}^2+\frac{t^4}{2\epsilon}\dbv{u}^2+t^{11}\dbv{u}^2
\end{align*}
where $\dbv{\cdot}=\dbv{\cdot}_{L^2,h_{t_0}^{\tx{app}}}$, $\rd{\cdot,\cdot}=\rd{\cdot,\cdot}_{L^2,h_{t_0}^{\tx{app}}}$, and we will denote below $\dbv{\cdot}_t=\dbv{\cdot}_{L^2,h_t^{\tx{app}}}$. By taking $\epsilon=t^{-4}$, we get that
\ben
\dbv{\Delta_{h_{t_0}^{\tx{app}}}u}^2 \lesssim \dbv{L_t u}^2+t^{13}\dbv{u}^2
\een
By the elliptic estimate of $\Delta_{h_{t_0}^{\tx{app}}}$ on sections of $\tx{End}(F)$, we have 
\begin{align*}
&\dbv{u}_{L_2^2,h_{t_0}^{\tx{app}}}^2\lesssim \dbv{\Delta_{h_{t_0}^{\tx{app}}} u}^2+\dbv{u}^2\lesssim \dbv{L_t u}^2+t^{11}\dbv{u}^2\lesssim \dbv{L_t u}^2+t^{11+13}\dbv{u}_t^2 \\
&\lesssim \dbv{L_t u}^2+t^{24}\rd{\log t}^2\dbv{L_t u}_t^2\lesssim \rd{1+t^{24+13}\rd{\log t}^2}\dbv{L_t u}^2\lesssim t^{38}\dbv{L_t u}^2\,,
\end{align*}
where we applied Lemma \ref{lem:htappcomp} to compare $\dbv{\cdot}_t^2$ and $\dbv{\cdot}^2$ as well as Cor \ref{cor:mainL2est}, the conclusion follows.
\end{proof}

\subsection{Estimates for the remainder term}

In this section, we prove an upper bound on the remainder term defined in (\ref{eq:defnRtu}),
\begin{align*}
& R_t(u)=2i\Lambda e^{u/2}\mch_{t,h} e^{-u/2}-2i\Lambda \mch_{t,h}-L_t(u)  \\
& =2i\Lambda e^{u/2}F_{\nabla_{h\cdot e^u}}e^{-u/2} \\
&+2i t^2\Lambda e^{u/2}\rd{\bet\wedge \bet^{\ast_{h\cdot e^u}}} e^{-u/2}-2it\Lambda \rd{\bet\wedge\bet_h} \\
&-\Delta_h u-t^2\cl{\psi_{\bet,\gam,h},\hat u}-2i\Lambda F_{\nabla_h}  \\
&+2it^2\Lambda e^{u/2}\rd{\gam^{\ast_{h\cdot e^u}} \wedge\gam}e^{-u/2}-2it^2\Lambda \rd{\gam^{\ast_{h\cdot e^u}}\wedge \gam}\,,
\end{align*}
where $h=h_t^{\tx{app}}$. For a metric $h$ and $g\in \tx{Aut}(F)$ we have
\begin{align*}
&\bet\wedge \bet^{\ast_{h\cdot g}}=\rd{\bet\wedge\bet^{\ast_h}} \circ\widetilde g \\
&\gam^{\ast_{h\cdot g}}\wedge\gam=\rd{\widetilde g}^{-1}\circ\rd{\gam^{\ast_h} \wedge \gam}
\end{align*}
where $\widetilde g=\rd{\det g}g$. We have
\be \label{eq:decompRt}
R_t(u)=R_t^{(0)}(u)+R_t^{(1)}(u)+R_t^{(2)}(u)
\ee
with
\begin{align}
& R_t^{(0)}(u)=i\Lambda S_t(u) \nonumber \\
& S_t(u)=2 e^{u/2}F_{\nabla_h}e^{-u/2}-2F_{\nabla_h}+2e^{u/2}\bar\pd\rd{e^{-u}\pd_h e^u}e^{u/2}-\bar\pd \pd_h u+\pd_h\bar\pd u \nonumber \\
& R_t^{(1)}(u)=2t^2 e^{\hat u/2}B_t e^{\hat u/2}-2t^2 B_t-t^2\cl{B_t,\hat u}\nonumber \\
& R_t^{(2)}(u)=2t^2 e^{-\hat u/2}C_t e^{-\hat u/2}-2t^2 C_t+t^2\cl{C_t,\hat u} \label{eq:defRt0StRt1Rt2}
\end{align}
and
\ben
B_t=i\Lambda \bet\wedge \bet^{\ast_h},\,\, C_t=i\Lambda \gam^{\ast_h}\wedge \gam
\een

Given $u\in \tx{End}(F)$, let
\be \label{eq:defnfj}
f_j(u)=\sum_{k=j}^\infty \ov{k!}u^k
\ee
The lemma below follows easily from the Sobolev inequalities.

\begin{lem} \label{lem:bddifffj}
Given $0<r<1$, and $u_0,u_1\in B(0,r)\subset L_2^2\rd{\tx{Herm}\rd{F,h_t^{\tx{app}}}}$, for each $j\ge 1$, there is $C_j>0$ such that
\ben
\dbv{f_j(u_0)-f_j(u_1)}_{L_2^2,h_{t_0}^{\tx{app}}}\le C_j r^{j-1}\dbv{u_0-u_1}_{L_2^2,h_{t_0}^{\tx{app}}}
\een
\end{lem}

\begin{lem} \label{lem:delhest}
There is $C>0$ such that for $u\in L_2^2\rd{\tx{Herm}\rd{F,h_t^{\tx{app}}}}$, $k=1,2$ and $t\gg 1$
\ben
\dbv{\pd_{h_t^{\tx{app}}} u}_{L_{k-1}^2,h_{t_0}^{\tx{app}}}^2\le C t^4 \dbv{u}_{L_k^2,h_{t_0}^{\tx{app}}}^2
\een
\end{lem}

\begin{proof}
We prove the estimate as in the proof of Lemmas \ref{lem:htappcomp}, \ref{lem:compLaplacians}, \ref{lem:psibetagammaest} on four types of regions decomposing $X$. By the equivalent definitions of $L_k^2$ norm in (\ref{eq:defLk2norm1}) and (\ref{eq:defLk2norm2}), as well as the local form of $\pd_h u$ with respect to a local holomorphic frame, the conclusion will follow from bounds on
\be \label{eq:delhest3quantities}
H^{-1}\pd_z H, \,\, \,\pd_{\bar z}\rd{H^{-1}\pd_z H},\,\, \pd_z\rd{H^{-1}\pd_z H}\,,
\ee
where $H$ is the local form of $h_t^{\tx{app}}$. Note that (\ref{eq:compLaplacians1}), (\ref{eq:compLaplacians11}), (\ref{eq:compLaplacians2}), and (\ref{eq:compLaplacians4}) in the proof of Lemmas \ref{lem:compLaplacians}, \ref{lem:htappcurvbound} already provide appropriate bounds on the first two.

{\it Region (1)}: On $\widetilde \mbd_{j,t}$ with $p_j\in D_r$, let $H=\rd{h_t^{\tx{app}}}_{\uls_j^{(0)}}$ we have $H=\rd{\widetilde T}^\ast H_{t,\lam_j(t)} \widetilde T$ where $\widetilde T=S^{-1}T_{\ulam(t),j}S$. We have
\begin{align}
& H^{-1}\pd_{\zeta} H=\widetilde T^{-1}\rd{H_{t,\lam_j(t)}^{-1}\pd_{\zeta} H_{t,\lam_j(t)}}\widetilde T+\widetilde T^{-1}\pd_\zeta \widetilde T \nonumber \\
& \pd_\zeta \rd{H^{-1}\pd_\zeta H}=\widetilde T^{-1}\pd_\zeta \rd{H_{t,\lam_j(t)}^{-1}\pd_\zeta H_{t,\lam_j(t)}} \widetilde T+\pd_\zeta\rd{\widetilde T^{-1}}\rd{H_{t,\lam_j(t)}^{-1}\pd_\zeta H_{t,\lam_j(t)}}\widetilde T\,,\nonumber \\
& +\widetilde T^{-1}\rd{H_{t,\lam_j(t)}^{-1}\pd_\zeta H_{t,\lam_j(t)}}\pd_\zeta\rd{\widetilde T}+\pd_\zeta\rd{\widetilde T^{-1}\pd_\zeta \widetilde T}\,, \label{eq:dhinvdh} 
\end{align}
where $\zeta=\zeta_j$, and we used $\pd_{\bar \zeta}\widetilde T=0$. By (\ref{eq:tscalinglaw}) and Def \ref{def:metrics}, $\pd_\zeta\rd{H_{t,\lam}^{-1}\pd_\zeta H_{t,\lam}}=t^{4/3}\Gamma_t \pd_\zeta\rd{\zeta\mapsto \atp{H_{1,\lam}^{-1}\pd_\zeta H_{1,\lam}}{t^{2/3}\zeta}}\Gamma_t$ on $\widetilde \mbd_{j,t}$, where $\Gamma_t=\tx{diag}\rd{t^{1/3},t^{-1/3}}$. The same arguments in the proof of Lemma \ref{lem:compLaplacians} applies. We have $\vb{\Gamma_t A\Gamma_t^{-1}}\lesssim t^{2/3}\vb{A}$ for $A$ a $2\times 2$ matrix. By Cor \ref{cor:metriccont} and (\ref{eq:bdSinvTS}), there is $C_1>0$ independent in $\lam\in I$ such that for $t\gg 1$, $\vb{\pd_\zeta \rd{H_{t,\lam}^{-1}\pd_\zeta H_{t,\lam}}}\le C_1 t^2$, $\vb{H_{t,\lam}^{-1}\pd_\zeta H_{t,\lam}}\le C_1 t^{4/3}$. On the other hand, the arguments in the proof of Lemma \ref{lem:compLaplacians} implies that there is $C_2>0$ such that for $t\gg 1$
\ben
\vb{\widetilde T^{\pm 1}},\, \vb{\pd_\zeta \rd{\widetilde T^{\pm 1}}},\, \vb{\pd_\zeta \rd{\widetilde T^{-1}\pd_\zeta \widetilde T}}\le C_2
\een
Combining these, there is $C_3>0$ such that
\be \label{eq:delhestproof1}
\vb{\pd_\zeta\rd{H^{-1}\pd_\zeta H}}\le C_3t^2
\ee

{\it Region (2)}: On $\mbd_j'-\widetilde \mbd_{j,t}$, let $T'=T_{\ulam(t),j}S$, $H=\rd{h_t^{\tx{app}}}_{\uls_j^{(0)}}$ we have $H=\rd{T'}^\ast \widetilde H_{t,\lam_j(t)}^{\tx{int}} T'$. As in the proof of Lemma \ref{lem:compLaplacians} using Prop \ref{prop:clamcont}, there are $C_4$ and $C_4'>0$ such that for $t\gg 1$
\ben
\vb{\pd_\zeta \rd{\rd{\widetilde H_{t,\lam_j(t)}^{\tx{ext}}}^{-1}\pd_\zeta\rd{\widetilde H_{t,\lam_j(t)}^{\tx{ext}}}}}\le C_4't^{4/3}
\een
By (\ref{eq:Htappbd}) and the arguments in the proof of Lemma \ref{lem:appsol1}, there are $C_5$, $c_5>0$ such that on $\mbd_j'-\widetilde \mbd_{j,t}$, 
\ben
\vb{\pd_\zeta \rd{\rd{\widetilde H_{t,\lam_j(t)}^{\tx{int}}}^{-1}\pd_\zeta\rd{\widetilde H_{t,\lam_j(t)}^{\tx{int}}}}-\pd_\zeta \rd{\rd{\widetilde H_{t,\lam_j(t)}^{\tx{ext}}}^{-1}\pd_\zeta\rd{\widetilde H_{t,\lam_j(t)}^{\tx{ext}}}}}\le C_5 e^{-c_5t^{2/3}}
\een
It follows that there is $C_6>0$ such that
\ben
\vb{\pd_\zeta \rd{\rd{\widetilde H_{t,\lam_j(t)}^{\tx{int}}}^{-1}\pd_\zeta\rd{\widetilde H_{t,\lam_j(t)}^{\tx{int}}}}}\le C_6t^{4/3}
\een
By arguments in the proof of Lemma \ref{lem:compLaplacians} there is $C_8>0$ such that $\vb{T'}$, $\vb{\pd_\zeta T'}\le C_8$, $\vb{\rd{T'}^{-1}}$, $\vb{\rd{T'}^{-1}\pd_\zeta \rd{T'}}\le C_8t^{2/3}$ and $\vb{\pd_\zeta\rd{\rd{T'}^{-1}}}$, $\vb{\pd_\zeta\rd{\rd{T'}^{-1}\pd_\zeta T'}}\le C_8 t^{4/3}$. Combining these estimates and (\ref{eq:bdhinvdh}) where $H_1=\widetilde H_{t,\lam_j(t)}^{\tx{int}}$, there is $C_9>0$ such that for $t\gg 1$
\be \label{eq:delhestproof2}
\vb{\pd_\zeta\rd{H^{-1}\pd_\zeta H}}\le C_9 t^2\,.
\ee

{\it Region (3)}: For $p_j\in D_\gam$, recall that $H_t^{\tx{int},\gam}=\tx{diag}\rd{\rho^{-1/2}e^{-\chi\psi_P},1}$ where $\psi_P$ is the Painlev\'e function (see (\ref{eq:defpsiP})). On $\mbd_j''$, $\chi\equiv 1$ and
\ben
\vb{\pd_\zeta \rd{\rd{H_t^{\tx{int},\gam}}^{-1}\pd_\zeta H_t^{\tx{int},\gam} }}\le (\rho^{-2}/4)\vb{1+\rho\pd_\rho\psi_P-\rho^2\pd_\rho^2\psi_P}=2\vb{-t^2\rho\sinh \rd{2\psi_P}+\rho^{-2}\eta_P}
\een
where $\eta_P=(1+2\rho\pd_\rho \psi_P)/8$ as defined in the proof of Lemma \ref{lem:compLaplacians}. Recall that $\rho^{-2}\eta_P\lesssim t^{4/3}$ and we have that $t^2\rho \sinh(2\psi_P)\lesssim t^{4/3}$. Therefore there is $C_{10}>0$ such that on $\mbd_j''$,
\ben
\vb{\pd_\zeta \rd{\rd{H_t^{\tx{int}}}^{-1}\pd_\zeta H_t^{\tx{int}}}}_{H_1^{\tx{int}}}\le C_{10} t^{4/3}
\een
On $\mbd_j'-\mbd_j''$ by (\ref{eq:compLaplacians21}), there are $C_{11}$ and $C_{11}'>0$ such that for $t\gg 1$
\begin{align*}
& \vb{\pd_\zeta \rd{\rd{H_t^{\tx{int}}}^{-1}\pd_\zeta H_t^{\tx{int}}}}_{H_1^{\tx{int}}}\le \rd{\chi/(4\rho^2)}\vb{1+\rho\pd_\rho\psi_P-\rho^2\pd_\rho^2\psi_P}+\rd{\chi'/(4\rho)}\rd{\vb{\psi_P}+2\rho\vb{\pd_\rho\psi_P}}\\
& +\rd{\chi''/4}\vb{\psi_P} \le C_{11}\rd{\vb{-2t^2\rho\sinh(2\psi_P)+2\rho^{-2}\eta_P}+\vb{\psi_P}+\vb{8\eta_P-1}}\le C_{11}'t^{4/3}\,.
\end{align*}
By (\ref{eq:defTulamjtprime}), there is $C_{12}>0$ such that 
\ben
\vb{\rd{T_{\ulam(t),j,t}'}^{-1}\pd_\zeta T_{\ulam(t),j,t}'},\,\, \vb{\pd_\zeta \rd{\rd{T_{\ulam(t),j,t}'}^{-1}\pd_\zeta T_{\ulam(t),j,t}'}}\le C_{12}\,.
\een
Let $H=\rd{h_t^{\tx{app}}}_{\uls_j^{(0)}}$ we have $H=\rd{T_{\ulam(t),j,t}'}^\ast H_t^{\tx{int},\gam} T_{\ulam(t),j,t}'$. By calculation in (\ref{eq:dhinvdh}) and the estimate of $\vb{\rd{H_t^{\tx{int},\gam}}^{-1}\pd_\zeta H_t^{\tx{int},\gam}}$ in (\ref{eq:compLaplacians2}) there is $C_{13}>0$ such that for $t\gg 1$
\be \label{eq:delhestproof3}
\vb{\pd_\zeta\rd{H^{-1}\pd_\zeta H}}\le C_{13}t^{4/3}\,.
\ee
Since $H_t^{\tx{int},\bet}=H^{-1}$, we have the same bound of the last quantity in (\ref{eq:delhest3quantities}) for $\mbd_j'$ with $p_j\in D_\bet$.

{\it Region (4)}: On $X-\coprod_j \mbd_j'$, notation as in the proof of Lemma \ref{lem:compLaplacians}, we have
\ben
\pd_{z_\alp} \rd{H^{-1}\pd_{z_\alp}H}=\tx{diag}\rd{-2\pd_{z_\alp}^2 \log h_0-2\pd_{z_\alp}^2 \varphi_{\ulam(t)}, \pd_{z_\alp}^2 \log h_0+\pd_{z_\alp}^2 \varphi_{\ulam(t)}}
\een
where $h_0$ is given in (\ref{eq:compLaplacians31}). $\varphi_{\ulam}$ is a linear combination with coefficients linear in $\ulam$, there is $C_{14}>0$ such that 
\be \label{eq:delhestproof4}
\max_\alp \sup_{U_\alp}\vb{\pd_{z_\alp} \rd{H^{-1}\pd_{z_\alp}H}}\le C_{14}
\ee
The global bound on the last quantity in (\ref{eq:delhest3quantities}) now follows from (\ref{eq:delhestproof1}), (\ref{eq:delhestproof2}), (\ref{eq:delhestproof3}), and (\ref{eq:delhestproof4}).
\end{proof}

\begin{prop} \label{prop:Rtupperbound}
There is $C>0$, such that for $0<r<1$, $u_0,u_1\subset B(0,r)\in L_2^2\rd{\tx{Herm}(F,h_t^{\tx{app}})}$ and $t\gg 1$
\ben
\dbv{R_t(u_0)-R_t(u_1)}_{L^2,h_{t_0}^{\tx{app}}}\le C r t^{17/3} \dbv{u_0-u_1}_{L_2^2,h_{t_0}^{\tx{app}}}\,.
\een
\end{prop}

\begin{proof}
We have
\ben
R_t^{(1)}(u)=2t^2\cl{B_t,f_2\rd{\hat u/2}}+2t^2f_1\rd{\hat u/2} B_t f_1\rd{\hat u/2}\,,
\een
where $R_t^{(1)}$ is defined in (\ref{eq:defRt0StRt1Rt2}) and we used $e^{\hat u/2}=1+f_1(\hat u/2)=1+\hat u/2+f_2(\hat u/2)$ where $f_j$ is defined in (\ref{eq:defnfj}). By the Sobolev embedding theorems there is $C_j>0$ such that $\sup_X \vb{f_j(u_0)}, \vb{f_j(u_1)}\le C_j r^j$ for $j\ge 1$. As a consequence of the proof of Lemma \ref{lem:psibetagammaest}, $\sup_X \vb{B_t}$, $\sup_X \vb{C_t}\lesssim t^{11/3}$. Therefore by Lemma \ref{lem:bddifffj}, there is $C_0>0$ such that for $t\gg 1$ 
\be \label{eq:Rt1est}
\dbv{R_t^{(1)}(u_0)-R_t^{(1)}(u_1)}_{L^2}\le C_0 t^{17/3} r \dbv{u_0-u_1}_{L_2^2}
\ee
Similarly, there is $C_0'>0$ such that for $t\gg 1$
\be \label{eq:Rt2est}
\dbv{R_t^{(2)}(u_0)-R_t^{(2)}(u_1)}_{L^2}\le C_0't^{17/3} r \dbv{u_0-u_1}_{L_2^2}\,,
\ee
where $R_t^{(2)}$ is given in (\ref{eq:defRt0StRt1Rt2}). The conclusion follows once we obtain similar bounds for $S_t(u_0)-S_t(u_1)$. After canceling the linear terms we have
\begin{align*}
& S_t(u)=2F_{\nabla_h}f_2\rd{-\frac{u}{2}}+2f_2\rd{-\frac{u}{2}}F_{\nabla_h}+2f_1\rd{\frac{u}{2}}F_{\nabla_h}f_1\rd{-\frac{u}{2}}+2\bar\pd_h f_2(u)\\
&+2\bar\pd\rd{f_1(-u)\pd_h f_1(u)}+[u,\bar\pd\pd_h f_1(u)]+[u,\bar\pd\rd{f_1(-u)\pd_h f_1(u)}]+2\rd{\bar\pd\pd_h f_1(u)}f_2\rd{-\frac{u}{2}}\\
&+2\bar\pd\rd{f_1(-u)\pd_h f_1(u)}f_2\rd{-\frac{u}{2}}+2f_2\rd{\frac{u}{2}}\rd{\bar\pd\pd_h f_1(u)}+2f_2\rd{\frac{u}{2}}\bar\pd\rd{f_1(-u)\pd_h f_1(u)}\\
&+2f_1\rd{\frac{u}{2}}\rd{\bar\pd\pd_h f_1(u)}f_1\rd{-\frac{u}{2}}+2f_1\rd{\frac{u}{2}}\bar\pd\rd{f_1(-u)\pd_hf_1(u)}f_1\rd{-\frac{u}{2}}
\end{align*}
where $h=h_t^{\tx{app}}$. By Lemma \ref{lem:htappcurvbound}, $\vb{F_{\nabla_h}}\lesssim t^2$. The relevant terms can be expanded further using $\bar\pd\rd{A\pd_h B}=\rd{\bar\pd A}\wedge \rd{\pd_h B}+A\bar\pd\pd_h B$ where $A,B\in \Omega^0\rd{\tx{End}\rd{F}}$. For the terms containing products of $\bar\pd\rd{\cdot}$ or $\pd_h\rd{\cdot}$, we use the bounded inclusions $L_1^2\times L_1^2\subset L^4\times L^4\subset L^2$, for all other terms use bounded inclusion $L_2^2\subset C^0$ and apply Lemmas \ref{lem:delhest}, \ref{lem:bddifffj}. We have for instance for the following term of the first kind,
\begin{align*}
& \dbv{\bar\pd\rd{f_1\rd{-u_0}-f_1\rd{-u_1}}\wedge \pd_h f_1(u_1)}_{L^2} \le \dbv{\bar\pd\rd{f_1\rd{-u_0}-f_1\rd{-u_1}}}_{L^4}\dbv{\pd_h f_1(u_1)}_{L^4} \\
& \lesssim \dbv{\bar\pd\rd{f_1\rd{-u_0}-f_1\rd{-u_1}}}_{L_1^2}\dbv{\pd_h f_1(u_1)}_{L_1^2}\lesssim t^2\dbv{f_1(-u_0)-f_1(-u_1)}_{L_1^2}\dbv{f_1(u_1)}_{L_1^2} \\
&\lesssim t^2r\dbv{u_0-u_1}_{L_2^2}
\end{align*}
It follows from these term-by-term estimates that there is $C_1>0$ such that for $t\gg 1$,
\be \label{eq:Stest}
\dbv{S_t(u_0)-S_t(u_1)}_{L^2}\le C_1t^2r\dbv{u_0-u_1}_{L_2^2}
\ee
The conclusion follows from (\ref{eq:Rt1est}), (\ref{eq:Rt2est}), and (\ref{eq:Stest}) in view of the decomposition (\ref{eq:decompRt}).
\end{proof}

\subsection{Proof of Theorem \ref{thm:main}}
\label{sec:proofmainthm}

On $L_2^2\rd{End(F)}$ define
\be \label{eq:defmcft}
\mcf_t: u\lmapsto -L_t^{-1}\rd{2i\Lambda e^{u/2}\mch_{t,h_{t,\ulam(t)}^{\tx{app}}}(0)e^{-u/2}+R_t(u)}
\ee
For $t\gg 1$, with the help of Props \ref{prop:Rtupperbound}, \ref{prop:L22toL2lowerbound}, and \ref{prop:htapp}, we use the contraction mapping principle to find a fixed point.

\begin{thm} \label{thm:contraction}
There are $C_1$, $C_2>0$ such that for $t\gg 1$,
\ben
\dbv{g_t-\tx{Id}}_{L_2^2}\le C_1e^{-C_2t^{2/3}}
\een
where $h_t=h_t^{\tx{app}}\cdot g_t$ is the unique solution of the SU(1,2) Hitchin equation associated to $(F,t\bet,t\gam)$.
\end{thm}

\begin{proof}
By Props \ref{prop:Rtupperbound}, \ref{prop:L22toL2lowerbound}, there are $C_3, C_4>0$ such that for $0<r<1$ and $u_0,u_1\in B(0,r)$, $t\gg 1$
\begin{align*}
& \dbv{\mcf_t(u_0)-\mcf_t(u_1)}_{L_2^2,h_{t_0}^{\tx{app}}}\le C_3 t^{19}\dbv{R_t(u_0)-R_t(u_1)}_{L^2,h_{t_0}^{\tx{app}}} \\
& \le C_4 t^{19+17/3} r \dbv{u_0-u_1}_{L_2^2,h_{1,\ulamz}^{\tx{app}}}\,.
\end{align*}
By Prop \ref{prop:htapp}, there are $C_5$, $C_6$, $C_7>0$ such that for $t\gg 1$,
\begin{align*}
& \dbv{\mcf_t\rd{0}}_{L_2^2,h_{t_0}^{\tx{app}}}=\dbv{L_t^{-1}\rd{2i\Lambda \mch_{t,h_t^{\tx{app}}}(0)}}_{L_2^2,h_{t_0}^{\tx{app}}} \\
&\le C_5 t^{19}\dbv{\mch_{t,h_t^{\tx{app}}}(0)}_{L^2,h_{t_0}^{\tx{app}}}\le C_6 t^{19}e^{-C_7 t^{2/3}}
\end{align*}

Take $t_1\gg 1$ that the above estimates hold and that 
\ben
2C_6t^{19}e^{-C_7 t^{2/3}}<(2C_4)^{-1}t^{-19-17/3}<1
\een
for $t\ge t_1$ and we assume below $t\ge t_1$. Take $r$ with $2C_6t^{19}e^{-C_7 t^{2/3}}<r<C_4^{-1}t^{-19-17/3}/2$. For $u_0,u_1\in B(0,r)\subset L_2^2\rd{\tx{Herm}\rd{F,h_t^{\tx{app}}}}$, we have $\dbv{\mcf_t(u_0)-\mcf_t(u_1)}\le (1/2)\dbv{u_0-u_1}$. On the other hand, $B(0,r')\subset B(0,r)$ where $r'=(1-C_4t^{19+17/3})^{-1}\dbv{\mcf_t(0)}$. For any $n\ge 0$,
\ben
\dbv{\mcf_t^n(0)}=\dbv{\mcf_t^n(0)-0}\le \sum_{k=0}^{n-1}\dbv{\mcf_t^{k+1}(0)-\mcf_t^k(0)}\le \rd{1-2^{-n}}r'\le r'
\een
The sequence $\cl{\mcf_t^n(0)}_n$ converges to $u_{\infty,t}\in B(0,r')$ satisfying $\mcf_t(u_{\infty,t})=u_{\infty,t}$. By the definition of $\mcf_t$ in (\ref{eq:defmcft}) as well as (\ref{eq:defnRtu}) and (\ref{eq:defnmchth}), we see that $h_t=h_t^{\tx{app}}\cdot g_t$ where $g_t=e^{u_{\infty,t}}$ is a solution to the SU(1,2) Hitchin equation associated to $(F,t\bet,t\gam)$. By Lemma \ref{lem:bddifffj}, there are $C_1$, $C_2>0$ such that for $t\ge t_1$,
\ben
\dbv{g_t-\tx{Id}}_{L_2^2,h_{t_0}^{\tx{app}}}\le r'\le C_6t^{19}e^{-C_7 t^{2/3}} \le C_1e^{-C_2t^{2/3}}
\een
\end{proof}

\begin{cor} \label{cor:htclosetoext}
On any compact set $X_0\Subset X-D$ for $k\ge 0$, there are $t_{1,k}$, $C_k$, $c_k>0$ such that for all $t\ge t_{1,k}$,
\ben
\dbv{g_t-\tx{Id}}_{C^k(X_0)}\le C_ke^{-c_kt^{2/3}}
\een
where $h_t=h_{\infty,t}\cdot g_t$, $h_{\infty,t}=\iota^\ast \rd{h_{L,\ulam(t),t}^{-2}h_K\oplus h_{L,\ulam(t),t}h_K}$ and $h_t$ is the unique solution of the SU(1,2) Hitchin equation associated to $(F,t\bet,t\gam)$.
\end{cor}

\begin{proof}
By Prop \ref{prop:htappclosetoext} and the above theorem, we have $C_0$, $c_0>0$, and $t_{1,0}$ such that for $t\ge t_{1,0}$,
\be \label{eq:htclosetoextineq1}
\dbv{g_t-\tx{Id}}_{L_2^2(X_0)}\le C_0e^{-c_0t^{2/3}}\,.
\ee
This proves the statement for $k=0$ and the rest of the proof will be a standard bootstrap argument. Let $\cl{(U_{\alp,k};z_\alp)}_{\alp\in \msa,k\ge 0}$ be a sequence of atlas with $U_{\alp,\ell}\subseteq U_{\alp,k}$ for $\ell\ge k$ and such that $F$ is trivialized on $U_{\alp,0}$. Choose local frames over each $U_{\alp,0}$ and let $H_t^{(\alp)}$ (resp. $H_{t,\infty}^{(\alp)}$, $\varphi^{(\alp)}dz_\alp$) be local forms of $h_t$ (resp. $h_{t,\infty}$, $\Phi$) over $U_{0,\alp}$ with respect to corresponding frames. We have that $H=H_t^{(\alp)}$ or $H_{t,\infty}^{(\alp)}$ and $\varphi=\varphi^{(\alp)}$ satisfies
\ben
\pd_{\bar z}\pd_z H=\rd{\pd_{\bar z}H}H^{-1}\rd{\pd_z H}+t^2 H\sq{\varphi,H^{-1}\varphi^\ast H}
\een
where $z=z_\alp$. Suppose for $0\le \ell\le k+1$ there are $C_{\ell}'$, $c_\ell'>0$, and $t_{1,\ell+1}\ge t_{1,0}$ such that for all $\alp\in\msa$, $t\ge t_{1,\ell+1}$, $H_t=H_t^{(\alp)}$ and $H_{t,\infty}=H_{t,\infty}^{(\alp)}$ satisfies
\be \label{eq:Lell2inductionstep}
\dbv{H_t-H_{t,\infty}}_{L_\ell^2\rd{U_{\alp,\ell}}}\le C_\ell' e^{-c_\ell' t^{2/3}}\,.
\ee
Note the case $k=0$ follows from (\ref{eq:htclosetoextineq1}). By the local form of the Hitchin equation, we have with $z=z_\alp$,
\begin{align}
& \pd_{\bar z}\pd_z\rd{H_t-H_{t,\infty}}=\pd_{\bar z}\rd{H_t-H_{t,\infty}}H_t^{-1}\rd{\pd_z H_t} + \rd{\pd_{\bar z} H_{t,\infty}}\rd{H_t^{-1}-H_{t,\infty}^{-1}}\rd{\pd_z H_t} \nonumber \\
&+\rd{\pd_{\bar z}H_{t,\infty}} H_{t,\infty}^{-1}\pd_z \rd{H_t-H_{t,\infty}}\nonumber \\
&+t^2\rd{H_t-H_{t,\infty}}\sq{\varphi,H_t^{-1}\varphi^\ast H_t}+t^2H_{t,\infty}\sq{\varphi,\rd{H_t^{-1}-H_{t,\infty}^{-1}}\varphi^\ast H_t}\nonumber \\
& +t^2H_{t,\infty}\sq{\varphi,H_{t,\infty}^{-1}\varphi^\ast\rd{H_t-H_{t,\infty}}}\,. \label{eq:htclosetoext1}
\end{align}
By the interior elliptic regularity estimate, there is $C>0$ such that for all $\alp$,
\ben
\dbv{H_t-H_{t,\infty}}_{L_{k+2}^2\rd{U_{\alp,k+2}}}\le C\rd{
\dbv{\pd_{\bar z}\pd_z \rd{H_t-H_{t,\infty}}}_{L_k^2\rd{U_{\alp,k+1}}}+\dbv{H_t-H_{t,\infty}}_{L^2\rd{U_{\alp,k+2}}}}\,.
\een
In (\ref{eq:htclosetoext1}), we expand higher derivatives of each term to get a sum of terms of the form
\ben
\nabla^{m}\rd{\pd_{\bar z}\rd{H_t-H_{t,\infty}}H_t^{-1}\rd{\pd_z H_t}},\,\ldots\,, \nabla^{m}\rd{t^2H_{t,\infty}\sq{\varphi,H_{t,\infty}^{-1}\varphi^\ast\rd{H_t-H_{t,\infty}}}}\,,
\een
with $0\le m\le k$. Following inductive arguments as in the proof of Prop \ref{prop:locmodsym} and note that $t^2e^{-ct^{2/3}}\to 0$ as $t\to 0$ for $c>0$, we have (\ref{eq:Lell2inductionstep}) for all $\ell\ge 0$. The conclusion follows from the bounded inclusion $L_{\ell+2}^2\subset C^{\ell'}$ for each $\ell'\ge 0$.

\end{proof}

Theorem \ref{thm:main} will now follow by combining all the above analysis. We restate it in a more detailed form.

\begin{thm} \label{thm:main2}
Fix $x_0\in X-D$, $v_0\in \atp{L}{x_0}$ and $X_0\subset X-D$ a compact set. Let $h_t$ be Hermitian-Yang-Mills metric of $(F,t\beta,t\gamma)$ and $\iota: F\to V$ the Hecke modification corresponding to $(F,\beta,\gamma)$ as in Theorem \ref{thm:hecke}. For $a\neq 0$ let $S_a=\iota^{-1}\circ \tx{diag}(a^2,a^{-1})\circ\iota$, an endomorphism of $\atp{F}{X-D}$. Let $\widetilde h_t=S_{\vb{v_0}_{k_t}}^\ast h_t$ where $k_t$ is metric on $\atp{L}{X-D}=\atp{\det F^\ast}{X-D}$ induced from $h_t$. Let $X_0\Subset X-D$ be a compact set. Let $h_\infty=\iota^\ast \rd{h_{L,\infty}^{-2}h_K\oplus h_{L,\infty}h_K}$ where $h_{L,\infty}$ is a harmonic metric adapted to the filtered bundle $(L,\ulami)$ over $(X,D)$ such that $\vb{v_0}_{h_{L,\infty}}=1$ and 
\be \label{eq:barycenter}
\lam_{\infty,j}=\begin{cases}
1/4 & p_j\in D_\bet \\
-1/4 & p_j\in D_\gam \\
-d_r^{-1}\rd{\deg L+\ov{4}\rd{d_\bet-d_\gam}} & p_j\in D_r
\end{cases}\,.
\ee
Let $\psi_t$ be given by $\widetilde h_t=h_\infty\cdot \psi_t$. Then for each $k\ge 0$, there is $C_k$ such
\ben
\dbv{\psi_t-\tx{Id}}_{C^k\rd{X_0}}\le \frac{C_k}{\log t}\,.
\een
\end{thm}

\begin{proof}
Let $h_{\infty,t}=\iota^\ast \rd{h_{L,\ulam(t),t}^{-2}h_K\oplus h_{L,\ulam(t),t}h_K}$ and $g_t$ be as in Cor \ref{cor:htclosetoext}. We will first define several auxiliary metrics on $\atp{F}{X_0}$ and automorphisms relating them. 

Let $k_t$ (resp. $k_{\infty,t}$) be the metric on $\atp{L}{X-D}$ induced from $h_t$ (resp. $h_{\infty,t}$). Let $a_t=\vb{v_0}_{k_t}$ and 
\be \label{eq:mainthmat}
a_{\infty,t}=\vb{v_0}_{k_\infty,t}=\vb{v_0}_{h_{L,\ulam(t),t}}
\ee

Let $\widetilde h_t=S_{a_t}^\ast h_t$, $\widetilde h_{\infty,t}=S_{a_t}^\ast h_{\infty,t}$, and $h_{\infty,t}'=S_{a_{\infty,t}}^\ast h_{\infty,t}$. Let $g_t'$ (resp. $\widetilde g_t$) be defined by $\widetilde h_t=\widetilde h_{\infty,t}\cdot g_t'$ (resp. $\widetilde h_{\infty,t}=h_{\infty,t}'\cdot\widetilde g_t$). Let $g_t''$ be defined by $h_{\infty,t}'=h_\infty\cdot g_t''$. The metrics and automorphism relating them are summarized in the following diagram. In particular, $\psi_t=g_t''\cdot \widetilde g_t\cdot g_t'$.

\begin{center}
\begin{tikzcd}
\widetilde h_t  & \widetilde h_{\infty,t} \arrow[l,"g_t'"] & h_{\infty,t}' \arrow[l,"\widetilde g_t"] & h_\infty \arrow[l,"g_t''"] \\
h_t \arrow[u,"S_{a_t}^\ast"] & h_{\infty,t} \arrow[u,"S_{a_t}^\ast"] \arrow[ur, bend right,"S_{a_{\infty,t}}^\ast"] \arrow[l,"g_t"] & &
\end{tikzcd}
\end{center}

Recall again by Prop \ref{prop:su12stab2} $d_r>0$. Let $j_0$ be a fixed index such that $p_{j_0}\in D_r$. By (\ref{eq:mainthmat}) and (\ref{eq:defhLulamt}), we have that 
\ben
a_{\infty,t}=\vb{v_0}_{h_{L,\tx{HE}}}e^{(\varphi_{\ulam(t)}+\eta_{\ulam(t),t})/2}\,.
\een
Without loss of generality we may assume $\vb{v_0}_{h_{L,\tx{HE}}}=1$. By Props \ref{prop:clamcont}, \ref{prop:tcompfamily}, and (\ref{eq:culamt}) with $j=j_0$, there is $C_0>0$ such that for $t\gg 1$
\ben
\vb{\eta_{\ulam(t),t}-\frac{4}{3}\lam_{j_0}\log t}\le C_0\,.
\een
We have from Def \ref{def:someharmfcns} that $\varphi_{\ulam}$ is a linear combination of fixed smooth functions $G_j$ on $X_0$ with coefficients linear in $\ulam(t)$. Note also that $\vb{\lam_{j_0}(t)}<1/4$ for all $t\ge 1$. Therefore there is $C_1>1$ such that for $t\gg 1$, $C_1^{-1}t^{-1/3}\le e^{\varphi_{\ulam(t)}+\eta_{\ulam(t),t}}\le C_1 t^{1/3}$. Therefore there is $C_2>0$ such that for $t\gg 1$, $a_{\infty,t}\le C_2 t^{1/6}$. We have $a_t=a_{\infty,t}\rd{\det g_t(p_0)}^{-1/2}$. By Cor \ref{cor:htclosetoext} there is $C_2'>0$ such that for $t\gg 1$, $a_t\le C_2'a_{\infty,t}\le C_2'C_2t^{1/6}$. Note that $g_t'=S_{a_t}^{-1}\circ g_t\circ S_{a_t}$. It follows from Cor \ref{cor:htclosetoext} that for $k\ge 0$, there are $C_{3,k}$, $c_{3,k}>0$ such that for $t\gg 1$
\be \label{eq:mainthm1}
\dbv{g_t'-\tx{Id}}_{C^k\rd{X_0}}\le C_{3,k}e^{-c_{3,k}t^{2/3}}\,.
\ee

We have
\ben
\widetilde g_t=\iota^{-1}\circ\tx{diag}\rd{\det g_t\rd{p_0}^{-1},\rd{\det g_t\rd{p_0}}^{1/2}}\circ\iota\,.
\een
Note that for any two choices of frame in (\ref{eq:Cknorm}) over $X_0$, the resulting norms are equivalent. Under a frame compatible with decomposition $V=L^{-2}K\oplus LK$, the local forms of $\widetilde g_t$ are constant. Therefore by Cor \ref{cor:htclosetoext}, there are $C_4$, $c_4>0$ such that for $t\gg 1$,
\be \label{eq:mainthm2}
\dbv{\widetilde g_t-\tx{Id}}_{C^k(X_0)}=\dbv{\widetilde g_t-\tx{Id}}_{C^0(X_0)}\le C_4 e^{-c_4 t^{2/3}}\,.
\ee

Let $h_{L,\infty,t}$ be the metric on $L$ with $h_{\infty,t}'=\iota^\ast \rd{h_{L,\infty,t}^{-2}h_K\oplus h_{L,\infty,t}h_K}$. We have that 
\begin{align*}
& h_{L,\infty}=h_{L,\infty,t}\exp\rd{ \varphi_{\ulami}-\varphi_{\ulam(t)}+\varphi_{\ulami}\rd{p_0}-\varphi_{\ulam(t)}\rd{p_0}
}  \\
&=h_{L,\infty,t}\exp\rd{
\sum_{j=1}^{4g-4}\sum_{\ell=1}^j\rd{\lam_{\infty,\ell}-\lam_\ell(t)}\rd{G_j-G_j\rd{p_0}}
}\,.
\end{align*}
For $k\ge 0$, there is $C_{5,k}>0$ such that $\dbv{G_j-G_j\rd{p_0}}_{C^k(X_0)}\le C_{5,k}$. By Prop \ref{prop:tcompfamily} and the above, for $k\ge 0$, there is $C_{6,k}>0$ such that for $t$ large enough
\be \label{eq:mainthm3}
\dbv{g_t''-\tx{Id}}_{C^k\rd{X_0}}\le \frac{C_{6,k}}{\log t}\,.
\ee
The conclusion follows from (\ref{eq:mainthm1}), (\ref{eq:mainthm2}) and (\ref{eq:mainthm3}).

\end{proof}

Lastly, we comment on a special case of Thm \ref{thm:main} where the convergence rate is improved. For the SU(1,2) Higgs bundle $(F=K^{-1}\oplus K,t\bet,t\gam)$ in (\ref{eq:specialSU12}), we have $\ulb=\rd{1,\ldots,1}$. By the discussion in Prop \ref{prop:tcompfamily}, $\ulam=\rd{0,\ldots,0}$ is $t$-compatible with $\ulb$ for all $t\ge 1$, and $\eta_{\ulam(t),t}$ (defined in (\ref{eq:defhLulamt})) is independent of $t$. It follows from the proof that $a_t$ is uniformly bounded in $t$ and that the convergence in $C^k(X_0)$ may be improved to be exponential in $t$.

In fact, this also follows directly from the SL$(2,\mbc)$-Hitchin equation case. The associated Hitchin equation reduces to a scalar PDE, and the solution is given by a metric $h_t^{-1}\oplus h_t$ on $F$ with $h_t^{1/2}$ the unique solution to the Hitchin's equation associated to SL$(2,\mbc)$ Higgs bundle
\ben
\rd{E=K_X^{1/2}\oplus K_X^{-1/2}, \Phi=t \pmt{ & 1 \\ q & }}\,.
\een
In this manner, the exponential in $t$ convergence $h_t\to h_K$ ($h_K$ defined in Lemma \ref{lem:decoupledsoln}) is a direct consequence of the result in \cite{MSWW16}.

\end{document}